\documentclass{amsart}
\usepackage{ainfty}

\title[Koszul homomorphisms]{Koszul homomorphisms and universal resolutions in local algebra} 

\date{\today}



\author[B.~Briggs]{Benjamin Briggs}
\address{Benjamin Briggs,
Department of Mathematics, 
Huxley Building, South Kensington Campus, 
Imperial College, 
London SW7 2AZ, UK}
\email{b.briggs@imperial.ac.uk}
\author[J.C.~Cameron]{James C.~Cameron}
\address{James C.~Cameron}
\email{cameron@math.utah.edu}
\author[J.C.~Letz]{Janina~C.~Letz}
\address{Janina~C.~Letz,
Faculty of Mathematics,
Bielefeld University,
PO Box 100 131,
33501 Bielefeld,
Germany}
\email{jletz@math.uni-bielefeld.de}
\author[J.~Pollitz]{Josh Pollitz}
\address{Josh Pollitz,
Mathematics Department, 
Syracuse University, 
Syracuse, NY 13244 U.S.A.}
\email{jhpollit@syr.edu}

\begin{document}

\begin{abstract}
We define a local homomorphism $(Q,k)\to (R,\ell)$ to be Koszul if its derived fiber $R\lotimes_Q k$ is formal, and if $\Tor QRk$ is Koszul in the classical sense. This recovers the classical definition when $Q$ is a field, and more generally includes all flat deformations of Koszul algebras. The non-flat case is significantly more interesting, and there is no need for examples to be quadratic: all complete intersection and all Golod quotients are Koszul homomorphisms. We show that the class of Koszul homomorphisms enjoys excellent homological properties, and we give many more examples, especially various monomial and Gorenstein examples. We then study Koszul homomorphisms from the perspective of $\ai$-structures on resolutions. We use this machinery to construct universal free resolutions of $R$-modules by generalizing a classical construction of Priddy. The resulting (infinite) free resolution of an $R$-module $M$ is often minimal, and can be described by a finite amount of data whenever $M$ and $R$ have finite projective dimension over $Q$. Our construction simultaneously recovers the resolutions of Shamash and Eisenbud over a complete intersection ring, and the bar resolutions of Iyengar and Burke over a Golod ring, and produces analogous resolutions for various other classes of local rings.
\end{abstract}

\maketitle
 
\vspace{-2.1em}

\tableofcontents

\section{Introduction}

The phenomenon of Koszul duality has been observed in many forms across algebra, geometry and topology. It provides explicit computational tools for answering homological questions and opens up deep connections between a number of seemingly unrelated areas of mathematics. The goal of the present work is to develop a relative theory of Koszul duality in local commutative algebra, and to give concrete applications for  understanding infinite free resolutions.

For a finite homomorphism $\phi\colon Q\to R$ of commutative noetherian local rings, the derived fiber $F=R\lotimes_Qk$, where $k$ is the residue field of $Q$, is a differential graded $k$-algebra that  encodes important ring theoretic properties of $\phi$. We define $\phi$ to be \emph{Koszul} if $F$ is formal (see \ref{c_formal}), and if its homology $\H(F)=\Tor QRk$ is a Koszul $k$-algebra (see \ref{koszul-algebra}). This recovers the classical definition when the source is a field. Through the looking glass that connects local algebra with rational homotopy theory, the definition is directly analogous to Berglund's notion of a Koszul space. 

Flat local maps that have a Koszul fiber are natural examples of Koszul homomorphisms, but the non-flat case is significantly more interesting: all complete intersection and all Golod quotient homomorphisms are Koszul, and we give many other monomial and Gorenstein examples. In particular, there is no need for the homomorphism to be quadratic in any sense.

The definition also has structural consequences connecting the homological algebra over $R$ and $Q$. Our main theorem provides an algorithmic way to transfer free resolutions over $Q$ into free resolutions over $R$. To achieve this we introduce a slightly stronger ``strictly Koszul" property (see \ref{sec:special-koszul}) that is satisfied in our main examples. These ideas borrow from a long history, and we will discuss the context and technology behind the construction following this summary of our main results. 

\begin{introthm} \label{intro:transfer}
For any strictly Koszul local homomorphism $\phi \colon Q \to R$  there is a non-negatively graded, degreewise finite rank free $Q$-module $C$ such that:
\begin{enumerate}[(1)]
\item \label{introitem1}
For each finitely generated $R$-module $M$ with a minimal $Q$-free resolution $G\to M$, there is a differential $\partial^{\tau}$ on the graded $R$-module $R \otimes C \otimes G$ such that the resulting ``twisted tensor product'' complex
\begin{equation*}
(R \otimes C\otimes G,\partial^\tau) \xrightarrow{\ \simeq\ } M
\end{equation*}
is an $R$-free resolution of $M$. If $R$ and $M$ have finite projective dimension over $Q$, then both $C$ and the twisted tensor product differential can be explicitly described in their entirety with a finite amount of data.
\item \label{introitem2} Assume that $\phi$ is small (a central case of interest is $(Q,\fm_Q)$ regular and $\ker(\phi)\subseteq \fm^2_Q$). The twisted tensor product complex is minimal for the residue field $k$ of $R$. More generally, the resolution is minimal whenever $M$ is inert with respect to $\phi$, in the sense of Lescot. In particular, 
\[
\textstyle{\sum_i }\rank_Q(C_i)t^i=\frac{ \ps^R_k(t)}{P^Q_k(t)}\,.
\]
\end{enumerate} 
The following homomorphisms are strictly Koszul:
\begin{enumerate}[(a)]
    \item\label{introex1} Surjective complete intersection homomorphisms.
    \item\label{introex2} Surjective Golod homomorphisms.
    \item\label{introex3} Surjective Gorenstein homomorphisms of projective dimension three or less.
    \item\label{introex4} Cohen presentations of compressed artinian Gorenstein rings having characteristic zero, odd embedding dimension, and socle degree not $3$.
\end{enumerate} 
\end{introthm}

Part \cref{introitem1} of the \namecref{intro:transfer}, with an explicit description of the twisted tensor product differential, is \cref{priddy-resolution}, while part \cref{introitem2} is contained in \cref{t_inert}. The examples \cref{introex1,introex2,introex3,introex4}, and several more, are introduced in \cref{sec:examples-koszul} and treated again in \cref{sec:examples-special-koszul},  with a complete description of the corresponding $C$ in each case. 

Universal resolutions, that is, free resolutions over a ring that are defined in a uniform way for all finitely generated modules, have been of central interest in homological commutative algebra since at least the 60s, often importing tools such as Massey operations and bar resolutions from algebraic topology.

Let $\phi\colon Q \to R$ be a local homomorphism. Shamash constructed universal resolutions for $R$-modules when $\phi$ is a hypersurface quotient \cite{Shamash:1969}, and these were clarified and extended to complete intersection quotients by Eisenbud using the theory of higher homotopies \cite{Eisenbud:1980}. Burke recognized in \cite{Burke:2015} that the higher homotopies are a manifestation of certain $\ai$-structures (we will return to these later in the introduction).  In the presence of a $Q$-free differential graded algebra resolution $A\to R$, and a $Q$-free differential graded $A$-module resolution $G$ of an $R$-module $M$, Iyengar constructed a bar resolution for $M$ over $R$ \cite{Iyengar:1997}. By endowing $A$ with an $\ai$-algebra structure, and $G$ with an $\ai$-module structure over $A$, Burke constructed a bar resolution even when associative multiplicative resolutions do not exist \cite{Burke:2015}. The resulting resolution is minimal when $M$ is Golod with respect to $\phi$, and is otherwise typically far from minimal. \cref{intro:transfer} recovers both the resolutions of Shamash and Eisenbud, when $\phi$ is a complete intersection quotient, and the bar resolutions of Iyengar and Burke, when $\phi$ is a Golod quotient.

In parallel, the universal resolutions introduced by Priddy over Koszul algebras \cite{Priddy:1970} have had far-reaching impact, not least as a computational tool. Our theory directly builds on and recovers his construction, while providing a common framework for the universal resolutions above. 

The technical foundation for our universal resolutions is in  \cref{sec:twisted-tensor}. Here we develop a general theory of twisted tensor products over a commutative ring $Q$. The data that goes into this construction is a curved differential graded coalgebra $C$ over $Q$, a quasi-isomorphism $\cobc{C}\to R$ from the cobar construction of $C$ to $R$, and a differential graded module structure over $\cobc{C}$ on $G$. These terms are defined in  \cref{sec:ai}. From this, in \cref{resn-trans-qi-coalgebras}, we construct a canonical  resolution
\[
R \totimes C \totimes G =(R \otimes C\otimes G,\partial^\tau) \xrightarrow{\ \simeq\ } M\,.
\]
The key to proving \cref{intro:transfer} is to show that $C$ can be defined in an explicit, canonical, and minimal way when $\varphi$ is strictly Koszul. We will return to this at the end of the introduction, with more context in hand.

\begin{center}
{\bf \textdiv} 
\end{center}

We turn our attention back to the Koszul homomorphisms.
The first half of this work develops the theory of  these maps; this part of the paper does not involve $\ai$-structures, using only ordinary differential graded algebras. 

Similar Koszul-type conditions have been considered by other authors \cite{Croll/etal:2020,Herzog/Iyengar:2005,Herzog/Reiner/Welker:1998,Myers:2021}, and we compare our definition with theirs in \cref{r_other_defs}. We motivate our condition as well by drawing connections with other areas, such as rational homotopy theory (\cref{r_formality,r_highly_connected}) and toric topology (\cref{r_toric_top}). 

We pay particular attention to the case that $Q$ is regular. In this situation, the resolutions constructed in \cref{intro:transfer} essentially depend only on the ring $R$, and they are always finitely determined. When $R$ is a local ring such that every Cohen presentation $\phi\colon Q\to \widehat{R}$ is a Koszul homomorphism, we say that $R$ is \emph{Cohen Koszul}; see \cref{sec_cohen_kos}. These rings enjoy excellent homological properties while being surprisingly abundant; they behave in many ways like classical Koszul algebras despite not necessarily being quadratic.

We show that Cohen Koszul local rings have rational Poincar\'e series that can be computed explicitly from their Koszul homology: 
\[
\ps^R_k(t)=\frac{(1+t)^e}{\sum_{i,w}(-1)^w \rank_k (\H_{i}(K^R)_{(w)})t^{i+w}}\,;
\]
see \cref{p_cohen_koszul_poincare}, where the notation is explained.

 \cref{sec:examples-koszul} is devoted entirely to examples, and we prove that surjective complete intersection homomorphisms are Koszul (\cref{e:koszul:ci}), along with surjective 
Golod homomorphisms (\cref{e:koszul:golod}), and surjective Gorenstein homomorphisms of projective dimension three (\cref{e:koszul:buchsbaum-eisenbud}). We exactly determine the local rings of codepth three or less that are Cohen Koszul in terms of the classification into the types described in \cite{Avramov/Kustin/Miller:1988}; we find that in every type except one the local ring is Cohen Koszul  (\cref{p_codepth_three}). Graded local rings having an almost linear resolution in the sense of \cite{Dao/Eisenbud:2022} are also Cohen Koszul (\cref{r_list_of_examples_almost_golod_gor}). We treat monomial rings in \cref{e:monomial}, making connections with combinatorial commutative algebra and with the topology of moment angle complexes. Monomial rings are classically Koszul exactly when they are quadratic \cite{Froeberg:1975}, while Cohen Koszul monomial rings need not be, and we produce many nontrivial examples in \cref{p_monomial_example}.
We further give examples  that illuminate how the Koszul condition relates with classical Koszulity,  formality, being quadratic, and various other technical conditions.

One of our main examples is a class of rings that we call \emph{almost Golod Gorenstein}, treated in \cref{e:minl-non-golod}. Within the class of Gorenstein local rings, these display extremal homological behavior analogous to Golod rings within the class of all local rings; cf.\@  \cref{p_almost_golod_poincare}. In \cref{t_almost_golod_gor_characterisation} we establish a characterization in terms of the derived fiber that is similar to Avramov's characterization of Golod rings \cite{Avramov:1986}, and, under some additional technical assumptions, we deduce that almost Golod Gorenstein rings are Cohen Koszul.

\begin{center}
$\pmb{\vee\hspace{-2.65mm}\mid}$ 
\end{center}

In the second half of the paper we study Koszul homomorphims from the perspective of $\ai$-structures on resolutions. An $\ai$-algebra is a complex $A$ equipped with multilinear operations $m_n\colon A^{\otimes n}\to A$ for $n\geqslant 2$ that together satisfy certain associativity conditions generalizing the definition of a differential graded algebra (which one recovers by assuming $m_n=0$ for $n\geqslant 3$). These objects were introduced by Stasheff to characterize loop spaces in algebraic topology \cite{Stasheff:1963a}.

Koszulity is well-known to be connected with formality (\cref{r_formality}), and in turn it has been understood since \cite{Kadeishvili:1982} that formality can be made visible through $\ai$-structures. In the present context, these structures are important because they carry the information necessary to construct the universal resolutions in \cref{intro:transfer} while being flexible enough that all resolutions can always be given $\ai$-structures.

An introduction to $\ai$-algebras and $\ai$-modules over commutative rings is given in \cref{sec:ai}. In \cref{sec:transfer-ai} we prove some quite general transfer results, in particular, constructing $\ai$-structures on minimal resolutions of local rings and modules.  
Burke was one of the first to develop and apply the theory of $\ai$-algebras over a commutative ring (rather than over a field) \cite{Burke:2015,Burke:2018}, and our treatment owes a substantial intellectual debt to his work. 

We develop the theory of cyclic $\ai$-algebras over commutative rings in \cref{cyclic_ai_algebra}. These were introduced by Kontsevich as part of his homological mirror symmetry program \cite{Kontsevich:1994}. These $\ai$-algebras possess extra structure that takes advantage of the Poincar\'e duality on the minimal resolution of a Gorenstein ring, and we apply this theory to almost Golod Gorenstein rings.

With this perspective in hand we return to Koszul homomorphisms in \cref{sec:koszul-via-ai}. The next result shows that, at the derived level, Koszul homomorphisms may be thought of as deformations of classical Koszul algebras. A more precise and more general statement is given in  \cref{koszul-ai-presentation}. We write $\tmod{V}$ for the tensor algebra $\bigoplus_{n\geqslant 0}V^{\otimes_Q n}$ on a graded $Q$-module $V$.

\begin{introthm} \label{intro:Koszul-ai}
A surjective local homomorphism $\phi \colon Q \to R$ is Koszul if and only if there is a positively graded, degreewise finite rank free $Q$-module $V$, a direct summand $W\subseteq V \otimes_Q V$, and an $\ai$-structure $\{m_n\}$ on $A=\tmod{V}/(W)$  such that
\begin{enumerate}
\item the induced quotient $A\to R$ is an $\ai$-algebra quasi-isomorphism,
\item modulo the maximal ideal of $Q$, the $\ai$-structure $\{m_n\}$ on $A$ agrees with the usual algebra structure on $\tmod{V}/(W)$, that is, 
\begin{equation*}
m_2 \otimes_Q k=\mu \otimes_Q k \quad\text{and}\quad m_n\otimes_Q k = 0 \ \text{ for }\ n \neq 2\,,
\end{equation*}
where $\mu$ is the usual product on the quotient of a tensor algebra,
\item the $k$-algebra $\tmod{V \otimes_Q k}/(W \otimes_Q k)$ is Koszul with this algebra structure.
\end{enumerate}
\end{introthm}

Coming full circle, we are now able to describe the coalgebra $C$ that appears in \cref{intro:transfer}. 
Fixing a Koszul homomorphism $\phi$ with $V$ and $W$ as in  \cref{intro:Koszul-ai}, we define
\[
C\coloneqq\bigoplus_n\Big({\textstyle\bigcap\limits_{i+2+j=n} V^{\otimes_Q i} \otimes_Q  W \otimes_Q  V^{\otimes_Q j}}\Big)\,.
\]
This is modeled on the work of Priddy \cite{Priddy:1970}. By construction, the graded $Q$-dual $C^\vee$ is the quadratic dual $\tmod{V^\vee}/(W^\perp)$ of the algebra $\tmod{V}/(W)$. The strict Koszul condition introduced in \cref{sec:special-koszul} guarantees that the $\ai$-structure on $A$ induces the structure of a curved differential graded coalgebra on $C$; see \cref{d:special-Koszul}. We think of $C$ as Koszul dual to $R$ \emph{relative to $Q$}, as justified by \cref{priddy-subcoalgebra}. 

We conclude the paper with a reexamination of examples, in  \cref{sec:examples-special-koszul}. We start by showing that certain deformations of classical Koszul algebras yield strictly Koszul homomorphisms, and we obtain resolutions that directly deform the original resolutions of Priddy. We study surjective Golod homomorphisms in \cref{e:special-koszul:golod}; in this case $C=\bc{A}$ is the bar construction of the $\ai$-algebra $A$, and we recover the bar resolution of Iyengar and Burke. In \cref{e:koszul-ai:buchsbaum-eisenbud} we show that surjective Gorenstein homomorphisms of projective dimension three are strictly Koszul, and describe $C$ as the dual of a noncommutative hypersurface. We show that surjective complete intersection homomorphisms are strictly Koszul in \cref{e:special-koszul:ci}, and we show that $C$ is the free divided power algebra on $V$; our twisted tensor products encode the theory of higher homotopies, and the resulting resolutions recover those of Shamash and Eisenbud. We end by treating almost Golod Gorenstein rings in \cref{e:koszul-ai:minl-non-golod}; this requires a substantial amount of machinery, and the result is a large class of interesting Gorenstein rings over which we have explicit, small, universal resolutions.

\begin{ack} 
We owe thanks to Steven Amelotte for many useful conversations, in particular improving our understanding of moment angle manifolds, the combinatorics of simplicial complexes, and almost linear resolutions. 
We thank Alexander Berglund for sharing important insights on Koszul-type phenomena in algebra and topology. For many fruitful discussions on systems of higher homotopies, and for allowing us to share some of these ideas here, we are grateful to Eloísa Grifo. We thank Srikanth Iyengar for his encouragement and many useful discussions on $\ai$-structures. We also thank Trung Chau, Michael DeBellevue, and Keller VandeBogert for helpful conversations regarding concrete examples. We are also very grateful to the referee, whose close reading and many helpful comments improved the paper substantially.

Part of this work was done at the Hausdorff Research Institute for Mathematics, Bonn, when  Briggs, Letz, and Pollitz were at the ``Spectral Methods in Algebra, Geometry, and Topology" trimester, funded by the Deutsche Forschungsgemeinschaft under Germany's Excellence Strategy--EXC-2047/1--390685813. 
Briggs is supported by the European Union under the Grant Agreement no.\ 101064551, and part of this work was completed when he was funded by NSF grant DMS-1928930. Letz was partly supported by the Deutsche Forschungsgemeinschaft (SFB-TRR 358/1 2023 - 491392403) and by the Alexander von Humboldt Foundation in the framework of a Feodor Lynen research fellowship endowed by the German Federal Ministry of Education and Research. Pollitz was supported by NSF grants DMS-1840190, DMS-2002173, and DMS-2302567. 
\end{ack} 

\section{Koszul homomorphisms} \label{sec:Koszul}
In this section we discuss the Koszul property in various settings, starting from the classical notion for an algebra over a field and leading up to a definition of a Koszul local homomorphism. Examples have been collected in the \cref{sec:examples-koszul}. While later sections exploit the machinery of $\ai$-algebras, this section requires only knowledge of differential graded (dg) algebras; a suitable reference for the latter is \cite{Avramov:1998}. 

We fix a field $k$, and work with local rings having residue field $k$. We also consider modules with two $\BZ$-gradings, the \emph{weight grading} and the \emph{homological grading}. The weight grading is denoted $M = M_{(\star)}$ and the corresponding shift $M(d)$ is given by $M(d)_{(w)} = M_{(w+d)}$. For the homological grading we write $M = M_\bullet$ and the suspension $\susp^d M$ is given by $(\susp^d M)_i = M_{i-d}$. The homological degree of an element $m\in M$ is denoted $|m|$. We assume that the two gradings are \emph{compatible} in the sense that $M$ is bigraded by its submodules $M_{i,(w)} \coloneqq M_i \cap M_{(w)}$. If $M$ is a complex, the differential $\partial^M$ should preserve the weight grading and decrease the homological grading by one, and we equip $\susp M$ with the differential $
\partial^{\susp M} \coloneqq -\partial^M$. 

\begin{definition}\label{koszul-algebra}
An augmented $k$-algebra $K$ is \emph{Koszul (over $k$)} if it admits an algebra grading $K=\bigoplus_{w\geqslant 0} K_{(w)}$, known as a \emph{weight grading}, such that $K_{(0)}=k$ and such that the minimal resolution of $k$ is linear with respect to this grading. 

\end{definition}
 
\begin{remark}
\cref{koszul-algebra} is essentially the classical definition of a Koszul algebra due to Priddy \cite[Section~2]{Priddy:1970}, see also \cite{Lofwall:1986,Froeberg:1999}, except that we do not consider the weight grading to be part of the given data. We emphasize that the additional grading may be \emph{different} to the given one; cf.\@ \cref{r_different_grading}. By \cite[Secton~2.3]{Beilinson/Ginzburg/Soergel:1996} a Koszul algebra is quadratic with respect to this new grading; that is, generated as an associative $k$-algebra by elements of weight one, subject to relations of weight two. For example if $U$ is a graded vector space then the trivial extension algebra $K=k \ltimes U$ is a graded augmented algebra, and the weight grading $K_{(0)}=k$ and $K_{(1)}=U$ makes $K$ Koszul; see \cref{c:trivial_extension}. 
\end{remark}

Next we define a Koszul property for dg algebras. This definition appears implicitly in the work of Berglund \cite{Berglund:2014}, where it is applied to Sullivan models for topological spaces; it is \emph{not} the same as the Koszul property introduced in \cite{He/Wu:2008}, which is too general for our purposes.

\begin{chunk}\label{c_formal}
Let $A$ and $B$ be dg $k$-algebras. Recall that $A$ is \emph{quasi-isomorphic} to $B$, denoted $A \simeq B$, if there exists a zig-zag of quasi-isomorphisms of dg $k$-algebras connecting $A$ and $B$. A dg $k$-algebra $K$ is called \emph{formal} if it is quasi-isomorphic to $\H(K)$.
\end{chunk}

\begin{definition}\label{koszul-dg-algebra}
An augmented dg $k$-algebra $K$ is \emph{Koszul} if it is formal and $\H(K)$ is Koszul in the sense of \cref{koszul-algebra}.
\end{definition}

\begin{remark}\label{r_formality} 
It is well-known that formality is closely related with the Koszul property. In fact, an augmented $k$-algebra $K$ is Koszul in the sense of \cref{koszul-algebra} if and only if the dg $k$-algebra $\RHom{K}{k}{k}$ is formal; see \cite{Gugenheim/May:1974} and also \cite[2.2]{Keller:2002} and \cite[Theorem~2.9]{Berglund:2014}. This condition is sometimes called \emph{coformality} of $K$. From this perspective, a dg $k$-algebra $K$ is Koszul in the sense of \cref{koszul-dg-algebra} if and only if it is both formal and coformal. 
\end{remark}

Before introducing the Koszul property for local homomorphisms, we need to recall the notion of the derived fiber. Let $\phi \colon Q \to R$ be a local homomorphism of commutative noetherian local rings having maximal ideals $\fm_Q$ and $\fm_R$, respectively, and common residue field $k$. 

Let $A \to R$ be a dg algebra resolution of $R$ over $Q$; that is, $A$ is a dg algebra concentrated in non-negative degrees, such that $A$ is degreewise a free $Q$-module, and $A\to R$ is a morphism of dg algebras inducing an isomorphism in homology. The \emph{derived fiber} of $\phi$ is the dg $k$-algebra 
\[
R \lotimes_Q k \coloneqq A\otimes_Q k\,.
\]
Up to a zig-zag of quasi-isomorphisms of dg $k$-algebras, $R \lotimes_Q k$ is independent of the choice of $A$. For more information see \cite{Avramov:1986}. We remark that one can equally well use a different species of model for the resolution $A$, such as simplicial algebras or $\ai$-algebras, and obtain an equivalent definition of Koszul homomorphism. Indeed, we make use of $\ai$-models in \cref{sec:koszul-via-ai}. 

\begin{definition}\label{koszul-ring-map}
Let $\phi \colon Q \to R$ be a finite local homomorphism. We say that $\phi$ is \emph{Koszul} if $R \lotimes_Q k$ is a Koszul dg $k$-algebra; that is, $R \lotimes_Q k$ is formal and its homology $\Tor{Q}{R}{k}$ is Koszul in the sense of \cref{koszul-algebra}. 
\end{definition}

According to \cref{koszul-algebra}, when $\phi$ is Koszul, the Tor algebra $\Tor{Q}{R}{k}$ admits a quadratic presentation, albeit not necessarily generated by elements in homological degree one.

Taking $Q=k$ to be a field, we recover the classical definition: The homomorphism $k\to R$ is Koszul if and only if $R$ is a Koszul $k$-algebra; cf.\@ \cref{koszul-algebra}. 

The examples given in the next section show that Koszul homomorphisms are extremely common. In particular, this class includes all flat local homomorphisms whose fibers are (classically) Koszul; all complete intersection and all Golod homomorphisms; Cohen presentations of most local rings having codepth at most $3$, and of generic Gorenstein rings. 

\begin{remark}\label{r_minres}
When the minimal $Q$-free resolution $A$ of $R$ admits a dg algebra structure, the derived fiber $R\lotimes_Q k= A\otimes_Rk$ has zero differential, and is automatically formal. This is the case for complete intersection homomorphisms (\cref{e:koszul:ci}) and when $\pdim_Q(R)\leqslant 3$ (\cref{p_codepth_three}). When $\phi$ is Golod the minimal resolution typically does not support a dg algebra structure, cf.\@ \cref{r_no_dg_structure}, but nonetheless $R\lotimes_Q k$ is formal (\cref{e:koszul:golod}). The monomial rings presented in \cref{e:koszul:nonexample,e:koszul:nonexample'} yield non-Koszul homomorphisms; the corresponding minimal resolution does not admit a dg algebra structure in the former example (see \cite[2.2]{Avramov:1981}), but it does in the latter example. 
\end{remark}

\subsection{Cohen Koszul local rings} 
\label{sec_cohen_kos}
Recall that every local ring $R$ admits a Cohen presentation, that is, a surjection $\phi\colon Q\to \widehat{R}$ from a regular local ring $Q$, and one may assume that $\phi$ is minimal in the sense that $\ker(\phi) \subseteq \fm_Q^2$. 

\begin{definition}\label{d_Cohen_Koszul}
A local ring $R$ is \emph{Cohen Koszul} if every homomorphism $\phi\colon Q\to \widehat{R}$ is Koszul, where $Q$ is a regular local ring and $\phi$ is surjective with $\ker(\phi) \subseteq \fm_Q^2$. In other words, every minimal Cohen presentation of $R$ is Koszul.
\end{definition}

\begin{remark}\label{r_characteristic_and_koszul_complex}
For equicharacteristic rings the minimal Cohen presentation $Q\to \widehat{R}$ is essentially unique. In this situation, if $R$ is already a quotient of a regular local ring, then by \cref{r_completion} there is no need to complete $R$ to determine whether $R$ is Cohen Koszul.

If we assume that $R$ contains its residue field $k$, then the Koszul complex $K^R$ on the maximal ideal of $R$ is a dg $k$-algebra. It is well-known that
\[
K^R\simeq \widehat{R}\lotimes_Q k
\]
as dg $k$-algebras, see for example \cite[Theorem~8.1]{Avramov/Buchweitz/Iyengar/Miller:2010}. Therefore in this situation we can say that $R$ is Cohen Koszul exactly when $K^R$ is a Koszul dg $k$-algebra. 

If $R$ does not contain its residue field, the fact that $K^R$ is not a dg $k$-algebra introduces subtleties. The distinction between formality of dg $k$-algebras and formality of dg rings means that it is not clear whether \cref{d_Cohen_Koszul} is independent of the choice of Cohen presentation. In all of our examples however, the choice will be irrelevant. 

Complete intersection rings are Cohen Koszul by \cref{e:koszul:ci}, Golod rings are Cohen Koszul by \cref{e:koszul:golod}, and most rings of codepth $3$ are Cohen Koszul according to \cref{p_codepth_three}.
\end{remark}

Cohen Koszul local rings have rational Poincar\'e series. Recall that the Poincar\'e series of a finitely generated $R$-module $M$ is 
\[
\ps^R_M(t)=\sum_{n \in \mathbb{Z}} \rank_k(\Tor[n]{R}{M}{k}) t^n\,.
\]

\begin{proposition} \label{p_cohen_koszul_poincare}
Let $R$ be a Cohen Koszul local ring with embedding dimension $e$ and residue field $k$. Fix a weight grading making the Koszul homology $\H(K^R)$ a Koszul $k$-algebra. Then
\[
\ps^R_k(t)=\frac{(1+t)^e}{\sum_{i,w}(-1)^w \rank_k (\H_i(K^R)_{(w)})t^{i+w}}\,.
\]
\end{proposition}

\begin{remark}
We note that since $\H_*(K^R)$ is generated in weight $1$, the rank of  $\H_i(K^R)_{(w)}$ is equal to the rank of $[\H_{>0}(K^R)^{w}/\H_{>0}(K^R)^{w+1}]_i$. Therefore is is not necessary to \emph{choose} a weight grading to calculate the Poincar\'e series above.
\end{remark}

\begin{proof}
Let $T=\H(K^R)$, bigraded by homological degree and by weight, and write
\[
\H_T(s,t) \coloneqq \sum_{i,w} \rank_k(T_{i,(w)})t^is^w \quad \text{and} \quad \ps^T_k(s,t) \coloneqq \sum_{i,w} \rank_k( \Tor[w]{T}{k}{k}_{i})t^is^w\,.
\]
In $\Tor[w]{T}{k}{k}_{i}$, the index $w$ is the usual homological grading of $\operatorname{Tor}$, and $i$ is the extra grading that comes from the homological grading on $T$. Since $T$ is Koszul with respect to its weight grading, the usual computation of the Poincar\'e series of a Koszul algebra shows that $\H_T(s,t)\ps^T_k(-s,t)=1$; see \cite{Lofwall:1986}. 

Let $Q\to \widehat{R}$ be a minimal Cohen presentation. Formality of $\widehat{R}\lotimes_Qk$ implies that the spectral sequence \cite[6.2.1]{Avramov:1981} is degenerate, and so from \cite[6.2~(b')]{Avramov:1981} we obtain the first equality below, which yields the desired series
\[
\ps^R_k(t)=(1+t)^e\ps^T_k(t,t)=\frac{(1+t)^e}{\H_T(-t,t)}\,.\qedhere
\]
\end{proof}

\cref{p_cohen_koszul_poincare} recovers the known Poincar\'e series for complete intersection rings, Golod rings, and almost Golod Gorenstein rings; see \cref{sec:examples-koszul} for the latter.

There are a number of results that apply to certain subsets of Cohen Koszul rings, motivating the study of whether such a property holds for these rings in general.  We highlight a couple instances below. 

\begin{remark}
    Recently, Brown--Dao--Sridhar have shown that over complete intersection and Golod rings, the ideals of minors of differentials in minimal free resolutions are eventually two-periodic \cite{Brown/Dao/Sridhar:2023}. It would be worthwhile, and seems plausible  (in light of the structural result in \cref{priddy-resolution}), to determine whether (strictly) Cohen Koszul rings satisfy this property more generally. 
\end{remark}

\begin{remark}
Lower bounds on the Loewy length of the homology module of perfect complexes are of interest in both algebra and topology; see, for example, \cite{Avramov/Buchweitz/Iyengar/Miller:2010,Allday/Puppe:1993,Carlsson:1983,Iyengar/Walker:2018,Walker:2017}. For Cohen Koszul rings one can establish such bounds. 

Let $R$ be a local ring with residue field $k$, and let $k[\chi_1,\ldots,\chi_n]$ denote a maximal polynomial subalgebra of the graded $k$-algebra $\Ext{R}{k}{k}$, generated by elements in even degree. For example, if $R$ is complete intersection, then $n$ is the codimension of $R$. If $R$ is Cohen Koszul, then for any finite free $R$-complex $F$ with $\H(F)\neq 0$ one has the inequality
\[
\sum_{n\in \mathbb{Z}} \ell\ell_R\H_n(F)\geqslant n+1\,.
\]
One can use similar ideas to those in \cite{Avramov/Buchweitz/Iyengar/Miller:2010}, as well as \cite{Briggs/Grifo/Pollitz:2024}, to establish this bound; here, however, formality of the derived fiber of a Cohen presentation of $R$ is a main ingredient. Moreover, when $R$ is complete intersection it agrees with the common bounds from \cite{Avramov/Buchweitz/Iyengar/Miller:2010,Briggs/Grifo/Pollitz:2024}.
\end{remark}

\subsection{Properties of Koszul homomorphisms} 

Before moving on to examples we establish some basic change of rings properties for Koszul homomorphisms. First, we note that being Koszul is invariant under certain flat base changes, and in particular under completion.

\begin{proposition}\label{r_completion}
Given a finite local homomorphism $\phi\colon Q\to R$ and a flat local homomorphism $\psi\colon Q\to Q'$ inducing an isomorphism on residue fields, $\phi$ is Koszul if and only if $\phi\otimes Q' \colon Q'\to R\otimes_Q Q'$ is Koszul.
\end{proposition}

\begin{proof}
The natural map $R\lotimes_{Q} k\to (R\otimes_Q Q')\lotimes_{Q'} k$ is a quasi-isomorphism of dg $k$-algebras. Indeed, if $A$ is a dg algebra resolution of $R$ over $Q$, then $A\otimes_QQ'$ is a dg algebra resolution of $R\otimes_Q Q'$ over $Q'$, and $(A\otimes_QQ')\otimes_{Q'}k \cong A\otimes_Q k$.
\end{proof}

The next proposition will often be useful in reducing the dimension of $Q$ or $R$.

\begin{proposition}\label{p_change_of_rings}
Let $\phi\colon Q\to R$ be a finite local homomorphism, and let $x\in \fm_Q$ and $y\in \fm_R$.
\begin{enumerate}[(1)]
\item\label{nzd_p1} 
If $x$ is regular on $Q$ and $y$ is regular on $R$, with $\phi(x)=y$, then $\phi$ is Koszul if and only if the map of quotients $Q/(x)\to R/(y)$ is Koszul.
\item\label{nzd_p2}
If $y$ is regular on $R$ then $\phi$ is Koszul if and only if the composition $Q\to R/(y)$ is Koszul.
\item\label{nzd_p3} 
If $\phi(x)=0$ and $x$ generates a free $R$-module summand of $\ker(\phi)/\ker(\phi)^2$, then $\phi$ is Koszul if and only if the induced map $Q/(x)\to R$ is Koszul. 
\end{enumerate}
\end{proposition}

\begin{proof}
For part \cref{nzd_p1}, let $A\xra{\simeq} R$ be a dg algebra resolution of $R$ over $Q$. The assumptions on $x$ and $y$ imply that $A\otimes_Q Q/(x)$ is a dg algebra resolution of $R/(y)$ over $Q/(x)$. In particular there are quasi-isomorphisms
\[
R\lotimes_Q k \simeq A\otimes_Q k\cong (A\otimes_Q Q/(x))\otimes_{Q/(x)} k \simeq R/(y)\lotimes_{Q/(x)} k
\]
of dg $k$-algebras, and the claim follows.

For part \cref{nzd_p2}, if $A$ is is a dg algebra resolution of $R$ over $Q$, as above, then there is an element $\tilde{y}\in \fm_QA_0$ mapping to $y\in R$. Taking an exterior variable $e$ of degree $1$, and setting $\partial(e)=\tilde{y}$, the extension $A\langle e\rangle$ is then a dg algebra resolution of $R/(y)$ over $Q$; see \cite[6.1]{Avramov:1998}. We see that 
\[
R/(y)\lotimes_Q k \simeq A\langle e\rangle \otimes_Q k\cong (A\otimes_Q k) \otimes_k \Lambda_k(e) \simeq (R \lotimes_Q k) \otimes_k \Lambda_k(e)\,,
\]
where $\Lambda_k(e)$ is the exterior algebra over $k$ on the degree $1$ variable $e$. Hence it remains to note that the tensor product of dg $k$-algebras is formal if and only if both of its factors are formal, and hence Koszul if and only if both of its factors are Koszul; see \cite[Theorem~2]{Froeberg:1999} for the latter. 

For part \cref{nzd_p3} we invoke \cite[Proposition~2.1]{Iyengar:2001} to obtain a dg algebra resolution $A$ of $R$ over $Q$ and an isomorphism of dg $k$-algebras $A\otimes_Q k \cong W \otimes_k \Lambda_k(e)$, where $W$ is a dg subalgebra of $A\otimes_Q k$ and $\Lambda_k(e)$ is the exterior algebra on a generator of degree $1$ (this result is based upon Andr\'e's theory of special cycles \cite{Andre:1982}). Moreover, the proof in \cite{Iyengar:2001} identifies the inclusion $\Lambda_k(e)\to A\otimes_Q k $ with the natural map $Q/(x)\lotimes_Qk \to R\lotimes_Qk$. It follows that 
\[
R\lotimes_{Q/(x)}k\simeq (R\lotimes_{Q}k)\lotimes_{ Q/(x)\lotimes_Qk}k \simeq (W \otimes_k \Lambda_k(e))\otimes_{\Lambda_k(e)} k \cong W\,.
\]
As in part \cref{nzd_p2} we may deduce that $R\lotimes_{Q/(x)}k$ is Koszul dg $k$-algebra if and only if $R\lotimes_{Q}k\simeq(R\lotimes_{Q/(x)}k)\otimes_k \Lambda_k(e) $ is as well.
\end{proof}

\begin{remark} \label{r_other_defs}
Many other Koszul-like properties have appeared in the literature. Within local commutative algebra, Herzog, Reiner, and Welker introduced a notion of Koszul local ring in \cite{Herzog/Reiner/Welker:1998}, and the same condition is investigated in \cite{Herzog/Iyengar:2005}. The local ring $k \llbracket x,y\rrbracket/(x^2-y^3)$ is Koszul in the sense of these references, but it is not a Koszul $k$-algebra according to \cref{koszul-algebra}, since it does not admit a quadratic presentation. However, the same ring $k\llbracket x,y\rrbracket/(x^2-y^3)$ is Koszul as a $k\llbracket y\rrbracket$-algebra by \cref{e:koszul:flat}, and it is Koszul as a $k\llbracket x,y\rrbracket$-algebra by \cref{e:koszul:ci}---in other words, it is Cohen Koszul.

Myers studied a Koszulity condition in \cite{Myers:2021} that is closely related to ours. That work begins with a standard graded $k$-algebra $R$, and its Koszul homology $\H(K^R)$ is said to be \emph{strand Koszul} if it is Koszul with respect to the induced weight grading by strands: $\H(K^R)_{(w)}=\bigoplus_{i+j=w}\H_i(K^R)_j$ (the total of the homological and internal gradings). According to \cite[Theorem~B]{Myers:2021} the Koszul complex $K^R$ is automatically \emph{quasi-formal}; this is a weakening of formality defined in terms of the degeneration of a certain Eilenberg--Moore spectral on its second page, see \cite[2.3]{Myers:2021}. In contrast, $R$ is Cohen Koszul if $K^R$ is formal and $\H(K^R)$ is Koszul with respect to \emph{some} weight grading. In the next section we see that there are many natural examples for which $\H(K^R)$ is Koszul with respect to a different grading than the strand grading. 

The authors of \cite{Croll/etal:2020} have also investigated how the Koszul condition on a local ring affects the algebra structure of the Koszul homology $\H(K^R)$. While this is connected to the present work, we note that there are many examples of local $k$-algebras that are Cohen Koszul but not Koszul as $k$-algebras.
\end{remark}

\begin{remark}
We end this section with remarks on the generality of \cref{koszul-ring-map}.

We have chosen to focus on the setting of finite $Q$-algebras because this is necessary to meaningfully talk about transferring homological information from $Q$ to $R$. However, the notion of Koszul homomorphism can be extended fruitfully to all local homomorphisms, with some additional technicalities. In particular, to accommodate non-finite algebras, \cref{koszul-algebra} should be adapted to require that the completion of $K$ at its augmentation ideal is isomorphic to $\prod_{w\geqslant 0} K_{(w)}$, and that $k$ admits a linear resolution over the corresponding graded ring $\bigoplus_{w\geqslant 0} K_{(w)}$. 

For $Q$ non-local, one can say a $Q$-algebra $R$ is \emph{Koszul at a prime} $\fp \in \operatorname{Spec}(Q)$ if $\kappa(\fp) \lotimes_Q R$ is a Koszul dg algebra over $\kappa(\fp)=Q_\fp/\fp Q_\fp$. From this perspective it is natural to replace $R$ with a sheaf of algebras on some scheme; examples related to this have appeared in the literature, such as the sheaf of Clifford algebras constructed by Buchweitz in \cite[Appendix]{Buchweitz/Eisenbud/Herzog:1985}. 

In this work we focus on applications to commutative algebra. The natural generalization to non-commutative algebras is interesting as well, using exactly the same definitions.
\end{remark}

\section{Examples of Koszul homomorphisms} \label{sec:examples-koszul}

This section contains examples (and counterexamples) demonstrating the ubiquity of the Koszul condition. The first class of examples generalizes the class of Koszul algebras over a field in a straightforward manner. 

\begin{example}[Flat homomorphisms with Koszul fiber] \label{e:koszul:flat}
A flat finite local homomorphism $\phi \colon Q \to R$ is Koszul if and only if its fiber $R \otimes_Q k$ is a Koszul $k$-algebra. Such examples are readily constructed by deforming presentations of known Koszul algebras. For example, the $k$-algebra $k[x]/(x^2)$ is Koszul, and so the homomorphism
\[
Q\to Q[x]/(x^2-ax-b)
\]
is Koszul for any $a,b\in \fm_Q$.
\end{example}

We will see that the non-flat case is significantly more interesting, and crucially \emph{there is no need for the map $\phi \colon Q \to R$ to be quadratic in any sense}, as the following examples demonstrate. Nonetheless, later we return to the idea that Koszul homomorphisms look like deformations of classical Koszul presentations; cf.\@ \cref{koszul-ai-presentation}.

\begin{example}[Complete intersection homomorphisms]\label{e:koszul:ci}
Let $\phi\colon Q\to R$ be a surjective, local, complete intersection homomorphism of codimension $c$. That is, $\ker(\phi)$ is generated by a $Q$-regular sequence $\bm{f}=f_1,\ldots, f_c$. In this case, the Koszul complex $A=\Kos^Q(\bm{f})$ is a dg algebra resolution of $R$ over $Q$. Then
\[
R \lotimes_Q k = A\otimes_Q k = \Lambda_k(\susp A_1\otimes_Q k)
\]
is the exterior algebra on a $k$-space of rank $c$ in homological degree one, with zero differential. Thus the derived fiber of $\phi$ is clearly formal, and it is well-known to be a Koszul $k$-algebra with its weight homological gradings coinciding; cf.\@ \cite[Examples~2.2(2)]{Priddy:1970}. In particular, a local complete intersection ring is Cohen Koszul. 
\end{example}

\begin{remark}
\label{r_different_grading}
For a Cohen Koszul ring $R$, the homological and weight grading on $\H(K^R)$ coincide if and only if $R$ is complete intersection. Indeed, the reverse implication was indicated in \cref{e:koszul:ci}, and the forward implication follows from a Theorem of Wiebe~\cite{Wiebe:1969}; see also \cite[Theorem~2.3.15]{Bruns/Herzog:1998}.
\end{remark}

\begin{example}[Trivial extension algebras]
\label{c:trivial_extension}
Given a graded ring $B$ and a graded $B$-module $U$, let $B\ltimes U$ denote the trivial extension of $B$ by $U$. This is the graded module $B\oplus U$ with multiplication 
\[
(b,u)\cdot (b',u')=(bb',bu'+b'u)\,,
\] 
and zero differential. The main case of interest is that $B$ is augmented to $k$, and $U$ is a graded $k$-space thought of as a trivial $B$-module. If $B$ is also a $k$-algebra, then $B$ is a Koszul if and only if $B\ltimes U$ is by \cite{Cheng:2017}.

In particular, for any $k$-space $U$ the local $k$-algebra $k\ltimes U$ is Koszul. It is also Cohen Koszul, according to \cref{e:koszul:golod}.
\end{example}

For a surjective local map $\phi\colon Q\to R$ and $R$-module $M$ there is the following coefficientwise inequality of Poincar\'e series:
\begin{equation}\label{golod_bound}
\ps^R_M(t)\preccurlyeq \dfrac{\ps^Q_M(t)}{1-t(\ps^Q_R(t)-1)}\,.
\end{equation}
This fact is due to Serre; see, for example, \cite[Proposition~3.3.2]{Avramov:1998}.

\begin{example}[Golod homomorphisms] \label{e:koszul:golod}
Let $\phi \colon Q \to R$ be a surjective, local, Golod homomorphism. That is, the Serre bound \cref{golod_bound} is an equality for the residue field $M=k$.

Avramov proved that a surjective, local homomorphism is Golod if and only if there is a quasi-isomorphism of dg algebras
\[
R \lotimes_Q k\simeq k \ltimes U\,,
\]
where $U$ is a positively graded vector space over $k$. This follows by applying \cite[Theorem~2.3]{Avramov:1986} to a (possibily non-minimal) dg $k$-algebra model for $R \lotimes_Q k$. The trivial extension algebra $k\ltimes U$ is Koszul by \cite[Proposition~3.4.9]{Loday/Vallette:2012}, see also \cref{c:trivial_extension}, and hence $\phi$ is Koszul. 
\end{example}

The examples above provide many instances of Cohen Koszul $k$-algebras appearing in local commutative algebra, and \emph{some} of these examples are Koszul in the classical sense: for example a local complete intersection $k$-algebra is Koszul if and only if it is quadratic; see \cite[3.1]{Froeberg:1999}. We give a small example of a local $k$-algebra that is neither Koszul nor Cohen Koszul.

\begin{example}
\label{e:koszul:nonexample'}
Suppose $R=k\llbracket a,b,c\rrbracket/(a^2,bc,ac+b^2)$. A computation shows $\Lambda_k(e_1,e_2,e_3)/(e_1e_2e_3)$ is an algebra retract of $\H(K^R)$; one could, for example, use \texttt{Macaulay2}~\cite{M2} for this calculation. In particular, $\H(K^R)$ has a relation of weight $3$ in any weight grading, and so $\H(K^R)$ cannot be a Koszul $k$-algebra. Thus, $R$ is not Cohen Koszul. Moreover, $R$ is the completion of a quadratic algebra that is not Koszul in the classical sense. One can see this by computing the third differential in the minimal free resolution of $k$ over $R$; alternatively, see \cite{Backelin/Froeberg:1985b}.

In \cref{p_codepth_three}, we see that this is part of an exceptional class of non-Cohen Koszul local rings among rings having embedding dimension at most three. 
\end{example}

\begin{example}[Short Gorenstein local rings]
\label{c:short_gorenstein} 
A local ring $R$ with maximal ideal $\fm_R$ is called \emph{short Gorenstein} if it is Gorenstein and $\fm_R^3=0$. Equivalently, these are the local rings having Hilbert series $\H_R(t)=1+nt+t^2$ for some $n$. This is an important class of local rings that occurs frequently in what follows. If $R$ is also a $k$-algebra, then $R$ is Koszul by \cite{Froeberg:1982} or \cite{Schenzel:1980} (this follows as well from the slightly earlier computations in \cite{Avramov/Levin:1978}). Short Gorenstein rings are also Cohen Koszul by \cref{e:koszul:punctured-golod}.
\end{example}

\subsection{Rings of small codepth} \label{e:koszul:small-codepth}

Recall that for a local ring $R$ with maximal ideal $\fm_R$ and residue field $k$, its codepth is 
\[
\codepth(R) \coloneqq \rank_k (\fm_R/\fm_R^2)-\operatorname{depth}(R)\,.
\]
This value is a measure of the singularity of $R$ in the sense that $\codepth(R)=0$ if and only if $R$ is regular. The next result illustrates that a local ring of small codepth is almost always Cohen Koszul. First we remind the reader of the structure theorem on Koszul homology for rings having codepth three. 

\begin{chunk}
\label{c:codepth3}
Assume $R$ has codepth three and fix a minimal Cohen presentation $\phi\colon Q\to \widehat{R}$. The minimal $Q$-free resolution $A$ of $\widehat{R}$ supports a dg algebra structure; see, for example, \cite{Buchsbaum/Eisenbud:1977}. Hence, $\widehat{R}\lotimes_Q k$ is formal and the algebra structure of its homology $T=\Tor{Q}{\widehat{R}}{k} =\H(K^R)$ has been classified as follows.

Fix bases $\{e_1, \ldots, e_\ell\}$, $\{f_1, \ldots, f_m\}$ and $\{g_1, \ldots, g_n\}$ for $T_1$, $T_2$ and $T_3$, respectively. By \cite{Avramov/Kustin/Miller:1988}, there are non-negative integer parameters $p,q,r$, satisfying 
\[
p \leqslant \ell-1, \quad q \leqslant m-p, \quad r \leqslant \min\{\ell,m\}\,,
\]
such that $T$ is one of the graded-commutative algebras determined below, where products between the basis elements not listed are zero: 
\begin{description}[labelindent=1em]
\item[\textbf{CI}] $e_1e_2=f_3,$ $e_1e_3=f_2$, $e_2e_3=f_1$, $e_if_i=g_1$ for $i=1,2,3.$ 
\item[\textbf{TE}] $e_1e_2=f_3,$ $e_1e_3=f_2$, $e_2e_3=f_1$.
\item[\textbf{B}] $e_1e_2=f_3,$ $e_1f_1=g_1$, $e_2f_2=g_1$.
\item[\textbf{G}$(r)$] $e_if_i=g_1$ for $i=1,\ldots, r$ and $r\ge 2$
\item[\textbf{H}$(p,q)$] $e_{p+1}e_i=f_i$ for $i=1,\ldots,p$, and $e_{p+1}f_{p+i}=g_i$ for $i=1,\ldots, q$.
\end{description}
In each case, let $T'$ denote the corresponding subalgebra on the basis elements appearing in the multiplication table above. If $U$ is the $k$-space spanned by the basis of elements of $T$ not recorded in the same multiplication table, then note that there is an isomorphism of graded $k$-algebras 
\begin{equation}
\label{e_T_trivial_extension}
T\cong T'\ltimes U\,.
\end{equation}
Finally, as a matter of terminology, we say a local ring $R$ belongs to one of these classes if $T=\Tor{Q}{\widehat{R}}{k}$ has the corresponding algebra structure.
\end{chunk}

\begin{theorem}\label{p_codepth_three}
A local ring of codepth two or less is Cohen Koszul. A local ring of codepth three is Cohen Koszul if and only if it belongs to $\mathbf{CI}$, $\mathbf{B}$, $\mathbf{G}(r)$ or $\mathbf{H}(p,q)$. 
\end{theorem}

\begin{proof}
Let $R$ be a local ring of codepth $c$ with residue field $k$. Fix a minimal Cohen presentation $\phi\colon Q\to \widehat{R}$ and set $T=\Tor Q{\widehat{R}} k$. 

If $c\leqslant 2$, then $R$ must be complete intersection or Golod; that is, $\phi$ is a complete intersection or Golod homomorphism. This follows from the Hilbert--Burch theorem (see \cite{Burch:1968}, as well \cite[Theorem~1.4.17]{Bruns/Herzog:1998}) combined with a result of \cite[Theorem~2.3]{Levin:1976}; see also \cite[Proposition~5.3.4]{Avramov:1999}. Therefore by \cref{e:koszul:ci,e:koszul:golod} $R$ is Cohen Koszul in either case.

Now assume $c=3$. The simplest case is when $R$ belongs to \textbf{CI}, since in this case $R$ is complete intersection and so $R$ is Cohen Koszul; cf.\@ \cref{e:koszul:ci}.
 
For the remainder of the proof we adopt the notation from \cref{c:codepth3} and analyze the graded algebra structure of $T$. By \cref{c:trivial_extension,e_T_trivial_extension}, $T$ is Koszul if and only if $T'$ is, and so we replace $T$ by $T'$ in what follows.

If $T$ is \textbf{H}$(p,q)$, then it is the tensor product of a trivial extension algebra on $e_1,\dots,e_p,f_{p+1},\dots,f_{p+q}$ and the exterior algebra on $e_{p+1}$, where each of these have weight 1. Hence, $T$ is a tensor product of Koszul $k$-algebras and so it is Koszul. 

If $T$ is \textbf{G}$(r)$, we give $\{e_i\}$ and $\{f_i\}$ weight one and $g_1$ weight two. Then $T$ is a short Gorenstein $k$-algebra and hence Koszul; see \cref{c:short_gorenstein}.

If $T$ is \textbf{B}, then we give $T$ weight grading 
\[
T_{(w)}=\begin{cases}
k & w=0\\
k e_1 \oplus k e_2 \oplus k f_1 \oplus k f_2& w=1 \\
kf_3\oplus kg_1 & w=2 \\
0 & w\geqslant 3\,.
\end{cases}
\]
As an algebra $T$ is the quotient of the exterior algebra 
\[
T \cong \Lambda_k(e_1,e_2,f_1,f_2)/(e_1f_2,e_2f_1,f_1f_2,e_1f_1-e_2f_2)\,.
\]
The graded $k$-algebra $T$ is Koszul since its defining ideal has a quadratic Gr{\" o}bner basis, and is therefore Koszul by \cite{McCullough/Mere:2022}. 

Finally, if $T$ is type \textbf{TE}, then $T$ is not quadratic with respect to any weight grading. Indeed, the products in \cref{c:codepth3} force $e_1,e_2,e_3$ to all have weight one, as well as the minimal relation $e_1e_2e_3=0$. Hence, $T$ is not Koszul.
\end{proof}

\begin{remark}
\label{r_codepth_three}
\cref{p_codepth_three} deals with the `absolute' case. That is to say, when a local ring is Cohen Koszul. However given a surjective map of local rings $\phi\colon Q\to R$ one has that this is always Koszul for $\pdim_Q(R)\leqslant 2$. Indeed, in this case $\phi$ is either a complete intersection homomorphism, or a Golod homomorphism. For the case $\pdim_Q(R)=3$, the structure theorem on $T=\Tor QRk$ discussed in \cref{c:codepth3} can be applied assuming that $\ker(\phi)$ is a perfect ideal. In this case $\phi$ is Koszul except in the case that $T$ belongs to \textbf{TE}. 
\end{remark}

A surjective local homomorphism $\phi\colon Q\to R$ of finite projective dimension is \emph{Gorenstein of projective dimension $d$} if 
\begin{equation}\label{eq_gor_def}
\mathrm{Ext}_Q^i(R,Q)=\begin{cases} R &i=d \\
0 & i\neq d\,.\end{cases}
\end{equation}
For example, Gorenstein rings of codimension $d$ are exactly those whose minimal Cohen presentations are Gorenstein of projective dimension $d$. If $d=3$ then $\Tor QRk$ belongs to {\bf G}$(r)$ and the dg algebra structure on the minimal resolution of $R$ over $Q$ can be described explicitly; one can hence verify directly, as is done below, that such maps are Koszul. 

\begin{example}[Gorenstein homomorphisms of projective dimension $3$] 
\label{e:koszul:buchsbaum-eisenbud} 
\hfill Assume \\
that $\phi$ is Gorenstein of projective dimension $3$. Buchsbaum and Eisenbud constructed the minimal free resolution of $R$ over $Q$ in \cite[Theorem~2.1 \& 4.1]{Buchsbaum/Eisenbud:1977}:
\[
A = 0\to Q \to Q^r\xra{\psi} Q^r\to Q\to 0
\]
where $r \geqslant 3$ is odd and the first and third differential of $A$ depend on Pfaffians of submatrices of the alternating matrix $\psi$. Furthermore $A$ is a graded-commutative dg algebra with the following multiplication: We fix bases $\{e_i\}_{i=1}^r$, $\{f_i\}_{i=1}^r,$ and $\{g\}$ for $A_1,$ $A_2,$ and $A_3$, respectively. The multiplication is determined by
\[
e_i e_j \coloneqq \sum_{\ell=1}^r (\pm 1) \operatorname{pf}(\psi_{ij\ell}) f_\ell \quad\text{for } i<j \,,\quad e_i f_j \coloneqq \delta_{ij} g \quad\text{and}\quad f_i f_j = 0
\]
where $\psi_{ij\ell}$ is the submatrix of $\psi$ obtained by deleting the $i$th, $j$th and $\ell$th row and column, and $\delta_{ij}$ is the Kronecker delta function. The exact description is not important for the sequel, see \cite[Example~2.1.3]{Avramov:1998} for detail. 

When $r=3$, it follows that $A$ is a Koszul complex on three elements and so $\phi$ is a surjective complete intersection map, hence $\phi$ is Koszul by \cref{e:koszul:ci}.

When $r \geqslant 5$, we have that $\operatorname{pf}(\psi_{ij\ell}) \in \fm_Q$ for any $i$, $j$ and $\ell$. Hence the only non-zero products in the graded algebra $A\otimes_Q k$ are 
\[
e_if_i=f_ie_i=g \quad \text{for } 1 \leqslant i \leqslant r\,.
\]
Giving $A\otimes_Q k$ the weight grading 
\[
(A\otimes_Q k)_{(w)}=\begin{cases}
A_0\otimes_Qk & w=0\\
(A_1\otimes_Qk)\oplus (A_2 \otimes_Qk)& w=1\\
A_3\otimes_Qk & w=2\\
0 & \text{else}\,,
\end{cases}
\]
it belongs to \textbf{G}($r$) in \cref{c:codepth3}. The fact that $A\otimes_Q k$ is a Koszul $k$-algebra was established in the proof of \cref{p_codepth_three}. Thus $\phi$ is a Koszul homomorphism.
\end{example}

\subsection{Almost Golod Gorenstein rings}
\label{e:minl-non-golod}

We discuss here a large class of local Gorenstein rings displaying interesting homological behavior, that has been studied before in \cite[Section 6]{Rossi/Sega:2014}.

\begin{definition}
We say that an artinian local ring $R$ is \emph{almost Golod} if the socle quotient $R/\soc(R)$ is Golod. A general local ring is \emph{almost Golod} if it is Cohen--Macaulay and $R/(\bm{x})$ is an almost Golod artinian ring, where $\bm{x}$ is a maximal regular sequence that is part of a minimal generating set for $\fm_R$. 
\end{definition}

\begin{example}[Almost Golod Gorenstein rings] \label{e:koszul:punctured-golod}
Let $R$ be an almost Golod local ring that is also Gorenstein of codepth $d$. 

Fix a minimal Cohen presentation $\phi \colon Q \to \widehat{R}$ and set $T= \Tor Q{\widehat{R}}k$. Since $Q$ is regular and $R$ is Gorenstein, $T$ is a Poincar\'e duality algebra by \cite{Avramov/Golod:1971}. That is to say, for each $0 \leqslant i \leqslant d$ the multiplication maps 
\[
T_i\times T_{d-i}\to T_d\cong k
\]
are perfect pairings. Furthermore, by \cite[Theorem~1]{Avramov/Levin:1978} the quotient $T/T_d$ is a subalgebra of the trivial extension algebra $\Tor Q{R/\soc(R)}k$, and hence is itself a trivial extension algebra. It follows that $T$ is a short Gorenstein $k$-algebra. Moreover, prescribing $T$ with the following weight grading 
\[
T_{(0)}=T_0,\quad T_{(1)} = {\textstyle\bigoplus_{i=1}^{d-1}} T_i,\ \text{ and }\ \ T_{(2)}=T_d 
\]
makes $T$ a Koszul $k$-algebra, with the multiplication of $T$ being equivalent to a perfect pairing on $T_{(1)}$; see the proof of \cref{p_codepth_three}. 

We prove that these rings are Cohen Koszul under the assumption that $R$ contains a field of characteristic zero and $d$ is odd, by giving a characterization analogous to Avramov's characterization of Golod rings \cite{Avramov:1986}. We do not know whether the assumption on the characteristic or on $d$ is necessary (but see \cref{rem_berglund}).

\begin{theorem}
\label{t_almost_golod_gor_characterisation}
Let $R$ be a local with a minimal Cohen presentation $Q\to\widehat{R}$. If there is a quasi-isomorphism of dg $k$-algebras $\widehat{R}\lotimes_Qk\simeq T$, where $T$ is a short Gorenstein graded $k$-algebra, then $R$ is almost Golod Gorenstein. Assuming that $R$ contains a field of characteristic zero, and that $\codepth(R)$ is odd, the converse holds as well. In particular, almost Golod Gorenstein rings (of characterstic zero and odd codepth) are Cohen Koszul.
\end{theorem}

\begin{proof}
If $R$ is almost Golod Gorenstein, we have already seen that $T=\Tor Q{\widehat{R}}k$ is a short Gorenstein algebra, and in particular Koszul. The proof that $\widehat{R}\lotimes_Q k$ is formal under the stated assumptions will be given in \cref{p_cyclic_ai} and \cref{t_golod_gorenstein_koszul}.

Conversely assume that $\widehat{R}\lotimes_Qk$ is quasi-isomorphic to a short Gorenstein algebra $T$, and write $e$ for the embedding dimension of $R$. By \cref{p_cohen_koszul_poincare}, 
\[
{\ps^{{R}}_k(t)} = \frac{(1+t)^e }{1-(\H_T(t) - 1 - t^{e})t+t^{e+2}}\,,
\]
and hence as a consequence of \cite[Proposition~6.2]{Rossi/Sega:2014}, $R$ is almost Golod. 
\end{proof}
\end{example}

\begin{remark}
\label{r_list_of_examples_almost_golod_gor}
The prototypical example of an almost Golod Gorenstein ring is a short Gorenstein local ring $R$. In this case $R/\soc(R)=R/\fm_R^2$ is Golod by \cite{Golod:1978}.

Among the complete intersection local rings, the almost Golod Gorenstein rings are exactly those having codimension two or less; see \cref{t_almost_golod_gor_characterisation}. If $R$ is a Gorenstein local ring of codimension $3$ that is not complete intersection, then $R$ is almost Golod Gorenstein by \cref{e:koszul:buchsbaum-eisenbud} and \cref{t_almost_golod_gor_characterisation}. 

By \cite[Proposition~6.3]{Rossi/Sega:2014} every Gorenstein compressed local ring of socle degree at least $4$ is almost Golod Gorenstein. Moreover for fixed emdedding dimension and socle degree, the generic Gorenstein local $k$-algebra is compressed by \cite[Theorem~I]{Iarrobino:1984}. Hence the almost Golod Gorenstein condition is extremely common. 

Recall that if $Q$ is a standard graded polynomial algebra, a homogeneous quotient $R=Q/I$ is said to have an \emph{almost linear} resolution over $Q$ if the ideal $I$ is generated by forms of degree $e$, and for all $0<i<\pdim_Q(R)$ we have $\Tor[i]{Q}{R}{k}_j=0$ unless $j-i=e-1$; confer \cite{Eisenbud/Huneke/Ulrich:2006}. Graded Gorenstein rings with almost linear resolutions are always almost Golod Gorenstein. This will be justified later in the paper, in \cref{ex_amost_lin_a_infinity}, along with the fact that such algebras with $e \geqslant 3$ are also Cohen Koszul; this is done without the assumptions on characteristic and codepth in \cref{t_almost_golod_gor_characterisation}.
\end{remark}
 
\begin{remark}\label{r_highly_connected}
In rational homotopy theory, Golod rings correspond to spaces that are (rationally) homotopy equivalent to a wedge of spheres, while Gorenstein rings are analogous to manifolds, or more generally Poincar\'e duality spaces; see the looking glass \cite{Avramov/Halperin:1986} for more information.

To be more precise, if $M$ is a simply connected manifold and the punctured space $M\smallsetminus\{{\rm pt}\}$ is rationally homotopy equivalent to a wedge of spheres, then the cohomology ring $\H^*(M;\mathbb{Q})$ is an almost Golod Gorenstein ring. In \cite{Stasheff:1983}, Stasheff proved such spaces are formal, and therefore they are Koszul in the sense of Berglund \cite{Berglund:2014}. A well studied class of manifolds satisfying this property are the \emph{highly connected manifolds}, that is, those $M$ with $\H^i(M;\mathbb{Q})=0$ when $0<i<\lfloor \dim(M)/2\rfloor$.
\end{remark}

Since Gorenstein rings that are not regular or hypersurfaces are never Golod, the Serre bound \cref{golod_bound} must be strict for such rings. However, a tighter bound can be established for Gorenstein local rings, as we show now. The case of equality below is equivalent (when $d=0$) to the formula for $\ps^R_k(t)$ given in \cite[Proposition~6.2]{Rossi/Sega:2014}, and our proof is essentially equivalent to that of \emph{loc.\@ cit}.

\begin{proposition}
\label{p_almost_golod_poincare}
Let $R$ be a local ring having dimension $d$ and embedding dimension $e$, with residue field $k$ and Koszul complex $K^R$. If $R$ is Gorenstein but not regular or a hypersurface, then there is a coefficientwise inequality
\begin{equation*}
\frac{\ps^R_k(t)}{(1+t)^d-t^2 \ps^R_k(t)} \preccurlyeq \frac{(1+t)^{e-d}}{1 - t^2 (1+t)^{e-d} + t^{e-d+2}-\sum_{i=1}^{e-d-1} \rank_k \H_i(K^R) t^{i+1}}
\end{equation*}
and equality holds if and only if $R$ is almost Golod.
\end{proposition}

While the left-hand side is not equal to $\ps^R_k(t)$, it increases monotonically with $\ps^R_k(t)$, and so it directly measures the growth of the resolution of $k$. Therefore within the class of Gorenstein local rings, almost Golod rings display extremal behavior analogous to Golod rings.

\begin{proof}
As $R$ is Gorenstein, prime avoidance yields a regular sequence $\bm{x}=x_1,\ldots ,x_d$ that is part of a minimal generating set for $\fm_R$, so that $\bar{R}=R/( \bm{x})$ is artinian Gorenstein and of embedding dimension $e-d$. By Nagata's theorem \cite[Section 27]{Nagata:1962} we have $\ps^R_k(t)=\ps^{\bar{R}}_k(t)(1+t)^d$ and so the first equality below holds:
\begin{align*}
\frac{\ps^R_k(t)}{(1+t)^d-t^2 \ps^R_k(t)} &=\frac{\ps^{\bar{R}}_k(t)}{1-t^2 \ps^{\bar{R}}_k(t)}\\ 
&=\ps^{\bar{R}/\soc(\bar{R})}_k(t) \\
&\preccurlyeq \frac{\ps^Q_k(t)}{1 - t (\ps^Q_{\bar{R}/\soc(\bar{R})}(t) - 1)}\\ 
&= \frac{\ps^Q_{k}(t)}{1 - t (\ps^Q_{\bar{R}}(t) - t^{e-d} + t \ps^Q_{k}(t) - t^{e-d+1} - 1)}\,.
\end{align*}
The second equality holds using \cite[Theorem~2]{Avramov/Levin:1978}, the coefficientwise inequality is the Serre bound \cref{golod_bound} for $\bar{R}/\soc(\bar{R})$, and the last equality holds using \cite[Theorem~1]{Avramov/Levin:1978}. Since $\ps^Q_{k}(t)=(1+t)^{e-d}$ and $\ps^Q_{\bar{R}}(t)=\sum_{i=0}^{i=e-d}\rank_k \H_i(K^R) t^{i}$, we obtain the claimed inequality.
It remains to note that $R$ is almost Golod if and only if $\bar{R}/\soc(\bar{R})$ is Golod, if and only if equality holds in the third line above.
\end{proof}

Later, in \cref{e:koszul-ai:minl-non-golod}, we explicitly construct resolutions over almost Golod Gorenstein rings that achieve the bound of \cref{p_almost_golod_poincare}. 

\subsection{Monomial rings} \label{e:monomial}

In this subsection we consider rings of the form $R=Q/I$, where $Q=k[x_1,\ldots,x_n]$ and $I$ is generated by monomials $m_1,\ldots,m_r$. By Fr\"oberg's theorem \cite{Froeberg:1975}, $R$ is Koszul as a $k$-algebra if and only if each $m_i$ is quadratic. However, the condition that $R$ is Cohen Koszul is more common and more subtle.

\begin{example}[Almost linear monomial ring] \label{e_almost_linear_monomial}
As discussed in \cref{r_list_of_examples_almost_golod_gor}, graded Gorenstein rings having almost linear resolutions are Cohen Koszul. To give an explicit example, let $I$ be the ideal in $Q=k[x_1,\ldots,x_8]$ generated by
\[
x_{2}x_{4}x_{5},\,x_{1}x_{3}x_{6},\,x_{2}x_{5}x_{6},\,x_{3}x_{5}x_{6},\,x_{1}x_{3}x_{7},\,x_{1}x_{4}x_{7},\,x_{2}x_{4}x_{7},\,x_{2}x_{6}x_{7},\quad
\]
\[
\quad x_{3}x_{6}x_{7},\,x_{4}x_{6}x_{7},\,x_{1}x_{3}x_{8},\,x_{1}x_{4}x_{8},\,x_{2}x_{4}x_{8},\,x_{1}x_{5}x_{8},\,x_{2}x_{5}x_{8},\,x_{3}x_{5}x_{8}.
\]
The Betti table of $R=Q/I$ is
\[
\begin{tabular}{c|c c c c c }
      & 0 & 1 & 2 & 3 & 4 \\ \hline
      \text{0} & 1 & $\cdot$ & $\cdot$ & $\cdot$ & $\cdot$ \\
      \text{1} & $\cdot$ & $\cdot$ & $\cdot$ & $\cdot$ & $\cdot$ \\
      \text{2} & $\cdot$ & 16 & 30 & 16 & $\cdot$ \\
      \text{3} & $\cdot$ & $\cdot$ & $\cdot$ & $\cdot$ & $\cdot$ \\
      \text{4} & $\cdot$ & $\cdot$ & $\cdot$ & $\cdot$ & 1 \nospacepunct{.}
\end{tabular}
\]
Of course, this example was found with the help of \texttt{Macaulay2}~\cite{M2}; it is the Stanley--Reisner ring of a triangulation of the $3$-sphere, taken from the enumeration compiled by Lutz \cite{Lutz:manifold}.
\end{example}

An exact combinatorial characterization of which monomial rings are Cohen Koszul would be very interesting; this seems possible but likely non-trivial. We describe a special case that produces a large number of explicit examples. 

Let $Q=k[x_1,\ldots,x_n]$ be a polynomial ring over a field $k$. A monomial ideal $I$ is called \emph{dominant} if it is generated by a set of monomials $G$ such that for all $m\in G$ there is a variable $x_i$ and an integer $a$ such that $x_i^a$ divides $m$ and $x_i^a$ does not divide any monomial $m'\in G\smallsetminus\{m\}$; see \cite[Definition~4.1]{Alesandroni:2017}.

\begin{proposition}\label{p_monomial_example}
If $Q=k[x_1,\ldots,x_n]$ be a polynomial ring over a field $k$ and $I$ is a dominant monomial ideal in $Q$ then $R=Q/I$ is Cohen Koszul.
\end{proposition}

\begin{proof}
The Taylor resolution $A$ of $R$ over $Q$ has a basis $\{e_I\}$ indexed by subsets $I\subseteq \{1,\ldots,n\}$, with $e_I$ in homological degree $|I|$, and the differential is defined by 
\[
\partial(e_I)=\sum_{i\in I} \pm \frac{m_I}{m_{I\smallsetminus \{i\}}} e_{I\smallsetminus \{i\}}\,,
\]
where $m_I=\operatorname{lcm}\set{m_i}{i\in I}$; see \cite{Taylor:1966} for details and signs. The hypothesis of the proposition exactly guarantees that the Taylor resolution is minimal by \cite[Theorem~4.4]{Alesandroni:2017}.

Gemeda \cite{Gemeda:1976} proved that the Taylor resolution has a dg algebra structure with product
\[
e_Ie_J = \pm \frac{m_Im_J}{m_{I\cup J}}e_{I\cup J}
\]
when $I\cap J=\varnothing$, and with $e_Ie_J = 0$ otherwise. By \cref{r_minres} it follows that $R\lotimes_Qk=A\otimes_Qk$ is formal.

It remains to show that $A\otimes_Qk$ is a Koszul $k$-algebra. Let $M$ be the graph with vertices $\{1,\ldots, r\}$ and an edge connecting $i$ and $j$ if and only if $\gcd(m_i,m_j)\neq 1$. From the description of $A$ above it follows that
\[
A\otimes_Qk = \frac{k[e_I \mid I\subseteq M \text{ connected}]}{(e_Ie_J \mid \gcd(m_I,m_J)\neq 1)}\,,
\]
where $k[e_I]$ is the free graded-commutative algebra on the indicated $e_I$; compare this with \cite[6.2]{Berglund:2005}. Assigning each $e_I$ weight $1$, we are done because quadratic monomial quotients of free graded-commutative algebras are Koszul by Fr\"oberg's theorem \cite{Froeberg:1975} (such rings belong to class B in \cite[Section 3]{Froeberg:1975}, and Fr\"oberg constructs linear resolutions of the residue field for all rings of class B).
\end{proof}

\begin{remark}
One can readily exhibit monomial rings satisfying the hypothesis of the proposition, and not falling into the other classes described above. For example, $R= k\llbracket a,b,c,d,e,f\rrbracket/(abc,cd,ae,acf)$.
\end{remark}

The next examples are $k$-algebras that fail to be Cohen Koszul; the first is a Koszul $k$-algebra and the second has $\H(K^R)$ a Koszul $k$-algebra. Both examples fail to be Cohen Koszul since in each case $K^R$ admits a nonzero triple Massey product, and hence is not formal; cf.\@ \cite{May:1969} for more details on Massey products. 

\begin{example} \label{e:koszul:nonexample}
The $k$-algebra $R=k\llbracket a,b,c,d\rrbracket/(a^2,ab,bc,cd,d^2)$ is the completion of a Koszul $k$-algebra (in the classical sense) by \cite[Corollary~1]{Froeberg:1975}. However, the map $k \llbracket a,b,c,d \rrbracket \to R$ is not Koszul. Indeed, by \cite[Example~5.1.4]{Avramov:1981}, $K^R$ has a nonzero triple Massey product, and so $K^R$ is not formal.
\end{example}

\begin{example}\label{e:koszul:nonexample_Katthan}
Let $Q=k\llbracket a,b,c,d,e\rrbracket$ and consider the quotient map
\[ 
\phi\colon Q\to R\coloneqq Q/(ab^2,cd^2,e^3,abcd,d^2e^2,b^2e^2,ace,b^2d^2e)\,.
\]
In \cite[Theorem~3.1]{Katthan:2017}, it is shown that $\H(K^R)$ is a trivial extension that admits a nonzero triple Massey product; the latter is an obstruction to the formality of $K^R$, while the former justifies that $\H(K^R)$ is a Koszul $k$-algebra.
\end{example}

\begin{remark}\label{r_polarisation}
To any monomial ideal $I\subseteq Q$ one may associate a square-free monomial ideal $I^\circ$ in a larger polynomial ring $Q^\circ$, known as the \emph{polarization} of $I$. Fr\"oberg \cite{Froeberg:1982} proved that there is a regular sequence of linear forms $y_1,\ldots,y_t$ in the quotient $R^\circ=Q^\circ/I^\circ$ such that $ R=R^\circ/(y_1,\ldots,y_t)$. From \cref{p_change_of_rings} it follows that $R$ is Cohen Koszul if and only if $R^\circ$ is Cohen Koszul.
\end{remark}

A \emph{simplicial complex} $\Delta$ on $[n]=\{1,\ldots,n\}$ is a nonempty family of subsets of $[n]$, closed under taking subsets. The \emph{Stanley--Reisner ring} associated to $\Delta$, denoted $k[\Delta]$, is the quotient of $k[x_1,\ldots,x_n]$ by the ideal generated by monomials $x_{i_1}\cdots x_{i_t}$ such that $\{i_1,\ldots,i_t\}\notin \Delta$. Every square free monomial ring is the Stanley--Reisner ring of some simplical complex, and so by \cref{r_polarisation} we may restrict to such monomials rings.

\begin{remark}\label{r_toric_top} We make some remarks about the connections to toric topology; for precise definitions and background on this area the reader may consult \cite{Buchstaber/Panov:2015}. 

To a simplical complex $\Delta$ on $[n]$ one also associates the moment angle complex $\mathcal{Z}_\Delta$, a finite CW-complex with an action of the torus $(S^1)^n$. The homotopy quotient $\mathcal{DJ}_\Delta=\mathcal{Z}_\Delta/\!/(S^1)^n$ is known as the Davis--Januszkiewicz space of $\Delta$. By \cite[Theorem~4.8]{Davis/Januszkiewicz:1991} and \cite[Theorem~4.8]{Notbhom/Ray:2005} the cochain algebra of this space is formal, and quasi-isomorphic to the Stanley--Reisner ring:
\[
C^*(\mathcal{DJ}_\Delta;k)\simeq k[\Delta]\,,
\]
where the variables $x_i$ are given cohomological degree $2$. From this it follows that $k[\Delta]$ is a Koszul $k$-algebra if and only if $\mathcal{DJ}_\Delta$ is a Koszul space in the sense of \cite{Berglund:2014}. As remarked in \cite[Example 5.8]{Berglund:2014}, this happens exactly when $k[\Delta]$ is a quadratic algebra, or equivalently if $\Delta$ is a \emph{flag complex}, that is, the minimal faces not belonging to $\Delta$ are all edges.

The question of when $\mathcal{Z}_\Delta$ is a Koszul space seems to be more interesting. By \cite[Lemma~3.1]{Buchstaber/Panov:2015} 
there is a quasi-isomorphism of dg $k$-algebras
\[
C^*(\mathcal{Z}_\Delta;k)\simeq k[\Delta]\lotimes_Qk\,.
\]
Thus $k[\Delta]$ is Cohen Koszul if and only if $\mathcal{Z}_\Delta$ is a Koszul space, since the latter means that $C^*(\mathcal{Z}_\Delta;k)$ is formal with Koszul homology algebra. The related condition that $\mathcal{Z}_\Delta$ is formal has been investigated in \cite{Denham/Suciu:2007,Limonchenko:2019}. 

The almost Golod condition is also connected with the minimally non-Golod condition for simplicial complexes introduced in \cite{Berglund/Joellenbeck:2007}. Indeed, the proof of \cite[Theorem~1.1]{Amelotte:2020} shows that if $M=\mathcal{Z}_\Delta$ is a moment angle manifold, and if $M\smallsetminus\{{\rm pt}\}$ is rationally homotopy equivalent to a wedge of spheres, then $\Delta$ is minimally non-Golod (over $\mathbb{Q}$).
\end{remark}

\section{Background on \texorpdfstring{$\ai$}{A-infinity}-algebras and coalgebras} \label{sec:ai}

Stasheff introduced $\ai$-algebras in topology to characterize loop spaces \cite{Stasheff:1963a,Stasheff:1963b}, and they have since proven a powerful tool in algebra as a flexible generalization of dg algebras; for an overview see \cite{Keller:2001}. In our context, the minimal $Q$-free resolution of a finite $Q$-algebra $R$ can be equipped with an $\ai$-algebra structure (see \cref{sec:transfer-ai}), and this will be leveraged to characterize Koszul homomorphisms in terms of presentations similar to the quadratic presentations for classical Koszul algebras (see \cref{sec:koszul-via-ai}). 

From now on $Q$ is always a local ring with maximal ideal $\fm_Q$ and residue field $k$, and unadorned tensor products and Hom sets are taken over $Q$. 

\begin{chunk} \label{ai-alg}
An \emph{$\ai$-algebra} is a graded $Q$-module $A$ equipped with $Q$-linear maps
\begin{equation*}
m_n \colon A^{\otimes n} \to A \quad\text{for } n \geqslant 1
\end{equation*}
of degree $(n-2)$ satisfying the \emph{Stasheff identities}
\begin{equation} \label{eq:stasheff-alg}
\sum_{\substack{r+s+t=n\\r,t \geqslant 0, s \geqslant 1}} (-1)^{r+st} m_{r+1+t} \left(\id^{\otimes r} \otimes m_s \otimes \id^{\otimes t}\right) = 0\,.
\end{equation}
Taking $n=1$ this says that $m_1$ is a degree $-1$ square zero endomorphism of $A$, so we can (and will) make $A$ a complex with $\partial=m_1$. Taking $n=2$ yields a product satisfying the Leibniz rule $\partial m_2=m_2(\partial\otimes \id+\id\otimes \partial)$. The next Stasheff identity, for $n=3$, can be interpreted as saying that $m_2$ is associative up to a homotopy given by $m_3$, that is, 
\begin{equation*}
m_2( \id \otimes m_2-m_2\otimes \id) = \partial m_3+m_3(\partial \otimes \id\otimes \id+\id \otimes \partial \otimes \id + \id \otimes \id\otimes\partial )\,.
\end{equation*}

If for some $n$ the Stasheff identity \cref{eq:stasheff-alg} holds for every integer less than $n$, then the obstruction
\begin{equation}\label{eq_obstruction}
\obs{A}{n} \coloneqq \sum_{\substack{r+s+t=n\\r,t \geqslant 0, n > s > 1}} (-1)^{r+st} m_{r+1+t} \left(\id^{\otimes r} \otimes m_s \otimes \id^{\otimes t}\right)
\end{equation}
is a chain map $A^{\otimes n}\to A$; see \cite[Corollaire~B.1.2]{LefevreHasegawa:2003}. 

A \emph{morphism of $\ai$-algebras} $\phi \colon A \to B$ consists of $Q$-linear maps
\begin{equation*}
\phi_n \colon A^{\otimes n} \to B \quad\text{for } n \geqslant 1
\end{equation*}
of degree $(n-1)$ satisfying
\begin{equation} \label{eq:stasheff-alg-mor}
\begin{aligned}
\sum_{\substack{r+s+t=n\\r,t \geqslant 0, s \geqslant 1}} (-1)^{r+st} \phi_{r+1+t} & \left(\id^{\otimes r} \otimes m^A_s \otimes \id^{\otimes t}\right) \\
&= \sum_{p=1}^n \sum_{\substack{\bmalpha \in \BN^p\\|\bmalpha|=n}} (-1)^{v(\bmalpha)} m^B_p (\phi_{\alpha_1} \otimes \cdots \otimes \phi_{\alpha_p})
\end{aligned}
\end{equation}
where $\bmalpha = (\alpha_1, \ldots, \alpha_p)$ and $|\bmalpha|=\sum_{k=1}^p \alpha_k$, with $v(\bmalpha) = \sum_{k=1}^p (p-k) (\alpha_k - 1)$. 

If for some $n$ the Stasheff identity \cref{eq:stasheff-alg-mor} holds for every integer less than $n$, then we define
\begin{equation*}
\begin{aligned}
\obs{\phi}{n} \coloneqq \sum_{\substack{r+s+t=n\\r,t \geqslant 0, s \geqslant 2}} (-1)^{r+st} \phi_{r+1+t} & \left(\id^{\otimes r} \otimes m^A_s \otimes \id^{\otimes t}\right) \\
&- \sum_{p=2}^{n} \sum_{\substack{\bmalpha \in \BN^p\\|\bmalpha|=n}} (-1)^{v(\bmalpha)} m^B_p (\phi_{\alpha_1} \otimes \cdots \otimes \phi_{\alpha_p})\,.
\end{aligned}
\end{equation*}
Then the Stasheff identity \cref{eq:stasheff-alg-mor} holds if and only if
\begin{equation*}
\obs{\phi}{n} = m_1 \phi_n + (-1)^n \phi_n( m_1 \otimes \id^{\otimes (n-1)} + \cdots + \id^{\otimes (n-1)} \otimes m_1)\,.
\end{equation*}

A morphism $\phi$ of $\ai$-algebras is a \emph{quasi-isomorphism} if the chain map $\phi_1$ is a quasi-isomorphism of complexes. The morphism $\phi$ is \emph{strict} if $\phi_n = 0$ for $n > 1$.  In this case \cref{eq:stasheff-alg-mor} simplifies to
\begin{equation} \label{eq:stasheff-alg-strict-mor}
\phi_1 m_n^A = m_n^B (\phi_1 \otimes \cdots \otimes \phi_1)\,.
\end{equation}

The composition of morphisms $\phi \colon A \to B$ and $\psi \colon B \to C$ is defined by
\begin{equation*}
(\psi \circ \phi)_n \coloneqq \sum_{p=1}^n \sum_{\substack{\bmalpha \in \BN^p\\|\bmalpha|=n}} (-1)^{v(\bmalpha)} \psi_p (\phi_{\alpha_1} \otimes \cdots \otimes \phi_{\alpha_p})\,.
\end{equation*}
\end{chunk}

\begin{chunk} \label{ai-alg-strictly-unital}
An $\ai$-algebra $A$ is \emph{strictly unital} if there exists $1_A \in A_0$ such that
\begin{equation} \label{eq:unital-alg}
\begin{gathered}
m_2(1_A \otimes a) = a = m_2(a \otimes 1_A) \quad\text{for all } a \in A \quad\text{and} \\
m_n(a_1 \otimes \cdots \otimes a_{i-1} \otimes 1_A \otimes a_{i+1} \otimes \cdots \otimes a_n) = 0 \quad\text{for all } 1 \leqslant i \leqslant n
\end{gathered}
\end{equation}
for any $a_1, \ldots, a_n \in A$ and $n > 2$. A morphism of strictly unital $\ai$-algebras $\phi \colon A \to B$ is a morphism of $\ai$-algebras such that
\begin{equation} \label{eq:unital-alg-mor}
\begin{gathered}
\phi_1(1_A) = 1_B \quad\text{and} \\
\phi_n(a_1 \otimes \cdots \otimes a_{i-1} \otimes 1_A \otimes a_{i+1} \otimes \cdots \otimes a_n) = 0 \quad\text{for all } 1 \leqslant i \leqslant n
\end{gathered}
\end{equation}
for any $a_j \in A$ and $n > 1$.
\end{chunk}

\begin{chunk} \label{ai-alg-split-unit}
An $\ai$-algebra is \emph{connective} if it is concentrated in non-negative degrees. If $A = Q \oplus \bar{A}$ is a graded module and $1_A$ a free generator of the direct summand $Q$, then $A$ is an \emph{$\ai$-algebra with a split unit}. 

A split unital $\ai$-algebra structure on a graded module concentrated in non-negative degrees is equivalent to the existence of $Q$-linear maps
\begin{equation*}
\bar{m}_n \colon \bar{A}^{\otimes n} \to \bar{A} \quad\text{for } n \geqslant 1
\end{equation*}
of degree $(n-2)$ and $Q$-linear maps
\begin{equation*}
h_1 \colon \bar{A} \to Q \quad \text{and} \quad h_2 \colon \bar{A}^{\otimes 2} \to Q
\end{equation*}
of degrees $-1$ and $0$, respectively, such that for $n \neq 2,3$ the Stasheff identities \cref{eq:stasheff-alg} hold when replacing $m_i$ by $\bar{m}_i$, and for $n=2,3$
\begin{equation*}
\sum_{\substack{r+s+t=n\\r,t \geqslant 0, s \geqslant 1}} (-1)^{r+st} \bar{m}_{r+1+t} (\id^{\otimes r} \otimes \bar{m}_s \otimes \id^{\otimes t}) + (h_{n-1} \otimes \id - \id \otimes h_{n-1}) = 0\,,
\end{equation*}
and additionally
\begin{equation*}
\begin{gathered}
h_1 \bar{m}_1 = 0 \,,\quad h_1 \bar{m}_2 - h_2 (\bar{m}_1 \otimes \id + \id \otimes \bar{m}_1) = 0 \quad \text{and} \\
h_1 \bar{m}_3 + h_2 (\bar{m}_2 \otimes \id - \id \otimes \bar{m}_2) = 0\,;
\end{gathered}
\end{equation*}
the former replaces the second and third Stasheff identity and the latter supplements the first three Stasheff identities. In particular, for $n > 2$ the maps $\bar{m}_n$ are the appropriate restrictions of $m_n$. For $n=2$, we obtain $m_2$ by $\bar{m}_2 + h_2$ and additionally enforcing \cref{eq:unital-alg}. For $n=1$, we have $m_1 = \bar{m}_1 + h_1$. This treatment is similar to \cite[Section~3]{Burke:2015}, but is slightly more general since we allow $\bar{A}_0\neq 0$ and hence need $h_2$ as well as $h_1$. If $A$ were not connective then we would also need maps $h_n$ for $n\geqslant 3$. Taken together the $h_n$ will correspond to the curvature term on the bar construction of $A$; see \cref{bar_construction}.
\end{chunk}

\begin{remark}
Let $A$ be a connective $\ai$-algebra with a split unit. Then the projection $A \to Q$ onto the free summand containing the unit need not be a morphism of strictly unital $\ai$-algebras. In fact, this happens if and only if $h_1 = 0$ and $h_2 = 0$. Such $\ai$-algebras are called \emph{augmented}. 
\end{remark}

\begin{chunk}
Fix a graded coalgebra $(C,\Delta)$. Recall $C$ is \emph{counital} if there exists a counit map $\epsilon \colon C \to Q$ such that
\begin{equation*}
(\id \otimes \epsilon) \Delta = \id = (\epsilon \otimes \id) \Delta\,.
\end{equation*}
We say $C$ is a \emph{curved dg coalgebra} it is equipped with a coderivation $\partial$ of degree $-1$ and a curvature term $h \colon C \to Q$ of degree $-2$ such that 
\begin{equation*}
\partial^2 = (h \otimes \id - \id \otimes h) \Delta \quad\text{and}\quad h \partial = 0\,.
\end{equation*}
A curved dg coalgebra $C$ is \emph{connected} if it is non-negatively graded, counital and $C_0 = Q$. In this setting, we write $C = Q \oplus \bar{C}$ for $\bar{C} = \ker(\epsilon)$ and set
\begin{equation*}
\begin{gathered}
\bar{\Delta} \coloneqq \left(\bar{C} \to C \xra{\Delta} C \otimes C \to \bar{C} \otimes \bar{C}\right)\,, \\
\bar{\partial} \coloneqq \left(\bar{C} \to C \xra{\partial} C \to \bar{C}\right) \quad \text{and} \quad \bar{h} \coloneqq \left(\bar{C} \to C \xra{h} Q\right)
\end{gathered}
\end{equation*}
for the restrictions to $\bar{C}$. These maps satisfy the same relations as $\Delta$, $\partial$ and $h$. 
\end{chunk}

\begin{chunk} \label{tensor-co-algebra} \label{tensor-algebra}
The tensor algebra $\talg{V}$ on a graded $Q$-module $V$ has underlying graded module $\tmod{V} \coloneqq \bigoplus_{n \geqslant 0} V^{\otimes n}$, and the multiplication
\begin{equation*}
\mu((v_1 \otimes \cdots \otimes v_k) \otimes (v'_1 \otimes \cdots \otimes v'_\ell)) \coloneqq v_1 \otimes \cdots \otimes v_k \otimes v'_1 \otimes \cdots \otimes v'_\ell\,.
\end{equation*}
The tensor algebra is bigraded by $\talg[n]{V}_i=(V^{\otimes n})_i$. 

The tensor coalgebra $\tcoa{V}$ on a graded $Q$-module $V$ has underlying graded module $\tmod{V}$, and the comultiplication
\begin{equation*}
\Delta(v_1 \otimes \cdots \otimes v_n) \coloneqq \sum_{i=0}^n (v_1 \otimes \cdots \otimes v_i) \otimes (v_{i+1} \otimes \cdots \otimes v_n)\,.
\end{equation*}
The tensor coalgebra is bigraded by $\tcoa[n]{V}_i=(V^{\otimes n})_i$. The data of an $\ai$-algebra can equivalently be encoded as a differential on a tensor coalgebra, as we see next.
\end{chunk}

\begin{chunk} \label{bar_construction}
Let $A$ be a split unital connective $\ai$-algebra. Then the tensor coalgebra $\tcoa{\susp \bar{A}}$ has an induced curved dg coalgebra structure. The curvature term has components
\begin{equation*}
h_1 \shift^{-1} \colon \tcoa[1]{\susp \bar{A}} \to Q \quad \text{and} \quad h_2 (\shift^{-1})^{\otimes 2} \colon \tcoa[2]{\susp \bar{A}} \to Q
\end{equation*}
and zero otherwise. The coderivation $\partial$ has components
\begin{equation*}
(-1)^{\frac{k(k+1)}{2}} \sum_{\substack{i+j=n-k\\i,j \geqslant 0}} (\id^{\otimes i} \otimes \shift \bar{m}_{k} (\shift^{-1})^{\otimes k} \otimes \id^{\otimes j}) \colon \tcoa[n]{\susp \bar{A}} \to \tcoa[n-k+1]{\susp \bar{A}}
\end{equation*}
for $k \geqslant 1$, and zero otherwise. The map $\partial$ is well-defined since $A$ is concentrated in non-negative homological degree. With this structure $\tcoa{\susp \bar{A}}$ is a connected curved dg coalgebra, and we call
\begin{equation*}
\bc[\star]{A}_\bullet \coloneqq \left(\tcoa[\star]{\susp \bar{A}}_\bullet,h,\partial,\Delta\right)
\end{equation*}
the \emph{bar construction of $A$}. For $a_1, \ldots, a_n \in \bar{A}$ we write
\begin{equation*}
[a_1 | \ldots | a_n] \coloneqq (\shift a_1 \otimes \cdots \otimes \shift a_n) \in \bc[n]{A}\,.
\end{equation*}
For a split unital connective $\ai$-algebra the canonical projection and inclusion induce a degree $-1$ map of graded modules 
\begin{equation*}
\bc{A} \surjar \susp \bar{A} \injar \susp A \to A\,.
\end{equation*}

Let $C$ be a connected curved dg coalgebra. Then the algebra $\talg{\susp^{-1} \bar{C}}$ has an induced dg algebra structure. The differential $m_1$ has components
\begin{equation*}
\begin{gathered}
-\bar{h} \shift \colon \talg[1]{\susp^{-1} \bar{C}} \to \talg[0]{\susp^{-1} \bar{C}} \,,\quad -\shift^{-1} \bar{\partial} \shift \colon \talg[1]{\susp^{-1} \bar{C}} \to \talg[1]{\susp^{-1} \bar{C}} \\
\text{and}\quad (\shift^{-1})^{\otimes 2} \bar{\Delta} \shift \colon \talg[1]{\susp^{-1} \bar{C}} \to \talg[2]{\susp^{-1} \bar{C}}\,;
\end{gathered}
\end{equation*}
and zero otherwise. With this structure $\talg{\susp^{-1} \bar{C}}$ is a split unital connective dg algebra, and we call
\begin{equation*}
\cobc[\star]{C}_\bullet \coloneqq \left(\talg[\star]{\susp^{-1} \bar{C}}_\bullet,m_1,m_2\right)
\end{equation*}
the \emph{cobar construction of $C$}. For $c_1, \ldots, c_n \in \bar{C}$ we write
\begin{equation*}
\langle c_1 | \ldots | c_n \rangle \coloneqq (\shift^{-1} c_1 \otimes \cdots \otimes \shift^{-1} c_n) \in \cobc[n]{C}\,.
\end{equation*}
For a connected curved dg coalgebra the canonical inclusion and projection maps induces a degree $-1$ map of graded modules
\begin{equation*}
C \to \susp^{-1} C \surjar \susp^{-1} \bar{C} \injar \cobc{C}\,.
\end{equation*}
\end{chunk}

\begin{remark}
The bar and cobar constructions define an adjoint pair of functors when restricted to split unital connective dg algebras and connected curved dg coalgebras; see \cite[Section~3]{Lyubashenko:2013}. It remains an adjunction when restricted to augmented connective dg algebras and connected dg coalgebras.
\end{remark}

\begin{chunk} \label{morphism_coalgebra}
A morphism $\phi \colon (C,\Delta,\epsilon,\partial,h) \to (C',\Delta',\epsilon',\partial',h')$ of connected curved dg coalgebras consists of $Q$-linear maps
\begin{equation*}
\phi_0 \colon C \to Q \quad \text{and} \quad \phi_1 \colon C \to C'
\end{equation*}
of degree $-1$ and 0, respectively, satisfying
\begin{equation*}
\begin{gathered}
\epsilon = \epsilon' \phi_1 \,,\quad h' \phi_1 = h - \phi_0 \partial + (\phi_0 \otimes \phi_0) \Delta \,, \\
\partial' \phi_1 = \phi_1 \partial + (\phi_0 \otimes \phi_1 - \phi_1 \otimes \phi_0) \Delta \quad \text{and} \quad \Delta' \phi_1 = (\phi_1 \otimes \phi_1) \Delta\,;
\end{gathered}
\end{equation*}
see \cite[Chapter~4]{Posiselski:2011}. This induces a map of dg algebras $\cobc{\phi} \colon \cobc{C} \to \cobc{C'}$, and we say $\phi$ is a \emph{weak equivalence} if $\cobc{\phi}$ is a quasi-isomorphism.
\end{chunk}

\begin{chunk}
\label{ai-module}
Let $A$ be an $\ai$-algebra. An \emph{$\ai$-module over $A$} is a graded module $M$ equipped with maps
\begin{equation*}
m_n^M \colon A^{\otimes (n-1)} \otimes M \to M \quad\text{for } n \geqslant 1
\end{equation*}
of degree $(n-2)$, satisfying
\begin{equation*} \label{eq:stasheff-mod}
\sum_{\substack{r+s+t=n\\r \geqslant 0, s,t \geqslant 1}} (-1)^{r+st} m^M_{r+1+t} \left(\id^{\otimes r} \otimes m_s \otimes \id^{\otimes t}\right) + \sum_{\substack{r+s=n\\r \geqslant 0, s \geqslant 1}} (-1)^r m^M_{r+1} \left(\id^{\otimes r} \otimes m_s^M\right) = 0\,.
\end{equation*}

If $A$ is strictly unital, we say an $\ai$-module $M$ over $A$ is \emph{strictly unital} if
\begin{equation*}
\begin{gathered}
m_2(1_A \otimes m) = m \quad\text{for all } m \in M \quad\text{and} \\
m_n(a_1 \otimes \cdots \otimes a_{i-1} \otimes 1_A \otimes a_{i+1} \otimes \cdots \otimes a_{n-1} \otimes m) = 0 \quad \text{for all } 1 \leqslant i \leqslant n-1
\end{gathered}
\end{equation*}
for any $a_1, \ldots, a_{n-1} \in A$ and $m \in M$ with $n \neq 2$. 

If $A$ is connective and has a split unit, then a strictly unital $\ai$-module structure over $A$ on $M$ is equivalent to the existence of maps 
\begin{equation*}
\bar{m}^M_n \colon \bar{A}^{\otimes (n-1)} \otimes M \to M \quad\text{for } n \geqslant 1
\end{equation*}
of degree $(n-2)$ such that for $n \neq 2,3$ the Stasheff identities hold when replacing $m$ by $\bar{m}$, and for $n=2,3$ there is an extra curvature term $h_{n-1} \otimes \id$ similar to \cref{ai-alg-split-unit}.
\end{chunk}

\begin{chunk} \label{ai-mod-dg-mod-cobar-bar}
Let $A$ be a split unital connective $\ai$-algebra. The data of a strictly unital $\ai$-module structure over $A$ is equivalent to that of a strictly unital dg module structure over $\cobc{\bc{A}}$. Explicitly, if $\{\bar{m}^M_n\}$ is a strictly unital $\ai$-module structure on a graded module $M$, then the dg module structure on $M$ is given by the same differential $m^M_1$, and the multiplication $\cobc{\bc{A}} \otimes M \to M$ induced by
\begin{equation*}
-(-1)^{\frac{n(n-1)}{2}} m^M_{n+1} ((\shift^{-1})^{\otimes n} \shift \otimes \id_M) \colon \susp^{-1} \rbc[n]{A} \otimes M \to M\,.
\end{equation*}
Moreover, this construction is natural in $A$ and $M$, and any quasi-isomorphism of $\ai$-modules over $A$ yields a quasi-isomorphism of dg modules over $\cobc{\bc{A}}$. 
\end{chunk}

\section{Transfer of \texorpdfstring{$\ai$}{A-infinity}-algebra structures} \label{sec:transfer-ai}

In this section, as above, $Q$ is a local ring with maximal ideal $\fm_Q$ and residue field $k$. Let $R$ be an $\ai$-algebra over $Q$, and let $A \to R$ be a quasi-isomorphism of complexes over $Q$. We would like to know whether the $\ai$-algebra structure on $R$ induces an $\ai$-algebra structure on $A$. This is well-understood in the case that $Q$ is field, so that $A \to R$ is a homotopy equivalence; the first result is due to Kadeishvili \cite{Kadeishvili:1982} when $A = \H(R)$. For general homotopy equivalences this was studied, for example, in \cite{Markl:2006}. Burke has shown that if $R$ is a quotient of $Q$ and $A$ is a $Q$-free resolution of $R$, then the product on $R$ lifts to an $\ai$-structure on $A$ \cite[Proposition~3.6]{Burke:2018}. We give a proof in a more general situation.

\begin{proposition} \label{ai-along-qi}
Let $R$ be a strictly unital connective $\ai$-algebra and $\epsilon \colon A \to R$ a surjective quasi-isomorphism of complexes over $Q$, with $A$ degree-wise free and concentrated in non-negative degrees. Then there exists an $\ai$-algebra structure with a split unit on $A$ such that $\epsilon$ is a strict quasi-isomorphism of $\ai$-algebras.
\end{proposition}
\begin{proof}
Since $\epsilon$ is surjective we may choose a splitting $A = \bar{A} \oplus Q$ such that $\epsilon$ maps the free generator of $Q$ to the unit of $R$. We inductively construct higher multiplication maps $m_n$ on $A$ satisfying the $n$th Stasheff identity. To begin with we set $m_1 \coloneqq \partial$ where $\partial$ is the differential of $A$. 

For $n=2$ we consider the commutative diagram
\begin{equation*}
\begin{tikzcd}
A \otimes Q + Q \otimes A \ar[rr] \ar[d] \&\& A \ar[d,"\simeq","\epsilon" swap] \\
A^{\otimes 2} \ar[r,"\epsilon^{\otimes 2}" swap] \ar[urr,dashed,"m_2^A"] \& R^{\otimes 2} \ar[r,"m_2^R" swap] \& R \nospacepunct{.}
\end{tikzcd}
\end{equation*}
The morphism  of complexes $m_2^A \colon A^{\otimes 2} \to A$ exists because $\epsilon$ is a surjective quasi-isomorphism and the left vertical arrow is injective in each degree and the cokernel in each degree is projective; see for example \cite[Section~7]{Dwyer/Spalinski:1995}. The morphism $m_2^A$ satisfies the desired properties by construction.

For $n > 2$, the obstruction $\obs{A}{n}$ from \cref{eq_obstruction} is a chain map. We have a short exact sequence of complexes
\begin{equation*}
0 \to \sum_{i+j=n-1} A^{\otimes i} \otimes Q \otimes A^{\otimes j} \xrightarrow{\eta_n} A^{\otimes n} \to \bar{A}^{\otimes n} \to 0\,.
\end{equation*}
By direct computation we obtain $\obs{A}{n} \eta_n = 0$. So the obstruction $\obs{A}{n}$ factors through $\bar{A}^{\otimes n}$. We consider the diagram 
\begin{equation*}
\begin{tikzcd}
\& \bar{A}^{\otimes n} \ar[d,"{\obs{A}{n}}" description] \ar[dr,"{0}"] \ar[dl,dashed,"{\alpha}" swap] \\
\susp^{-1} \cone(\epsilon) \ar[r,"{\pi}" swap] \& A \ar[r,"{\simeq}","{\epsilon}" swap] \& R \nospacepunct{.}
\end{tikzcd}
\end{equation*}
Since $\bar{A}^{\otimes n}$ is, as graded modules, a direct summand of $A^{\otimes n}$, and the higher multiplications $m_i^A$ for $i < n$ commute with $\epsilon$, the right triangle commutes up to the homotopy $m_n^R \epsilon^{\otimes n}$. Then there exists a chain map $\alpha$ such that the left triangle commutes up to a homotopy $\sigma$. Since $\epsilon$ is surjective, we may assume $m_n^R \epsilon^{\otimes n} = \epsilon \sigma$ by \cite[Proposition~1.3.1]{Avramov:1998}. That is 
\begin{equation*}
m^A_n \coloneqq \left(A^{\otimes n} \to \bar{A}^{\otimes n} \xrightarrow{\sigma} A\right)
\end{equation*}
satisfies the $n$th Stasheff identity \cref{eq:stasheff-alg-strict-mor}. 
\end{proof}

\begin{chunk}\label{c_ai_module}
In the setup of \cref{ai-along-qi} we can also transfer $\ai$-module structures: If $M$ is a strictly unital $\ai$-module over $R$ and with semifree resolution $\gamma\colon G \to M$ over $Q$, in the sense discussed later in \cref{c_semifree}, then there exists a strictly unital $\ai$-module structure on $G$ over $A$ and $\gamma$ is a strict morphism of $\ai$-modules; compare with \cite{Burke:2018}. When the homology of $M$ is bounded below (for example, if $M$ is an honest module), one can take $G$ to be a bounded below complex of free $Q$-modules.
\end{chunk}

\begin{proposition}\label{p_Ainfinity_lifting}
Let $\epsilon \colon R \to S$ be a surjective strict quasi-isomorphism of strictly unital $\ai$-algebras over $Q$. Further let $A$ be a split unital, connective, degree-wise free $\ai$-algebra and $\phi \colon A \to S$ a morphism of strictly unital $\ai$-algebras. Then there exists a morphism of strictly unital $\ai$-algebras $\psi \colon A \to R$ such that $\phi = \epsilon \psi$. 
\end{proposition}
\begin{proof}
The unit $Q\to S$ factors through $\epsilon$ and $\phi_1$, so by \cite[Section~7]{Dwyer/Spalinski:1995}, there is a chain map $\psi_1 \colon A \to R$ such that $\phi_1 = \epsilon \psi_1$ and $\psi_1(1_A) = 1_R$. 

Let $n \geqslant 2$ and assume that for $i < n$ the chain maps $\psi_i \colon A^{\otimes i} \to R$ exist, the $i$th Stasheff identities \cref{eq:stasheff-alg-mor} and \cref{eq:unital-alg-mor} hold, and $\epsilon \psi_i = \phi_i$. A computation shows that $\obs{\psi}{n}$ and $\obs{\phi}{n}$ vanish when any of its inputs is $1_A$. Hence we can view $\obs{\psi}{n}$ and $\obs{\phi}{n}$ as maps on $\bar{A}^{\otimes n}$. Taking homology classes in $\hom{\bar{A}^{\otimes n}}{S}$ we have
\begin{equation*}
\epsilon [\obs{\psi}{n}] = [\obs{\phi}{n}]= 0\,,
\end{equation*}
since $\epsilon \obs{\psi}{n} = \obs{\phi}{n}$ and $\phi$ is a morphism of strictly unital $\ai$-algebras; cf.\@ \cref{ai-alg}.
Since $\epsilon$ is a surjective quasi-isomorphism, and using the assumptions on $A$, the induced map $\hom{\bar{A}^{\otimes n}}{R}\to \hom{\bar{A}^{\otimes n}}{S}$ is a quasi-isomorphism and hence $\obs{\psi}{n}$ is a boundary in $\hom{\bar{A}^{\otimes n}}{R}$. That is, there is $\bar{\psi}_n \colon \bar{A}^{\otimes n}\to R$ such that 
\[
\obs{\psi}{n}=m_1^R\bar{\psi}_n+(-1)^n\bar{\psi}_n(\bar{m}_1^A \otimes \id^{\otimes n} + \cdots + \id^{\otimes n} \otimes \bar{m}_1^A)\,.
\]
Setting $\psi_n \coloneqq (A^{\otimes n} \to \bar{A}^{\otimes n} \xrightarrow{\bar{\psi}_n} R)$, we now have $\psi_1,\ldots, \psi_n$ satisfying the required identities \cref{eq:stasheff-alg-mor} and \cref{eq:unital-alg-mor}, completing the induction.
\end{proof}

It is well-known that $\ai$-algebras can be used to characterize formality of dg algebras over fields \cite{Kadeishvili:1982}. We record the following generalization in local algebra.

\begin{proposition}\label{p_formality}
Let $\phi\colon Q\to R$ be a finite local homomorphism and let $\epsilon\colon A \to R$ be the minimal $Q$-free resolution, equipped with an $\ai$-structure making $\epsilon$ a strict quasi-isomorphism of $\ai$-algebras. The $\ai$-algebra $A\otimes_Qk$ is quasi-isomorphic, as an $\ai$-algebra, to the derived fiber $R\lotimes_Q k$ defined as a dg $k$-algebra in \cref{sec:Koszul}. Moreover, $R\lotimes_Q k$ is formal as a dg $k$-algebra if and only if $A$ admits an $\ai$-structure $\{m_n\}$ as above that also satisfies $m_n\otimes_Qk =0$ for $n\geqslant 3$. 
\end{proposition}

\begin{proof}
Suppose that the minimal $Q$-free resolution $A$ of $R$ has an $\ai$-structure $\{m_n\}$ with the stated property. If $A'$ is a $Q$-free dg algebra resolution of $R$, then $A$ and $A'$ are quasi-isomorphic as $\ai$-algebras over $Q$ by \cref{p_Ainfinity_lifting}. Therefore $A\otimes_Qk$ and $R\lotimes_Qk=A'\otimes_Qk$ are quasi-isomorphic as $\ai$-algebras over $k$.

If $m_n\otimes_Qk =0$ for $n\geqslant 3$ then $A\otimes_Qk$ is a graded algebra, canonically isomorphic to $\Tor QRk$. Two dg $k$-algebras are quasi-isomorphic as dg algebras if and only if they are quasi-isomorphic as $\ai$-algebras \cite{Kadeishvili:1982}, and we can conclude that $R\lotimes_Qk$ is formal.

Suppose conversely that $R\lotimes_Q k$ is formal. By \cref{ai-along-qi} the minimal $Q$-free resolution $A$ of $R$ admits an $\ai$-structure $\{m_n'\}$. Using the same reasoning as above, since $R\lotimes_Q k$ is formal $A\otimes_Qk$ and $\Tor QRk$ are quasi-isomorphic as $\ai$-algebras over $k$. By the uniqueness of minimal models (that is, $\ai$-algebras over a field having zero differential; see \cite{Kadeishvili:1982}) there is an isomorphism of $\ai$-algebras 
\[
\psi\colon (\Tor QRk, 0, \mu, 0, \ldots) \xra{\cong} (A \otimes_Q k, 0, m_2' \otimes k, m_3' \otimes k, \ldots)\,,
\]
where $\mu$ is the ordinary product on $\Tor QRk$. We may make the identification $A \otimes_Q k=\Tor QRk$ and choose lifts $\Psi_i\colon A^{\otimes i} \to A$ with $\Psi_i\otimes_Qk=\psi_i$. By Nakayama's lemma $\Psi_1$ is an isomorphism and we can inductively define operations $m_n\colon A^{\otimes n} \to A$ by the formula $m_n\coloneqq $
\[
\Psi_1^{-1}\Big(-\sum_{\substack{r+s+t=n\\r,t \geqslant 0, s \geqslant 1}} (-1)^{r+st} {\Psi}_{r+1+t} \left(\id^{\otimes r} \otimes {m}_s \otimes \id^{\otimes t}\right) + \sum_{\substack{p,\bmalpha \in \BN^p\\|\bmalpha|=n}} (-1)^{v(\bmalpha)} {m}'_p {\bf\Psi}^{\otimes \bmalpha}\Big)\,.
\]
By construction the map $\Psi\colon (A,\{m_n'\}) \to (A,\{m_n\})$ now satisfies the Stasheff morphism identities \eqref{eq:stasheff-alg-mor}, and it follows that $(A,\{m_n\})$ is an $\ai$-algebra, isomorphic to $(A,\{m_n'\})$. Finally, from $\Psi\otimes_Q k=\psi$ it follows that $m_2\otimes_Qk=\mu$ and $m_n\otimes_Qk =0$ for $n\geqslant 3$, as stated in the proposition.
\end{proof}

The following technical lemma will be used later to help generate examples, by showing that certain $\ai$-operations are minimal.

\begin{lemma}\label{l_minimal_A_infty_map}
Let $\phi\colon A\to T$ be a map of connective split unital $\ai$-algebras, where $T$ is a trivial algebra. If for some $N$ the map $(\phi_1)_{< N} \colon A_{< N}\to T_{< N}$ is injective, then the $\ai$-structure of $A$ vanishes in degrees less than $N$, in the sense that $(\bar{m}_n(\bar{A}^{\otimes n}))_i=0$ for all $n\geqslant 1$ and $i<N$.
\end{lemma}

\begin{proof}
We prove this by induction on $n$. It is clear for $n=1$ since $\phi_1$ is a chain map. For $n\geqslant 2$, since $\bar{m}_s^{T}=0$ for all $s$, we can rearrange the Stasheff morphism identities \eqref{eq:stasheff-alg-mor}: 
\[
\bar{\phi}_1\bar{m}_n=-\sum_{\substack{r+s+t=n\\r,t \geqslant 0, s \geqslant 1}} (-1)^{r+st} \bar{\phi}_{r+1+t} \left(\id^{\otimes r} \otimes \bar{m}_s \otimes \id^{\otimes t}\right)\,.
\] 
We can assume by induction that $(\bar{m}_s(\bar{A}^{\otimes s}))_{<N}=0$ for $s<n$. Since each $\phi_r$ increases degree by $r-1$, this implies that the the right-hand side above is zero in degrees $i<N$. Since $(\bar{\phi}_1)_{<N}$ is injective, it follows that $\bar{m}_n(\bar{A}^{\otimes n}))_{<N}=0$.
\end{proof}

\subsection{Cyclic \texorpdfstring{$\ai$}{A-infinity}-algebras} \label{cyclic_ai_algebra}

For Gorenstein algebras the minimal resolution satisfies a Poincar\'e duality property that allows us, in favorable situations, to construct $\ai$-resolutions with additional duality properties.

A \emph{cyclic $\ai$-algebra} of degree $d$ over $Q$ is a complex $A$ of finitely generated free $Q$-modules with a perfect, $Q$-bilinear pairing
\begin{equation*}
\label{e_pairing}
\langle-,-\rangle\colon A\otimes A \to \susp^d Q\,,
\end{equation*}
and an $\ai$-structure $\{m_n\}$ on $A$ such that for each $n$
\[
\langle m_n(a_1,\ldots ,a_n),a_{n+1}\rangle = (-1)^{n+|a_1|(|a_2|+\cdots +|a_{n+1}|)}\langle m_n(a_{2},\ldots ,a_{n+1}),a_{1}\rangle\,;
\]
see \cite{Kontsevich:1994}.

There is for each $n$ an isomorphism of complexes
\[
\operatorname{cyc}\colon \Hom{}{A^{\otimes n}}{A}\xra{\cong} \Hom{}{A^{\otimes (n+1)}}{\susp^dQ}\,, \quad \operatorname{cyc}(f)= \langle f(-),-\rangle\,.
\]
We give $A^{\otimes (n+1)}$ the action of the cyclic group $C_{n+1}=\langle c\rangle $ with generator acting by $c\cdot (a_1\otimes \cdots \otimes a_{n+1}) = (-1)^{|a_1|(|a_2|+\cdots +|a_{n+1}|)} (a_2\otimes \cdots \otimes a_{n+1}\otimes a_1)$. From this perspective, an $\ai$-structure $\{m_n\}$ is cyclic if and only if 
\[
\operatorname{cyc}(m_n)\cdot c = (-1)^n\operatorname{cyc}(m_n) \quad\text{for all }n\,.
\]

Let $\phi\colon Q\to R$ be a surjective local Gorenstein homomorphism of projective dimension $d$, and let $A$ be the minimal resolution of $R$ over $Q$. Let $\mu\colon A^{\otimes 2}\to A$ be a chain map lifting the product on $R$; we can assume that $\mu$ is unital and graded-commutative by \cite[3.4.3]{Bruns/Herzog:1998}. The Gorenstein condition \cref{eq_gor_def} guarantees that $A\otimes_Q k=\Tor QRk$ is a Poincar\'e duality algebra with the product induced from $\mu$; see \cite[Theorem]{Avramov/Golod:1971}. It follows from Nakayama's lemma that $A_d\cong Q$ and we obtain a perfect pairing
\begin{equation}\label{eq_gorenstein_pairing}
    \langle-,-\rangle\colon A\otimes A \xrightarrow{\mu} A\twoheadrightarrow \susp^d A_d=\susp^d Q\,.
\end{equation}

\begin{theorem}\label{p_cyclic_ai}
Let $Q\to R$ be a surjective local Gorenstein homomorphism of odd projective dimension $d$. Assume that $Q$ contains a field of characteristic zero. The minimal resolution $A$ of $R$ over $Q$ admits the structure of a split unital, cyclic $\ai$-algebra of degree $d$, making the map $A\to R$ a strict $\ai$-algebra quasi-isomorphism.
\end{theorem}

We first need a lemma about projective resolutions.

\begin{lemma}\label{l_V_homology}
Let $M$ be a finitely generated $Q$-module of projective dimension $d>0$, with a projective resolution $A\to M$, and set $V=A_{<d}/A_0$. Then for any $n$ we have $\H_i( V^{\otimes n+1})=0$ whenever $i>n(d-1)+1$ and $i\neq (n+1)(d-1)$. 
\end{lemma}

\begin{proof}
We show this by inducing on $n$. Since $A$ is a projective resolution of $R$, the homology of $V$ is concentrated in degrees $1$ and $d-1$, and there is an exact triangle
\[
\susp^{d-1} Q^s\longrightarrow V\longrightarrow \susp N\,,
\]
of complexes of $Q$-modules, where $N=\ker(A_0\to M)$ and $A_d=Q^{\oplus s}$. This justifies the case $n=0$, and for each $n\geqslant 1$ yields another exact triangle
\[
\susp^{d-1} (V^{\otimes n})^{\oplus s}\longrightarrow V^{\otimes n+1}\longrightarrow \susp V^{\otimes n}\otimes N\,.
\]
Clearly $\H_i(\susp V^{\otimes n} \otimes N) = 0$ for $i > n(d-1)+1$, so by the long exact sequence in homology the map $\H_{i}(\susp^{d-1} V^{\otimes (n-1)})^{\oplus s} \to \H_{i}(V^{\otimes n})$ is surjective for $i > n(d-1)+1$. 
By the induction hypothesis $\H_i(\susp^{d-1} V^{\otimes (n-1)})=0$ if $ i\neq n(d-1)$ and $i>(n-1)(d-1)+1$. From this we conclude that the lemma holds for $n$.
\end{proof}

\begin{proof}[Proof of \cref{p_cyclic_ai}]
Recall from \cref{eq_gorenstein_pairing} that the pairing on $A$ was defined from a unital and graded-commutative product $\mu\colon A^{\otimes 2}\to A$. This restricts to a perfect pairing on $V=A_{<d}/A_0$, and we start by constructing operations $m_n^V\colon V^{\otimes n}\to V$.

 If $|a|+|b|=d+1$ then $\mu(a\otimes b)=0$ in $A$, so
\[
 \mu(\partial(a)\otimes b)+(-1)^{|a|}\mu(a\otimes \partial(b)) =\partial (\mu(a\otimes b))=0\,.
\]
Using graded-commutativity of $\mu$, this is equivalent to the cyclic identity
\[
\langle m_1^V(a),b\rangle =(-1)^{1+|a||b|} \langle m_1^V(b),a\rangle \,,
\]
where $m_1^V\coloneqq\partial$. 

Next, we truncate $\mu$ to obtain $\mu^V\colon V^{\otimes 2}\to V$, and we define $m_2^V$ by symmetrizing $\mu^V$ with respect to the $C_3$-action:
\[
\operatorname{cyc}(m_2^V)\coloneqq\operatorname{cyc}(\mu^V)\cdot \frac{1}{3}(1+ c+ c^2)\,.
\]
The obtained $m_2^V$ satisfies the required cyclic property by construction. However, the Stasheff identity \cref{eq:stasheff-alg} does not hold for $n=2$, and instead
\begin{equation}\label{eq_SI2_twisted}
    \partial(m^V_2(a,b))- m^V_2(\partial(a),b)-(-1)^{|a|}m^V_2(a,\partial(b)) = \langle a, b\rangle \partial(\omega) \,,
\end{equation}
where $\omega\in A_d$ is the generator with $\langle \omega ,1\rangle =1$. Nonetheless, since $\langle -,-\rangle$ is a chain map, the same computation as in \cref{eq_obstruction} shows that the obstruction $\obs{V}{3}$ is a chain map, that is, a cycle in $\Hom{}{V^{\otimes 3}}{V}$.

We proceed to construct  $m_n^V$ for $n\geqslant 3$ by induction, satisfying the Stasheff identities \cref{eq:stasheff-alg} for $n\geqslant 3$, and all satisfying $\operatorname{cyc}(m_n^V)\cdot c=(-1)^n\operatorname{cyc}(m_n^V)$. The argument is similar to the proof of \cref{ai-along-qi}. If $m_i^V$ have been constructed for $i<n$ with  required cyclic symmetry, a computation shows that the obstruction $\obs{V}{n}$ from \cref{eq_obstruction} is cyclic as well:
\[
\operatorname{cyc}(\obs{V}{n})=(-1)^n\operatorname{cyc}(\obs{V}{n})\cdot c\,.
\]
Since $\Hom{}{V^{\otimes n}}{V}\cong \susp^{-nd}V^{\otimes (n+1)}$ we can use \cref{l_V_homology} with $M=R$ to conclude that
\[
\H_i\big(\Hom{}{V^{\otimes n}}{V}\big)=0\ \text{ for } i>1-n\text{ and } i\neq  d-n-1\,.
\]
Since $d$ is odd, it is impossible to have $|\obs{V}{n}|=n-3=d-n-1$, hence the complex $\Hom{}{V^{\otimes n}}{V}$ is acyclic in degree $n-3$, and the class $[\obs{V}{n}]$ vanishes. This shows that there is an operation $\tilde{m}^V_n$ in $\Hom{}{V^{\otimes n}}{V}_{n-2}$ such that $\partial(\tilde{m}^V_n)=\obs{V}{n}$, and we symmetrize this to define $m_n$:
\[
\operatorname{cyc}(m_n^V)\coloneqq\operatorname{cyc}(\tilde{m}^V_n)\cdot \textstyle{\sum\limits_{i=0}^n \frac{(-1)^{in} c^i}{n+1}}\,.
\]
By construction $m_n^V$ has the required cyclic symmetry. We note that
\[
\partial(\operatorname{cyc}(m_n^V))= \operatorname{cyc}(\obs{V}{n})\cdot \textstyle{\sum\limits_{i=0}^n \frac{(-1)^{in} c^i}{n+1}} =\operatorname{cyc}(\obs{V}{n})\,.
\]
Therefore $\partial(m_n^V)=\obs{V}{n}$ and the operations $\{m_n^V\}$ satisfy the $n$th Stasheff identity. This concludes the induction.

To finish the proof we define the following operations on $A$:
\[
m_2(a_1,a_2) \coloneqq \begin{cases}
 m^V_2(a_1,a_2)& \text{if } |a_1|,|a_2|>0 \text{ and }|a_1|+|a_2|<d,\\
 \langle a_1,a_2\rangle \omega & \text{if }|a_1|+|a_2|=d,\\
 a_1a_2& \text{if }|a_1|=0 \text{ or }|a_2|=0,
\end{cases}
\]
and for $n \geqslant 3$
\[
m_n(a_1,\ldots,a_n) \coloneqq \begin{cases}
 m^V_n(a_1,\ldots,a_n)& \text{if all } |a_i|>0 \text{ and }|a_1|+\cdots +|a_2|<d,\\
 0& \text{if }|a_1|+\cdots+|a_2|=d \text{ or any }|a_i|=0.
\end{cases}
\]
The $n=2$ Stasheff identity for $A$ is equivalent to the identity \cref{eq_SI2_twisted} above.

To verify the $n$th Stasheff identity, with $n\geqslant 3$, we need to divide into cases depending on the inputs: when any of the inputs have degree zero; when the output has degree less than $d$; and when the output has degree $d$. The first two of these cases follow easily from the Stasheff identities for $\{m_n^V\}$. To check the third case, we suppose that $|a_1|+\cdots +|a_n|+n-3=d$, and we compute
\begin{align*}
    &\sum_{r+s+t=n} (-1)^{r+st} m_{r+1+t} \left(\id^{\otimes r} \otimes m_s \otimes \id^{\otimes t}\right)(a_1,\ldots,a_n)=\\
    & (-1)^{|a_1|(n-1)+1}\langle a_1 ,m^V_{n-1}(a_2,\ldots ,a_n)\rangle + (-1)^{n-1}\langle m^V_{n-1}(a_1,\ldots ,a_{n-1}),a_n\rangle\,;
\end{align*}
and this vanishes by the cyclic symmetry condition for $m_{n-1}^V$.

It follows that $A$ is an $\ai$-algebra with the operations $\{m_n\}$. Finally, the cyclic symmetry condition on $\{m_n^V\}$ implies $A$ is a cyclic $\ai$-algebra.
\end{proof}

\begin{remark}
The construction in the proof yields a bijection between unital cyclic $\ai$-algebra structures on $A$ and nonunital cyclic $\ai$-algebra structures on $V$, but with a modified version of the second Stasheff identity in the latter case. The cyclic condition is necessary to make this correspondence work.
\end{remark}

\begin{remark}\label{rem_berglund}
In work to appear, Alexander Berglund constructs cyclic $\ai$-algebra structures in significantly more generality than \cref{p_cyclic_ai}. In particular, his argument shows that the restriction to odd $d$ is unnecessary. We note the case $d \equiv 2 \text{ modulo } 4$ can obtained by a more careful analysis of the proof of \cref{p_cyclic_ai}, but the general case seems to require more machinery.
\end{remark}

\section{Twisted tensor products} \label{sec:twisted-tensor}

Twisted tensor products are an important tool in homological algebra, especially in the construction of resolutions. In this section we develop their theory over a commutative ring, producing universal resolutions via the resolution of the diagonal that will be applied to Koszul homomorphisms in later sections. However, our method of construction is new even when the base ring $Q$ is a field. Similar results have been obtained using different methods in unpublished work of Burke \cite{Burke:notes}. We do not explicitly use the language of twisting cochains, but these objects are present implicitly; and the reader may consult \cite[Section~2.1]{Loday/Vallette:2012} for more information on twisted tensor products and twisting cochains.

Let $C$ be a connected curved differential graded (dg) coalgebra over $Q$ that is free as a graded module. For a right dg module $M$ and a left dg module $N$ over $\cobc{C}$, we construct a complex $M \totimes C \totimes N$, with a ``twisted'' differential; we call this complex a \emph{twisted tensor product}. First, we define a dg bimodule $\cobc{C} \totimes C \totimes \cobc{C}$ over $\cobc{C}$.

\begin{construction} \label{twisted-tensor-product}
Ignoring differentials for now, there is a well-known exact sequence of graded $\cobc{C}$-bimodules 
\begin{equation} \label{eq:ses-tensoralg}
0 \to \cobc{C} \otimes \susp^{-1} \bar{C} \otimes \cobc{C} \xra{\iota} \cobc{C} \otimes Q \otimes \cobc{C} \xra{\mu} \cobc{C} \to 0\,,
\end{equation}
where $\iota(x \otimes \langle c \rangle \otimes y)= \mu(x,\langle c \rangle) \otimes 1 \otimes y - x \otimes 1 \otimes \mu(\langle c \rangle,y)$ for $c \in \bar{C}$ and $x,y \in \cobc{C}$, and $\mu$ the multiplication map. 

We give $\cobc{C} \otimes \susp^{-1} \bar{C} \otimes \cobc{C}$ the unique differential $\partial^\iota$ making $\iota$ a chain map, and we set
\[
\cobc{C} \totimes C \totimes \cobc{C} \coloneqq \cone(\cobc{C} \otimes \susp^{-1} \bar{C} \otimes \cobc{C} \xra{\iota} \cobc{C} \otimes Q \otimes \cobc{C})\,.
\]
We write $\partial^\tau$ for the differential on $\cobc{C} \totimes C \totimes \cobc{C}$, this is a dg $\cobc{C}$-bimodule whose underlying graded bimodule is $\cobc{C} \otimes C \otimes \cobc{C}$, using the evident multiplication by $\cobc{C}$ on either side.

For a right dg $\cobc{C}$-module $F$ and a left dg $\cobc{C}$-module $G$, we define
\begin{equation*}
F \totimes C \totimes G \coloneqq F \otimes_{\cobc{C}} \cobc{C} \totimes C \totimes \cobc{C} \otimes_{\cobc{C}} G\,.
\end{equation*}
Explicitly its underlying graded module is $F \otimes C \otimes G$ and the differential is
\begin{equation*}
\begin{aligned}
\partial^\tau &= \partial^F \otimes \id_C \otimes \id_G + \id_F \otimes \partial^C \otimes \id_G + \id_F \otimes \id_C \otimes \partial^G \\
&\quad + \left(\mu (\id_F \otimes \shift^{-1} p) \otimes \id_C \otimes \id_G - \id_F \otimes \id_C \otimes \mu(\shift^{-1} p \otimes \id_G)\right) (\id_F \otimes \Delta \otimes \id_G)
\end{aligned}
\end{equation*}
where $p \colon C \to \bar{C}$ is the natural projection and we use $\mu$ for the right action of $\cobc{C}$ on $F$ and the left action of $\cobc{C}$ on $G$. 
\end{construction}

\begin{chunk}
\label{c_semifree}
Given a dg algebra $A$, recall that a dg $A$-module $F$ is \emph{semifree} if it admits an exhaustive filtration 
\[
0= F(-1)\subseteq F(0)\subseteq F(1)\subseteq \ldots \subseteq F
\]
where each subquotient $F(i)/F(i-1)$ is a sum of shifts of $A$. As a matter of terminology, a semifree dg $A$-bimodule is a semifree dg module over $A \otimes\op{A}$.

Every dg $A$-module $M$ admits a semifree resolution in the sense that there exists a surjective quasi-isomorphism $F\xra{\simeq} M$, with $F$ a semifree dg $A$-module. Such resolutions are unique up to homotopy; see \cite[Chapter~6]{Felix/Halperin/Thomas:2001} for this fact, as well as other details regarding semifree dg modules.
\end{chunk}

\begin{lemma}
\label{l:diagonal_resolution}
The map $\cobc{C} \totimes C \totimes \cobc{C} \longrightarrow \cobc{C}$ induced from \cref{eq:ses-tensoralg} is a semifree resolution of $\cobc{C}$ as a dg $\cobc{C}$-bimodule.
\end{lemma}

\begin{proof}
By construction, \cref{eq:ses-tensoralg} is a short exact sequence of dg $\cobc{C}$-bimodules, and so it induces an exact triangle in the derived category of dg $\cobc{C}$-bimodules. We obtain the quasi-isomorphism by comparing this triangle to the triangle associated to the cone construction for $\cobc{C} \totimes C \totimes \cobc{C}$.

The dg module $\cobc{C} \totimes C \totimes \cobc{C}$ is semifree as a dg $\cobc{C}$-bimodule since $C$ is free as a module over $Q$ and non-negatively graded.
\end{proof}

\begin{chunk} \label{setup-ring-map-ai-qi}
Let $\phi \colon Q \to R$ be a finite local homomorphism with $A\to R$ a free resolution of $R$ over $Q$. Fix an $R$-complex $M$ and a semifree resolution $\gamma\colon G \to M$ over $Q$. By \cref{ai-along-qi}, there exists a split unital $\ai$-algebra structure $\{m_n\}$ on $A$ and a strictly unital $\ai$-module structure $\{m_n^G\}$ over $A$ on $G$. Then by \cref{ai-mod-dg-mod-cobar-bar}, this induces a dg module structure over $\cobc{\bc{A}}$ on $G$. 

Suppose further that $C$ is a connected curved dg coalgebra with counit $\epsilon\colon C\to Q$, equipped with a weak equivalence of connected curved dg coalgebras $C \to \bc{A}$; cf.\@ \cref{morphism_coalgebra}. Then $R$ and $G$ each have an induced dg module structure over $\cobc{C}$. 
\end{chunk}

\begin{theorem} \label{resn-trans-qi-coalgebras}
In the setting of \cref{setup-ring-map-ai-qi}, the map
\[
R \totimes C \totimes G\longrightarrow M\quad \text{given by}\quad r\otimes c\otimes g \mapsto r \epsilon(c) \gamma(g)
\]
is a semifree resolution of $M$ over $R$. 
\end{theorem}

\begin{proof}
The map $G \to M$ is a quasi-isomorphism of $\ai$-modules over $A$, and by \cref{ai-mod-dg-mod-cobar-bar}, it is a quasi-isomorphism of dg modules over $\cobc{\bc{A}}$, and thus over $\cobc{C}$. From \cref{l:diagonal_resolution}, we obtain the quasi-isomorphism of (left) dg $\cobc{C}$-modules
\begin{equation*}
\beta\colon \cobc{C} \totimes C \totimes G = \cobc{C} \totimes C \totimes \cobc{C} \otimes_{\cobc{C} } G \xra{\simeq} \cobc{C} \otimes_{\cobc{C}} G \cong G\,.
\end{equation*}
Since $G$ is semifree over $Q$, it follows that $\cobc{C} \totimes C \totimes G$ is semifree as a left dg module over $\cobc{C}$.

The claimed quasi-isomorphism fits into the commutative diagram below
\[
\begin{tikzcd}[row sep=3mm, column sep=2mm]
 \& \cobc{C} \totimes C \totimes G \ar[dl,"\alpha"']\ar[dr,"\gamma \beta"]\&\\
 R \totimes C \totimes G \ar[rr] \&\& M\,,
\end{tikzcd}
\]
where $\alpha$ is the map induced by the quasi-isomorphism of dg algebras 
\begin{equation}\label{e:acyclic_qi}
\cobc{C}\xra{\simeq} \cobc{\bc{A}} \xra{\simeq} A \xra{\simeq} R\,;
\end{equation}
where the fact that second map is a quasi-isomorphism follows from the derived version of Nakayama's lemma since, upon applying $-\otimes_Q k$, the map becomes a quasi-isomorphism ~\cite[Section~2.2.1]{LefevreHasegawa:2003}.

As the composition in \cref{e:acyclic_qi} is a quasi-isomorphism and $\cobc{C} \totimes C \totimes G $ is a semifree resolution of $M$ over $\cobc{C}$, it follows that $\alpha$ is a quasi-isomorphism. We have already justified that $\beta$ and $\gamma$ are quasi-isomorphisms, accounting for the downward arrow on the right. As both legs of the triangle in the diagram above are quasi-isomorphisms the horizontal map is a quasi-isomorphism, as claimed.
\end{proof}

\section{\texorpdfstring{$\ai$}{A-infinity}-algebra presentations for Koszul homomorphisms} \label{sec:koszul-via-ai}

We now have the machinery to show that Koszul homomorphisms admit presentations analogous to those of classical Koszul $k$-algebras. The next result lifts these classical quadratic presentations to local algebra, and explains how one may think of Koszul homomorphisms as $\ai$-deformations of Koszul algebras over fields. 

\begin{theorem} \label{koszul-ai-presentation}
A finite local homomorphism $\phi \colon Q \to R$ is Koszul if and only if there is
\begin{enumerate}[(1)]
\item\label{koszul-ai-presentation:V-W} a non-negatively graded, degreewise finite rank free $Q$-module $V$, and a direct summand $W\subseteq V \otimes V$,
\item\label{koszul-ai-presentation:ai} an $\ai$-structure $\{m_n\}$ on the $Q$-module $\tmod{V}/(W)$ with grading induced from the grading of $V$, and
\item\label{koszul-ai-presentation:VtoR} a $Q$-linear map $V_0 \to R$,
\end{enumerate}
such that
\begin{enumerate}[(i)]
\item\label{koszul-ai-presentation:Koszul} the $k$-algebra $\talg{V \otimes k}/(W \otimes k)$ is Koszul with respect to the tensor algebra weight grading,
\item\label{koszul-ai-presentation:modk-talg} $\{m_n\}$ agrees with the algebra structure on $\talg{V}/(W)$ modulo $\fm_Q$; that is 
\begin{equation*}
m_2 \otimes k=\mu \otimes k \quad\text{and}\quad m_n\otimes k = 0 \ \text{ for }\ n \neq 2\,,
\end{equation*}
where $\mu$ is the usual product on the quotient of a tensor algebra, and
\item\label{koszul-ai-presentation:ai-resn} with the structure $\{m_n\}$, the induced map $\tmod{V}/(W)\to R$ is a strict $\ai$-algebra quasi-isomorphism.
\end{enumerate}
Moreover, $V$ can be taken to be finite rank whenever $R$ has finite projective dimension over $Q$.
\end{theorem}

\begin{proof}
If the stated conditions hold then $R\lotimes_Q k$ is formal by \cref{p_formality}, and $\Tor QRk\cong\talg{V \otimes k}/(W \otimes k)$ is Koszul, therefore $\phi$ is Koszul by definition.

Assume, conversely, that $\phi$ is Koszul. 
Since $\Tor{Q}{R}{k}$ is Koszul, it admits a compatible weight grading making it quadratic. That is, we have an isomorphism
\begin{equation*}
\Tor[i]{Q}{R}{k}_{(w)} \cong \big(\talg[w]{\bar{V}}/(\bar{W})\big)_i
\end{equation*}
identifying the product on Tor with the product on the quotient of the tensor algebra, where $\bar{V} = \Tor{Q}{R}{k}_{(1)}$ and $\bar{W} \subseteq \bar{V} \otimes_k \bar{V}$ is a graded subspace.

Let $V$ be a free graded $Q$-module such that $V\otimes k =\bar{V}$, and choose a direct summand $W \subseteq V \otimes V$ such that $W\otimes k =\bar{W}$. If we define $A\coloneqq \tmod{V}/(W)$, then $A$ is a free, bigraded $Q$-module and
\[
A\otimes_Qk= \tmod{V}/(W)\otimes k = \tmod{\bar{V}}/(\bar{W}) \cong \Tor QRk\,.
\]
Therefore we may equip $A$ with a differential making it the minimal $Q$-resolution of $R$. We have constructed \cref{koszul-ai-presentation:V-W} and \cref{koszul-ai-presentation:VtoR} satisfying condition \cref{koszul-ai-presentation:Koszul}. Since $R\lotimes_Qk$ is formal, by \cref{p_formality} there is an $\ai$-structure on $A$ as required for \cref{koszul-ai-presentation:ai}, inducing the algebra structure on $\Tor QRk$ and satisfying the conditions \cref{koszul-ai-presentation:modk-talg} and \cref{koszul-ai-presentation:ai-resn}.
\end{proof}

In \cref{sec:examples-special-koszul} we illustrate \cref{koszul-ai-presentation} in detail using the examples in \cref{sec:examples-koszul}. 

\subsection{Strictly Koszul presentations} \label{sec:special-koszul}

For a local homomorphism $\phi \colon Q \to R$, \cref{resn-trans-qi-coalgebras} allowed us to obtain free resolutions over $R$ starting from free resolutions over $Q$. The main input to this theorem was a curved dg coalgebra $C$ over $Q$ with quasi-isomorphism $\cobc{C}\to R$. Our philosophy is that when $\phi$ is Koszul $C$ should have a simple description. In this section we introduce additional technical assumptions that will allow us to explicitly construct $C$, mimicking a classical construction of Priddy.

\begin{definition} \label{d:special-Koszul}
Let $\phi \colon Q \to R$ be Koszul. Recall from \cref{koszul-ai-presentation} that $R$ admits an $\ai$-algebra resolution $A$ over $Q$ with a quadratic presentation $A\cong \tmod{V}/(W)$ satisfying the conditions \cref{koszul-ai-presentation:Koszul,koszul-ai-presentation:modk-talg,koszul-ai-presentation:ai-resn}. The data $(A,V,W)$ is called a \emph{strictly Koszul presentation} for $\phi$ if, in addition to these conditions, 
\begin{equation} \label{eq:special-Koszul}
\bar{m}_1(V) \subseteq V \quad\text{and}\quad \bar{m}_n\big(\bigcap_{i+2+j=n} V^{\otimes i} \otimes W \otimes V^{\otimes j} \big) \subseteq V \quad \text{for } n \geqslant 2\,,
\end{equation}
where we have used the inclusion $\bigcap_{i+2+j=n} V^{\otimes i} \otimes W \otimes V^{\otimes j}\subseteq V^{\otimes n}\subseteq \bar{A}^{\otimes n}$ to apply the $\ai$-operations $\bar{m}_n$ of $\bar{A}$. If the homomorphism $\phi$ admits a strictly Koszul presentation, then $\phi$ is called \emph{strictly Koszul}.

In this setting, we define
\[
\priddy[n]{V,W} \coloneqq \bigcap_{i+2+j=n} (\susp V)^{\otimes i} \otimes \susp^2 W \otimes (\susp V)^{\otimes j} \subseteq \bc[n]{A}\,.
\]
By definition, the curvature term, the coderivation and the comultiplication on $\bc{A}$ restrict to maps on $\priddy{V,W}$, and hence $\priddy{V,W}$ is a counital curved dg coalgebra. We call $\priddy{V,W}$ the \emph{Priddy coalgebra associated to $(A,V,W)$}. If the presentation is clear from the context, we say it is the \emph{Priddy coalgebra of $\phi$}, and we write
\[
\mathsf{C}(\phi) \coloneqq \priddy{V,W}\,.
\]
\end{definition}

\begin{chunk} \label{dual_priddy}
Let $\phi$ have a strictly Koszul presentation $(A,V,W)$, and let $(-)^\vee$ denote graded $Q$-linear duality. In this setting, one can directly compute from the definition of the Priddy coalgebra of $\phi$ that
\begin{equation}\label{priddy-dual}
\mathsf{C}(\phi)^\vee = \talg{\susp^{-1} V^\vee}/(\susp^{-2} W^\perp)
\end{equation}
where $W^\perp = \set{f \in (V \otimes V)^\vee}{f(W) = 0}$; this uses that $V$ is free and $W\subseteq V\otimes V$ is a summand.
\end{chunk}

What we call the Priddy coalgebra first appeared, for algebras over a field, in the work of Priddy \cite[Section~3]{Priddy:1970}, where it is called the Koszul complex. See also \cite[Section~2.6]{Beilinson/Ginzburg/Soergel:1996} and \cite[Chapter 3]{Loday/Vallette:2012} (where our notation is taken from).

In \cref{sec:examples-special-koszul} we show that complete intersection and Golod homomorphisms are strictly Koszul, as well as Cohen presentations of almost Golod Gorenstein local rings. In fact, we are not able to construct surjective Koszul homomorphisms that are not strictly Koszul, therefore we ask:

\begin{question}\label{q_is_everyone_special}
For a surjective Koszul homomorphism $\phi \colon Q \to R$, is it always possible to construct a strictly Koszul presentation $(A,V,W)$ as in \cref{d:special-Koszul}?
\end{question}

We think of $R$ and $\mathsf{C}(\phi)^\vee$ as being Koszul dual to each other relative to $Q$. The next result justifies this, and in particular it says that this specializes, at the maximal ideal of $Q$, to classical Koszul duality over $k$.

\begin{theorem}\label{priddy-subcoalgebra}
Let $\phi \colon Q \to R$ be a strictly Koszul homomorphism. Then $\priddy{\phi}$ is minimal in the sense that $\partial(\priddy{\phi})\subseteq \fm_Q \priddy{\phi}$, and the inclusion $\priddy{\phi} \to \bc{A}$ is a weak equivalence of connected curved dg coalgebras. Moreover, both $T=\Tor{Q}{R}{k}$ and $E = \Hom{}{\mathsf{C}(\phi)}{k}$ are Koszul $k$-algebras and there $k$-algebra isomorphisms
\[
\Ext{T}{k}{k} \cong E \quad\text{and}\quad \Ext{E}{k}{k} \cong T\,.
\]
\end{theorem}

\begin{proof}
We fix a strictly Koszul presentation, so that $\priddy{\phi}=\priddy{V,W}$.

By \cref{koszul-ai-presentation} the $\ai$-structure on $A$ satisfies $m_n\otimes k=0$ for $n\neq 2$, and $A \otimes k=\talg{V \otimes k}/(W \otimes k)$ is a quadratic algebra. It then follows from \cite[Proposition~3.3.2]{Loday/Vallette:2012} that the differential of $\priddy{V,W}\otimes k =\priddy{V \otimes k,W \otimes k}\subseteq \bc{A \otimes k}$ is zero. Therefore $\priddy{V,W}$ is minimal.

Since $T=A\otimes k$ is Koszul by assumption, $\priddy{V \otimes k,W \otimes k} \to \bc{A \otimes k}$ is a weak equivalence by \cite[Theorem~3.4.6]{Loday/Vallette:2012}. Since $V$ is free and $W\subseteq V\otimes V$ is a summand, $\cobc{\priddy{V \otimes k,W \otimes k}}=\cobc{\priddy{V,W}}\otimes k$ and $\cobc{\bc{A \otimes k}}=\cobc{\bc{A}} \otimes k$, so it follows from the the derived version of Nakayama's lemma that $\priddy{V,W} \to \bc{A}$ is a weak equivalence as well.

Since $\mathsf{C}(\phi)$ is minimal, the coproduct induces the structure of a graded $k$-algebra on $\Hom{}{\mathsf{C}(\phi)}{k}$, with zero differential. Using \cref{dual_priddy}, it follows that
\[
E =\mathsf{C}(\phi)^\vee \otimes k = \talg{\susp^{-1} V^\vee\otimes k}/(\susp^{-2} W^\perp \otimes k)\,.
\]
Therefore $E$ is the quadratic dual of $T=\talg{V\otimes k}/(W\otimes k)$, and the final statement follows from \cite[2.10]{Beilinson/Ginzburg/Soergel:1996}.
\end{proof}

\subsection{The Priddy resolution}

We have arrived at one of the main applications of our techniques. The next result provides explicit ``universal resolutions'' for modules over the target of a strictly Koszul homomorphism. It recovers the Shamash resolution in the case of complete intersection homomorphisms, and the bar resolution of Iyengar and Burke in the case of Golod homomorphisms. We present these and other examples in the next section.

\begin{theorem} \label{priddy-resolution}
Let $\phi \colon Q \to R$ be a Koszul homomorphism with a strictly Koszul presentation $(A,V,W)$. Assume that $M$ is an $R$-complex with a semifree resolution $G \to M$ over $Q$ and that $G$ has a strictly unital $\ai$-module structure over $A$. Then
\begin{equation*}
R \totimes \priddy{V,W} \totimes G \longrightarrow M
\end{equation*}
is a semifree resolution over $R$, with differential given by
\begin{align*}
\partial^\tau &= \sum_{\substack{r+s+t=n\\r,t \geqslant 0, s \geqslant 1}} (-1)^{\frac{s(s+1)}{2}} \id_R \otimes (\id^{\otimes r} \otimes \shift \bar{m}_s (\shift^{-1})^{\otimes s} \otimes \id^{\otimes t}) \otimes \id_G \\
&\quad + \sum_{\substack{i+j=n+1\\i \geqslant 0, j \geqslant 1}} (-1)^{\frac{(j-1)(j-2)}{2}} \id_R \otimes \id^{\otimes i} \otimes \bar{m}^G_j ((\shift^{-1})^{\otimes (j-1)} \otimes \id_G)\,.
\end{align*}
\end{theorem}

\begin{proof}
By \cref{priddy-subcoalgebra} we can apply \cref{resn-trans-qi-coalgebras} to obtain the result.
\end{proof}

We call $R \totimes \priddy{V,W} \totimes G$ the \emph{Priddy resolution of $M$} associated to the strictly Koszul presentation $(A,V,W)$. We emphasize that (as long as $M$ and $R$ have finite projective dimension over $Q$) there is \emph{only a finite amount of data needed to construct the Priddy resolution}. Therefore, it would be especially interesting to give an effectively computable answer to \cref{q_is_everyone_special}. 

\begin{chunk} \label{c:lescot}
For any surjective map $\phi \colon Q \to R$ of local rings with common residue field $k$, and any finitely generated $R$-module $M$, Lescot~\cite{Lescot:1990} established the coefficientwise inequality
\[
\ps^R_M(t)\cdot \ps^Q_k(t)\preccurlyeq \ps^Q_M(t)\cdot \ps^R_k(t)\,.
\]
If equality holds, $M$ is said to be \emph{inert} by $\phi$.
\end{chunk}

\begin{chunk} \label{c:small}
A surjective map $\phi \colon Q \to R$ of local rings with common residue field $k$ is called \emph{small} if the induced map $\Tor Qkk \to \Tor Rkk$ is injective \cite{Avramov:1978}. For example, any minimal Cohen presentation is small. When $\phi$ is small, there is an equality $\ps^Q_k(t)\cdot \ps^{R\lotimes_Q k}_k(t)=\ps^R_k(t)$ by \cite[Corollary~5.3]{Avramov:1978}.
\end{chunk}

The next result addresses the (non-)minimality of the Priddy resolution.

\begin{theorem}
\label{t_inert}
Let $\phi\colon Q\to R$ be a surjective map of local rings with common residue field $k$. If $\phi$ is small and strictly Koszul with Priddy coalgebra $\priddy{\phi}$, then
\[
\sum_i \rank_Q(\priddy{\phi}_i)t^i=\frac{ \ps^R_k(t)}{P^Q_k(t)}\,.
\]
Moreover, for any finitely generated $R$-module $M$ there is a coefficientwise inequality 
\[
\ps^R_M(t)\preccurlyeq \frac{\ps^Q_M(t)\cdot \ps^R_k(t)}{P^Q_k(t)}\,,
\]
Equality holds if and only if $M$ is inert by $\phi$, if and only if its Priddy resolution with respect to $\phi$ is a minimal resolution. In particular, the Priddy resolution of the residue field $k$ is minimal. 
\end{theorem}

\begin{proof}
We may compute $\H_{\sf C}(t)\coloneqq\sum_i \rank_Q(\priddy{\phi}_i)t^i$ as follows:
\[
\H_{\sf C}(t)=\sum_i \rank_k(\Hom{}{\priddy{\phi}_i}{k}) t^i=\ps^{\Tor QRk}_k(t)=\ps^{R\lotimes_Q k}_k(t)=\frac{ \ps^R_k(t)}{P^Q_k(t)}\,;
\]
the second equality follows from that last statement in \cref{priddy-subcoalgebra}; the third uses formality of $R\lotimes_Q k$; and the last uses the small hypothesis, explained in \cref{c:small}.

\cref{priddy-resolution} directly yields the inequality
\begin{equation}\label{l_ps_inequality_proof}
\ps^R_M(t)\preccurlyeq \ps^Q_M(t)\cdot \H_{\sf C}(t)\,,
\end{equation}
with equality if and only if the Priddy resolution is minimal. At the same time, the computation of $\H_{\sf C}(t)$ above transforms \cref{l_ps_inequality_proof} into the inequality stated in the theorem, and equality holds there by definition when $M$ is inert; see \cref{c:lescot}.
\end{proof}

\begin{remark}
\cref{t_inert} recovers Lescot's bound in \cref{c:lescot} for the homomorphisms considered. One cannot directly recover the former from the latter using manipulations of formal power series as the coefficients of $\ps_k^Q(t)^{-1}$ can be negative. 
\end{remark}

\section{Examples of strictly Koszul presentations} \label{sec:examples-special-koszul}

In this final section we will apply the theory developed above in a series of examples, obtaining explicit resolutions for modules over various classes of rings. We also survey how these constructions relate to known resolutions in the literature.

We fix a local ring $Q$ with residue field $k$.

\begin{figure}[h]
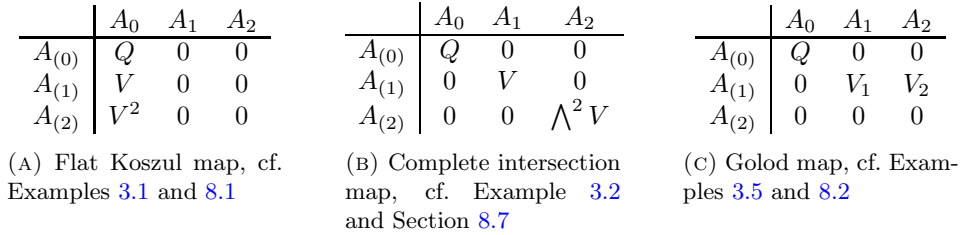

\begin{subfigure}[b]{0.29\textwidth}
\centering
\begin{tabular}{c|c c c}
 & $A_0$ & $A_1$ & $A_2$ \\ \hline 
$A_{(0)}$ & $Q$ & 0 & 0 \\ 
$A_{(1)}$ & $V$ & 0 & 0 \\ 
$A_{(2)}$ & $V^2$ & 0 & 0
\end{tabular}
\caption{Flat Koszul map, cf.\@ \cref{e:koszul:flat,e:special-koszul:flat}\\
 }
\label{fig:grading:flat}
\end{subfigure}
\hfill
\begin{subfigure}[b]{0.29\textwidth}
\centering
\begin{tabular}{c|c c c}
 & $A_0$ & $A_1$ & $A_2$ \\ \hline 
$A_{(0)}$& $Q$ & 0 & 0 \\ 
$A_{(1)}$ & 0 & $V$ & 0 \\ 
$A_{(2)}$& 0 & 0 & $\bigwedge^2V$
\end{tabular}
\caption{Complete intersection map, cf.\@ \cref{e:koszul:ci,e:special-koszul:ci}}
\label{fig:grading:ci}
\end{subfigure}
\hfill
\begin{subfigure}[b]{0.29\textwidth}
\centering
\begin{tabular}{c|c c c}
 & $A_0$ & $A_1$ & $A_2$ \\ \hline 
$A_{(0)}$& $Q$ & 0 & 0 \\ 
$A_{(1)}$ & 0 & $V_1$ & $V_2$ \\ 
$A_{(2)}$& 0 & 0 & 0
\end{tabular}
\caption{Golod map, cf.\@ \cref{e:koszul:golod,e:special-koszul:golod}\\ }
\label{fig:grading:golod}
\end{subfigure}

\caption{Illustration of the weight and homological gradings of the $\ai$-resolution $A$ for various examples.}
\end{figure}

\begin{example}[Flat Koszul homomorphisms] \label{e:special-koszul:flat}
 We begin with a presentation for a commutative Koszul $k$-algebra:
\[
K={k[x_1,\ldots,x_n]}/{(f_1,\ldots,f_m)}\,,
\]
where $f_1,\ldots,f_m$ are quadratic polynomials. To deform this presentation, we consider the $Q$-algebra $Q[x_1,\ldots,x_n]$, weight graded by polynomial degree. We choose elements $F_1,\ldots,F_m$ such that for each $i$
\[
F_i=F_{i,(2)}+F_{i,(1)}+F_{i,(0)}\quad \text{with}\quad F_{i,(w)}\in Q[x_1,\ldots,x_n]_{(w)}\,,
\]
and such that modulo $\fm_Q$, in $k[x_1,\ldots,x_n]$, we have
\[
F_{i,(2)} = f_i\quad \text{and} \quad F_{i,(1)}=F_{i,(0)}= 0\,.
\]
By construction, the homomorphism
\[
\phi\colon Q\longrightarrow R \coloneqq \frac{Q[x_1,\ldots,x_n]}{(F_1,\ldots,F_m)}
\]
is flat, and its fiber $K=R\otimes_Qk$ is Koszul. Therefore $\phi$ is a Koszul homomorphism, as in \cref{e:koszul:flat}.

To show that $\phi$ is strictly Koszul, we take $V=Q[x_1,\ldots,x_n]_{(1)}$ and 
\[
 W=\big\langle \{x_i\otimes x_j-x_j\otimes x_i\}_{ij}, \ \widetilde{F}_{1,(2)},\ldots,\widetilde{F}_{m,(2)} \big\rangle \subseteq V\otimes V\,,
\]
where $\widetilde{F}_{i,(2)}$ are preimages of $F_{i,(2)}$ in $V\otimes V$. Then $\talg{V}/(W)\otimes k\cong K$, so we may choose a compatible isomorphism of $Q$-modules
\[
R\cong \talg{V}/(W)
\]
that restricts to the identity of $V$. 

We obtain a presentation satisfying the conditions of \cref{koszul-ai-presentation}, using $A=R$ with only $m_2$ nonzero. To show that the presentation is strict we note that since $R$ is commutative
\[
m_2(x_i\otimes x_j-x_j\otimes x_i)=0\,,
\]
and we note that since $F_i=0$ in $R$,
\[
m_2(\widetilde{F}_{i,(2)})+F_{i,(1)}+F_{i,(0)} = m_2(\widetilde{F}_{i,(2)}+F_{i,(1)}\otimes 1 + F_{i,(0)}\otimes 1)=0\,.
\]
This shows that
\[
\bar{m}_2(W)\subseteq \big\langle F_{1,(1)},\ldots ,F_{m,(1)}\big\rangle \subseteq V\,.
\]
We can conclude that $(R,V,W)$ is a strictly Koszul presentation for $\phi$. In \cref{fig:grading:flat} we illustrate the grading of $A$.
\end{example}

\begin{example}[Golod homomorphisms] \label{e:special-koszul:golod}
This is the primary example treated by Burke in \cite{Burke:2015}, at least when $Q$ is regular. Continuing \cref{e:koszul:golod}, let $\phi \colon Q \to R$ be a surjective local Golod homomorphism, with a minimal resolution $A$ of $R$ over $Q$. By \cite[Theorem~6.13]{Burke:2015}, for every $\ai$-algebra structure $\{m_n\}$ on $A$ one has $m_n \otimes_Q k = 0$ for $n \neq 2$. Then 
\begin{equation*}
R \lotimes_Q k = A \otimes_Q k = k \ltimes U = \talg{U}/(U \otimes U)
\end{equation*}
where $U$ is the graded $k$-vector space $A_{\geqslant 1}\otimes k$. In particular, lifting this isomorphism to $Q$ we obtain $A \cong \talg{V}/(W)$ with $V = A_{\geqslant 1}$ and $W=V \otimes V$. This presentation satisfies the conditions of \cref{koszul-ai-presentation}. Further, the data $(A,V,W)$ is a strictly Koszul presentation, and the Priddy coalgebra of $\phi$ is the bar construction $\priddy{V,W} = \bc{A}$. In this case the Priddy resolution of a module $M$ recovers the bar resolution $R \totimes \bc{A} \totimes G$ from \cite[Theorem 3.13]{Burke:2015}. Comparing \cref{c:lescot} with \cref{golod_bound}, the resolution is minimal if and only if $M$ is inert with respect to $\phi$, if and only if $M$ is a $\phi$-Golod module (i.e.\@ the Serre bound \cref{golod_bound} is an equality), by \cref{t_inert}.

\begin{remark}\label{r_no_dg_structure}There are many examples of Golod (in particular, Koszul) homomorphisms $\phi\colon Q\to R$ such that the minimal resolution $A$ of $R$ over $Q$ does not admit a dg $Q$-algebra structure. In fact, this behavior seems to be typical. One way to construct them is as follows.

Let $I$ be an ideal in a local ring $P$, and consider the map $\phi\colon Q\to R$ with 
\[
Q=P[x]_{(\fm_P,x)} \quad \text{and}\quad R=Q/(x I)\,.
\]
By \cite[Theorem~2.4]{Levin:1976}, see also \cite{Shamash:1969}, $\phi$ is Golod. If $B$ is the minimal resolution of $P/I$ over $P$, then the minimal resolution $A$ of $R$ over $Q$ can be described as
\[
A_i=B_i\otimes_PQ, \quad \text{with}\quad \partial^{A}_1=x\cdot\partial^B_1\otimes_PQ \quad \text{and}\quad \partial^{A}_{\geqslant 2}=\partial^B_{\geqslant 2}\otimes_PQ\,.
\]
If $A$ were to admit a dg $Q$-algebra structure, then localizing would produce a dg $P(x)$-algebra structure on
\[
 A\otimes_Q Q_{(\fm_P[x])} \cong B\otimes_P P(x)\,,
\]
and this is a minimal resolution of $P(x)/I(x)$ over $P(x)$. 

However, we can start with examples of $P$ and $I$ such that this is impossible. To be concrete, \cref{e:koszul:nonexample} (replacing $k$ with $k(x)$) shows that there is no such dg algebra structure when 
$P(x)=k(x)\llbracket a,b,c,d\rrbracket$ and $I(x)=(a^2,ab,bc,cd,d^2)$. 
It follows that the homomorphism
\[
\phi\colon k\llbracket a,b,c,d,x\rrbracket \longrightarrow \dfrac{k\llbracket a,b,c,d,x\rrbracket}{(a^2x,abx,bcx,cdx,d^2x)}
\]
is Golod and there is no dg algebra structure on the minimal resolution of the target over the source.
\end{remark}
\end{example}

\begin{example}[Gorenstein homomorphisms of projective dimension $3$] \label{e:koszul-ai:buchsbaum-eisenbud} \hfill Assume 
that $\phi \colon Q \to R$ is a surjective local Gorenstein map of $\pdim_Q(R) = 3$. In \cref{e:koszul:buchsbaum-eisenbud} the dg algebra resolution $A$ of $R$ is described, with bases $\{e_i\}$, $\{f_i\}$ and $\{g\}$ for $A_1$, $A_2$ and $A_3$ respectively. The multiplication induces a perfect pairing
\[
\langle -,-\rangle \colon A\otimes A \longrightarrow \susp^3A_3 \cong \susp^3Q
\]
that makes $A$ a cyclic $\ai$-algebra. We take $V=A_1 \oplus A_2 $ and $W = \ker(\langle -,- \rangle|_{V\otimes V})$; alternatively, $W$ is freely spanned as a graded $Q$-module by
\[
\{e_i \otimes e_j, f_i \otimes f_j, e_i \otimes f_i - f_j \otimes e_j\}_{i,j} \cup \{e_i \otimes f_j, f_i \otimes e_j\}_{i \neq j}\,.
\]

The short Gorenstein description of $A \otimes_Q k$ lifts to an isomorphism of graded $Q$-modules $A \cong \tmod{V}/(W)$ satisfying the conditions of \cref{koszul-ai-presentation}. 

Using the explicit description in \cref{e:koszul:buchsbaum-eisenbud}
\begin{equation*}
\bar{m}_1(V) \subseteq A_1 \subseteq V \,,\quad \bar{m}_2(W) \subseteq A_2 \subseteq V \quad\text{and}\quad \bar{m}_n = 0 \quad\text{for } n \geqslant 3\,.
\end{equation*}
Therefore the homomorphism $\phi$ is strictly Koszul, using the presentation $(A,V,W)$. The corresponding Priddy coalgebra is given by
\[
 \priddy[n]{V,W} = \Big\{{\textstyle{\sum} v_1 \otimes \cdots \otimes v_n} ~\Big|~\textstyle{\sum}  v_1 \otimes \cdots \otimes\langle v_i, v_{i+1} \rangle v_{i+2} \otimes \cdots \otimes v_{n}= 0,\ 1 \leqslant i < n\Big\}.
 \]
Alternatively, $\priddy{V,W}$ can be described explicitly using the basis of $W$ above. 
We also note that the Priddy coalgebra is dual to the non-commutative hypersurface
\begin{equation*}
\priddy{V,W}^\vee \cong \talg{V^\vee}/( \rho )\,,
\end{equation*}
where $\rho = e_1^\vee \otimes f_1^\vee + f_1^\vee \otimes e_1^\vee + \cdots + e_r^\vee \otimes f_r^\vee + f_r^\vee \otimes e_r^\vee$.
\end{example}

\begin{remark}
In \cite[Example~3.10]{Burke:2015}, Burke examines the specific Gorenstein ring $R=Q/I$, of codimension three, where
\[
Q=k\llbracket x,y,z\rrbracket\quad\text{and} \quad I =(x^2,yz,xy+z^2,xz,y^2)\,.
\]
In particular, Burke explicitly computes the $\ai$-module $A$-structure on $K^Q$ and uses this to obtain the (non-minimal) bar resolution $R \totimes \bc{A} \totimes K^Q$ of $k$. In comparison, by \cref{priddy-resolution} the Priddy resolution $R \totimes \priddy{V,W} \totimes K^Q$ of $k$, with respect to $(A,V,W)$ in \cref{e:koszul-ai:buchsbaum-eisenbud}, is minimal by \cref{t_inert}.
\end{remark}

\begin{remark}
A similar argument to \cref{e:koszul-ai:buchsbaum-eisenbud} shows that if $\phi\colon Q\to R$ is a minimal Cohen presentation for an almost Golod Gorenstein ring, and if the minimal $Q$-free resolution of $R$ admits a dg algebra structure,  
then $\phi$ is strictly Koszul. The minimal resolution is known in the case of a compressed artinian Gorenstein ring \cite{Miller/Rahmati:2018}, and it is suspected to carry a dg algebra structure.

In \cref{t_golod_gor_strict} we will generalize this substantially, showing that it is only necessary for the minimal resolution to admit a cyclic $\ai$-algebra resolution.
\end{remark} 

\subsection{Complete intersection homomorphisms} \label{e:special-koszul:ci}

Next we show that surjective complete intersection homomorphisms are strictly Koszul, and that the resulting Priddy resolution recovers a well known construction of Eisenbud \cite{Eisenbud:1980} and Shamash \cite{Shamash:1969}. The latter uses \emph{systems of higher homotopies} 
to obtain free resolutions over the target of a surjective complete intersection homomorphism, starting from data over the source. We first recall this story, which provides context for some of the results in this subsection. We then proceed to verify that such maps are strictly Koszul, and conclude by laying out the connection between $\ai$-structures and systems of higher homotopies.

In what follows, we return to the setting of \cref{e:koszul:ci}. Namely, $\phi\colon Q\to R$ is a surjective, local homomorphism where $\ker\phi$ is generated by a $Q$-regular sequence $\bm{f}=f_1,\ldots,f_c$, and $A=\Kos^Q(\bm{f})$.

\begin{chunk}\label{c:higher_homotopies}
Let $M$ be an $R$-module and $ G \to M$ a free resolution over $Q$. A \emph{system of higher homotopies}, corresponding to $\bm{f}$, on $G$ is a collection of maps $\sigma^{(\bmalpha)} \colon G \to G$, one for each $\bmalpha \in \BN_0^c$,
of degree $2 |\bmalpha|-1$ such that:
\begin{enumerate}
\item \label{highhom_1}$\sigma^{({\bf{0}})} = \partial^G$ where ${\bf{0} }= (0, \ldots, 0)$;
\item \label{highhom_2} $\sigma^{({\bf{0}})} \sigma^{({\bf{e}}_i)} + \sigma^{({\bf{e}}_i)} \sigma^{({\bf{0}})} = f_i \id_G$ where ${\bf{e}}_i = (0, \ldots, 0,1,0, \ldots 0)$; 
\item \label{highhom_3}for any $\bmalpha \in \BN_0^c$ with $|\bmalpha| > 1$ one has
$
\sum_{\bmalpha = \bmbeta+\bmgamma} \sigma^{(\bmbeta)} \sigma^{(\bmgamma)} = 0
$.
\end{enumerate}
Such a system of maps always exists by \cite[Section~7]{Eisenbud:1980}. The utility of this data is summarized in the following construction: if $D$ denotes the graded $Q$-linear dual of $Q[\chi_1,\ldots,\chi_c]$, where each $\chi_i$ has homological degree $-2$, then the $R$-complex $R\otimes D \otimes G$ with differential $\sum_{\bmalpha\in \BN_0^c}1\otimes \bmchi^{\bmalpha}\otimes \sigma^{(\bmalpha)}$ is a free resolution of $M$ over $R$; see \cite[Section~7]{Eisenbud:1980} for more details.

When $M$ is an $R$-complex, one can take $\epsilon\colon G\to M$ to be a semifree resolution over $Q$ and impose also the following condition to obtain analogous results:
\begin{enumerate}
\setcounter{enumi}{3}
\item \label{highhom_4} $\epsilon \sigma^{(\bmalpha)} = 0$ for $|\bmalpha| > 0$\,.
\end{enumerate}
Such a system of maps exists and can be used to transfer semifree resolutions over $Q$ to ones over $R$, by an argument similar to the classical one in \cite{Eisenbud:1980}; this will be contained in future joint work of Grifo with the first and fourth author.
\end{chunk}

\begin{remark}\label{r_jesse_ci}
Let $M$ be an $R$-module and $ G \to M$ a free resolution over $Q$. In \cite{Burke:2015}, Burke notes that when $c=1$ an $\ai$-module structure on $G$ over $A$ is equivalent to a system of higher homotopies on $G$. Furthermore the bar resolution of $M$ in Burke's paper agrees with the Priddy resolution of $M$, introduced above. Such maps are also Golod, and so we are also in the setting of \cref{e:special-koszul:golod}.

For arbitrary codimension $c$ the bar resolution is not minimal. However, in unpublished work, Burke constructs an acyclic twisting cochain $D\to A$ and uses this to transfer a semifree resolution of an $R$-complex $M$ over $Q$ to one over $R$ that agrees with the construction of Eisenbud and Shamash; cf.\@ \cref{c:higher_homotopies} (see also \cite{Avramov/Buchweitz:2000a,Martin:2021}). The connection between higher homotopies and $\ai$-structures is also implicit in Burke's work. We will give an explicit description of how these structures relate in \cref{higher_homotopies_ai_structure}.
\end{remark}

\begin{chunk} \label{ci_strictly_kos}
The narrative above is subsumed by the one in this article. Specifically, the dg algebra resolution $A = \Kos^Q(\bmf)$ of $R$ over $Q$ has a quadratic presentation $\talg{V}/(W)$, with
\begin{equation*}
V = A_1 = \susp Q^c \quad\text{and}\quad W = \left\langle\{ a\otimes a\}_{a\in A_1}\cup \{a\otimes b+b\otimes a\}_{a,b\in A_1}\right\rangle \subseteq V^{\otimes 2}\,.
\end{equation*}
The graded module $V$ is concentrated in degree $1$, and the weight and homological gradings agree. It is straightforward to check that this presentation satisfies the conditions of \cref{koszul-ai-presentation}. By construction,
\begin{equation}\label{Kos_complex_ai_structure}
\bar{m}_1(V) =0 \,,\quad \bar{m}_2(W)=0\quad\text{and}\quad \bar{m}_n = 0 \quad\text{for } n \geqslant 3\,.
\end{equation}
Hence $\phi$ is strictly Koszul.
\end{chunk}

To conclude that the constructions in \cref{c:higher_homotopies,r_jesse_ci} are recovered by \cref{priddy-resolution}, we end this subsection with the following analysis. 

\begin{chunk} \label{symmetric-tensors} 
Let $\symgrp{n}$ be the symmetric group. For $\bmalpha = (\alpha_i) \in \BN_0^{c}$ with $|\bmalpha|=n$ we let
\begin{equation*}
\symgrp{\bmalpha} \coloneqq \set{\tau \in \symgrp{n}}{\tau(\alpha_1+\cdots+\alpha_i+1) < \cdots < \tau(\alpha_1+\cdots+\alpha_{i+1}) \text{ for } 0\leqslant i\leqslant c-1}
\end{equation*}
denote the subgroup of \emph{$\bmalpha$-shuffles} \cite[Chapter~IV, \S5.3]{Bourbaki:1981}.

The symmetric group $\symgrp{n}$ acts on $(\susp V)^{\otimes n}$ by permuting simple tensors, there are no signs appearing since $\susp V$ is in degree $2$. The \emph{module of symmetric tensors on $\susp V$} is the graded module $\stmod{\susp V} =\bigoplus_{n\geqslant 0} \stmod[n]{\susp V}$, with
\[
\stmod[n]{\susp V} \coloneqq \tmod[n]{\susp V}^{\symgrp{n}}\,.
\]
The coalgebra structure on $\tcoa{\susp V}$ restricts to a coalgebra structure on $\stmod{\susp V}$. We call this the \emph{coalgebra of symmetric tensors on $\susp V$} and denote it by $\stcoa{\susp V}$. 

We denote the basis of $\susp V = \susp^2 Q^c$ corresponding to $f_1, \ldots, f_c$ by $y_1, \ldots, y_c$. A basis of $\stmod{\susp V}$ is given by
\begin{equation*}
y^{(\bmalpha)} \coloneqq \sum_{\tau\in\symgrp{\bmalpha}}\tau\cdot(y_1^{\otimes \alpha_1}\otimes \cdots \otimes y_c^{\otimes \alpha_c}) \in \stcoa[|\bmalpha|]{\susp V}\,.
\end{equation*}
\end{chunk}

\begin{theorem}
\label{t_ci} \label{higher_homotopies_ai_structure}
 Let $\phi \colon Q \to R$ be a surjective complete intersection homomorphism with kernel generated by a $Q$-regular sequence $\bm{f}=f_1, \ldots, f_c$, and let $M$ denote an $R$-complex.
 \begin{enumerate}
 \item\label{t_ci_priddy} $\phi$ is strictly Koszul and its Priddy coalgebra is the curved coalgebra of symmetric tensors $\stcoa{\susp^2 Q^c}$, with curvature term $(f_1, \ldots, f_c) \colon \susp^2 Q^c \to Q$.
 \item \label{t_ci_homotopies} Given a semifree resolution $G\to M$ over $Q$ there exists a strictly unital $\ai$-module structure $\{m_n^G\}$ over $A=\Kos^Q(\bmf)$ making $G\to M$ a strict morphism of $\ai$-modules over $A$, where $A$ acts on $M$ via restricting scalars along the dg algebra map $A\to R$. Then setting
\[ 
\sigma^{(\bmalpha )} \coloneqq (-1)^{\frac{|\bmalpha|(|\bmalpha|-1)}{2}} m^G_{|\bmalpha|+1}((\shift^{-1})^{\otimes |\bmalpha|}y^{(\bmalpha)}\otimes\id)
\]
for $\bmalpha \in \BN_0^{c}$ and $y_1,\ldots,y_c$ the standard basis for $\susp^2Q^c$ defines a system of higher homotopies on $G$ corresponding to $\bmf$. 
\item \label{t_ci_resolution} Moreover, the Priddy resolution from \cref{priddy-resolution} recovers the Eisenbud--Shamash resolution described in \cref{c:higher_homotopies}.
\end{enumerate}
\end{theorem}

\begin{proof}
We saw in \cref{ci_strictly_kos} that $\phi$ is strictly Koszul. Using the same notation we first show that $\priddy{V,W} = \stmod{\susp V}$. Indeed for the nontrivial element $\tau \in \symgrp{2}$ one has
\begin{equation*}
\susp^2 W = \ker\left(\susp V \otimes \susp V \xra{\id-\tau} \susp V \otimes \susp V\right) = \stmod[2]{\susp V}\,,
\end{equation*}
and so using the transposition $\tau_{i} = (i \ \ i+1) \in \symgrp{n}$ we obtain
\begin{equation*}
(\susp V)^{\otimes i-1} \otimes \susp^2 W \otimes (\susp V)^{\otimes n-i-1} = \ker\left((\susp V)^{\otimes n} \xra{\tau_{i} - \id} (\susp V)^{\otimes n}\right).
\end{equation*}
Since $\symgrp{n}$ is generated by the transpositions $\tau_i$, it follows that $\priddy{V,W} = \stmod{\susp V}$. 

The coalgebra structure on $\priddy{V,W}$ is inherited from $\bc{A}$, and this coincides with the coalgebra structure on $\tcoa{\susp V}$ because of the compatible inclusions
\begin{equation*}
\priddy{V,W} = \stcoa{\susp V} \subseteq \tcoa{\susp V} \subseteq \tcoa{\susp\bar{A}}=\bc{A}\,.
\end{equation*}
The differential on $\priddy{V,W}$ is zero by \cref{Kos_complex_ai_structure}. It is straightforward to see that the curvature term on $\priddy{V,W}$ is (up to a shift) the first differential of $A$. This completes the proof of (\cref{t_ci_priddy}).

For (\cref{t_ci_homotopies}), such an $\ai$-module structure making $G \to M$ a strict morphism exists by \cref{c_ai_module}. Then, by definition \cref{c:higher_homotopies}(\cref{highhom_1}) holds. The fact that \cref{c:higher_homotopies}(\cref{highhom_2}) holds follows from using the second Stasheff identity from \cref{ai-module}. Another computation using the Stasheff identities, the unital structure on $G$, and the fact that $A$ is graded-commutative show that \cref{c:higher_homotopies}(\cref{highhom_3}) holds. The verification of \cref{c:higher_homotopies}(\cref{highhom_3}) for $\bmalpha=\textbf{e}_1+\textbf{e}_2$ is illustrative of the proof for the general case and so we sketch this case below. 

Fix basis elements $e_i\in A_1$ with $\partial^A(e_i)=f_i,$ and note that 
\begin{equation}
\label{e_shuffle}
(\shift^{-1})^{\otimes 2} y^{(\bmalpha)}=e_1\otimes e_2+e_2\otimes e_1\,.
\end{equation}
Observe that for $\{i,j\}=\{1,2\}$ one has
\begin{align*}
m_2^G&(e_i\otimes m^G_2(e_j\otimes\id)-(e_i\cdot e_j)\otimes\id)\\
&=\partial^G m_3^G(e_i\otimes e_j\otimes \id)
+m_3^G(f_i \otimes e_j\otimes \id-e_j\otimes f_i \otimes\id+ e_i \otimes e_j \otimes\partial^G )\\
&=\partial^G m_3^G(e_i\otimes e_j\otimes \id)
+m_3^G(e_i \otimes e_j \otimes\id )\partial^G\,,
\end{align*}
where the first precomposes the third Stasheff identity from \cref{ai-module} with $e_i\otimes e_j\otimes \id$, while the third equality uses that $\{m_n^G\}$ is a strictly unital $\ai$-module structure. It follows that 
\[
\partial^G m_3^G(e_i\otimes e_j\otimes \id) -m_2^G(e_i\otimes m^G_2(e_j\otimes\id))+m^G_2(e_i\cdot e_j\otimes\id)
+m_3^G(e_i \otimes e_j \otimes\id )\partial^G=0\,,
\]
and so adding these expressions for $(i,j)=(1,2)$ and $(i,j)=(2,1)$, and recalling \cref{e_shuffle}, we obtain
\[
\sum_{\bmbeta+\bmgamma=\bmalpha}\sigma^{(\bmbeta)} \sigma^{(\bmgamma)} +m_2^G(e_1\cdot e_2\otimes \id)+m_2^G(e_2\cdot e_1\otimes\id)=0\,.
\label{e_homotopies_ci}
\]
It now remains to observe that since $A$ is graded-commutative 
\[
m_2^G(e_1\cdot e_2\otimes \id)+m_2^G(e_2\cdot e_1\otimes\id)=m_2^G((e_1\cdot e_2+e_2\cdot e_1)\otimes \id)=0\,.
\]
Thus \cref{c:higher_homotopies}(\cref{highhom_3}) holds for $\bmalpha=\textbf{e}_1+\textbf{e}_2$.

The condition \cref{c:higher_homotopies}(\cref{highhom_4}) holds since $\epsilon$ is a strict morphism, and since $M$ is a dg $A$-module where the $e_i$'s act trivially. This completes the proof of (\cref{t_ci_homotopies}). 

It remains to show (\cref{t_ci_resolution}). By \cite[IV.\S.5.11]{Bourbaki:1981} we have a natural isomorphism of algebras
\begin{equation*}
Q[\chi_1,\ldots,\chi_c] \cong \stcoa{\susp^2 Q^c}^\vee =\priddy{V,W}^\vee
\end{equation*}
determined by $\chi_i\mapsto y_i^\vee;$
this correspondence can also be seen via \cref{priddy-dual}:
\begin{equation*}
\priddy{V,W}^\vee \cong \talg{\susp^{-2} (Q^c)^\vee}/(\susp^{-1} W^\perp) = \operatorname{Sym}(\susp^{-2} (Q^c)^\vee)\cong Q[\chi_1,\ldots,\chi_n]
\end{equation*}
where $W^\perp = \set{f \otimes g - g \otimes f}{f,g \in \susp^{-1} (Q^c)^\vee}$, identifying
$(y^{(\bmalpha)})^\vee$ with $\bm{\chi}^{\bmalpha}$ for each $\bmalpha\in \BN_0^c$. As a consequence, dualizing the correspondence above yields an isomorphism of graded $Q$-modules $D\cong \priddy{V,W}$ inducing an isomorphism of graded $R$-modules 
\[
\psi\colon R\otimes D\otimes G\xra{\ \cong\ } R\otimes \mathsf{C}(V,W)\otimes G\,.
\]
It remains to observe that 
\begin{equation*}
\psi \circ \sum_{|\bmalpha| = n-1} \bmchi^{\bmalpha} \otimes \sigma^{(\bmalpha)}=(-1)^{\frac{(n-1)(n-2)}{2}} \left.\bar{m}_n^G((\shift^{-1})^{\otimes (n-1)} \otimes \id)\right|_{\stcoa[n-1]{\susp^2 Q^c} \otimes G} \,.
\end{equation*}
Therefore $\psi$ is compatible with the differentials of its target and source, and so it is an isomorphism of $R$-complexes; cf.\@ \cref{priddy-resolution} and \cref{c:higher_homotopies}. 
\end{proof}

\begin{remark}
The higher homotopies $\sigma^{(\bmalpha)}$ with $|\bmalpha|=n$ induce a $\stcoa{\susp^2 Q^c}$-comodule structure on $\stcoa{\susp^2 Q^c} \otimes G$; in fact conditions \cref{c:higher_homotopies}(\cref{highhom_1,highhom_2,highhom_3,highhom_4}) are equivalent to this. On the other hand an $\ai$-module structure on $G$ is equivalent to a $\bc{A}$-comodule structure on $\bc{A} \otimes G$. Hence a system of higher homotopies on $G$ captures the `symmetric' part of an $\ai$-module structure on $G$ over $A$.
\end{remark}

\begin{remark}
    Moving beyond finite projective dimension, complete intersection homomorphisms fit into the well-studied class of quasi-complete intersection homomorphisms; cf.\@ \cite{Avramov/Henriques/Sega:2013,Avramov/Iyengar:2003}. In residual characteristic zero and two, it is straightforward to check that such maps are strictly Koszul; this provides more examples of strictly Koszul homomorphisms of infinite projective dimension. In odd characteristic the presence of divided powers prevents $\Tor QRk$  from admitting a quadratic presentation.
\end{remark}

\subsection{Almost Golod Gorenstein rings} \label{e:koszul-ai:minl-non-golod}

To end the paper, we return to the class of almost Golod Gorenstein rings that we studied in \cref{e:minl-non-golod}. We show that these rings are strictly Cohen Koszul, i.e.\@ every Cohen presentation is a strictly Koszul map, and we thereby obtain concrete free resolutions for all modules over such rings, using the machinery developed in \cref{sec:special-koszul}.

The next lemma is a general construction in the homological algebra of Gorenstein rings, building on work of Avramov and Levin \cite{Avramov/Levin:1978}. 

\begin{lemma}\label{l_socle_res}
Let $R$ be a zero dimensional Gorenstein ring of codimension $d$, with a minimal Cohen presentation $Q\to R$, and let $A\xrightarrow{\simeq}R$ be the minimal $Q$-free resolution of $R$. The inclusion of the socle lifts to a chain map
\[
\begin{tikzcd}[row sep=small]
 K^Q \ar[r,dashed]\ar[d,"\simeq"'] \& A\ar[d,"\simeq"]\\
 \soc(R) \ar[r] \& R
\end{tikzcd}
\]
where $K^Q$ is the Koszul complex of $Q$. The subcomplex
\[
A'\coloneqq\cone(K^Q_{< d} \to A_{< d}) \subseteq \cone(K^Q \to A)
\]
is then a minimal $Q$-free resolution of $R/\soc(R)$.
\end{lemma}

\begin{proof}
By the exact sequence of homology groups $\H_*(\cone(K^Q \to A))$ is isomorphic to $R/\soc(R)$, concentrated in degree zero. The proof of \cite[Theorem~1]{Avramov/Levin:1978} shows that the map $K^Q_i\otimes_Qk \to A_i\otimes_Qk$ is an isomorphism for $i=d$ and zero for $i<d$. The former fact implies that the inclusion $A' \subseteq \cone(K^Q \to A)$ is a quasi-isomorphism, and the later implies that $A'$ is minimal as a complex. Altogether this shows that $A'$ is the minimal resolution of $R/\soc(R)$. 
\end{proof}

\begin{lemma}\label{t_golod_gorenstein_koszul}
Let $R$ be an almost Golod Gorenstein ring of codimension $d$ having a minimal Cohen presentation $Q\to R$. Assume that the minimal $Q$-free resolution $A$ of $R$ is equipped with a cyclic $\ai$-structure. Then $(\bar{m}_n(\bar{A}^{\otimes n}))_i\subseteq \fm_Q \bar{A}_i$ for all $i<d$, and $(\bar{m}_n(\bar{A}^{\otimes n}))_d=0$ for $n \geqslant 3$. In particular $R\lotimes_Q k$ is formal and Koszul.
\end{lemma}

\begin{proof}
We first address what happens in degree $d$, and for this we use the fact that $A$ is a cyclic $\ai$-algebra. If $n\geqslant 3$ and $m_n(a_1,\ldots,a_n)$ has degree $d$, then
 \[
 \langle m_n(a_1,\ldots,a_n), 1\rangle =(-1)^n \langle m_n(1,a_1,\ldots,a_{n-1}), a_n\rangle =0
 \]
 since $A$ is strictly unital. But $\langle -, 1\rangle$ is the projection onto the (rank $1$) degree $d$ part of $A$, so this implies $ m_n(a_1,\ldots,a_n)=0$.

 For the rest of the argument we need to reduce to the case that $R$ has dimension zero. We may find a sequence $\bm{x}$ that is part of a minimal generating set of $\fm_Q$, and that maps to a maximal regular sequence in $\fm_R$. 
 All of the hypotheses, and the remaining assertions to prove, are unchanged if we replace $Q$, $R$ and $A$ with $Q/(\bm{x})$, $R/(\bm{x}R)$ and $A\otimes (Q/(\bm{x}))$ respectively, using \cref{p_change_of_rings} for the Koszul conclusion. Therefore we may assume that $R$ has dimension zero.

 We now use the notation and results of \cref{l_socle_res}. Since $A'=\cone(K^Q_{< d} \to A_{< d})$ is the minimal $Q$-free resolution of $R/\soc(R)$ there is a splitting
 \[
 \begin{tikzcd}
 A'\ar[r,hook,"\simeq"] \& \cone(K^Q \to A) \ar[l, bend right, dashed],
 \end{tikzcd}
 \]
 and we define $\phi_1$ to be the composition $A\to \cone(K^Q \to A) \to A'$. By construction $(\phi_1)_i\colon A_i\to A'_i$ is a split injection for $i<d$. Since $R/\soc(R)$ is Golod we may endow $A'$ with a strictly unital $\ai$-structure $\{m_n'\}$ satisfying $\bar{m}'_n(\bar{A}'^{\otimes n})\subseteq \fm_Q \bar{A}'$ for all $n\geqslant 1$ by \cite[Theorem~6.13]{Burke:2015}. Having done this, the chain map $\phi_1$ can be extended to a strictly unital map of $\ai$-algebras using \cref{p_Ainfinity_lifting}. 
 We apply \cref{l_minimal_A_infty_map} to the morphism $A\otimes_Q k\to A'\otimes_Qk$ to deduce that the $\ai$-structure of $A$ satisfies $(\bar{m}_n(\bar{A}^{\otimes n}))_i\subseteq \fm_Q \bar{A}_i$ for all $n\geqslant 1$ and all $i<d$.
 
Since the induced higher $\ai$-structure on $A\otimes_Qk$ vanishes, $R\lotimes_Q k\simeq A\otimes_Qk$ is formal by \cref{p_formality}. It also follows that $A\otimes_Qk$ is a short Gorenstein ring, and in particular it is Koszul by \cref{c:short_gorenstein}.
\end{proof}

We are finally able to prove that almost Golod Gorenstein rings satisfying certain technical assumptions are strictly Cohen Koszul, as promised in the proof of \cref{t_almost_golod_gor_characterisation}, and substantially generalizing the class of Gorenstein local rings of codimension three covered by \cref{e:koszul-ai:buchsbaum-eisenbud}.

\begin{theorem}\label{t_golod_gor_strict}
If $R$ is an almost Golod Gorenstein ring of odd codimension $d$, containing a field of characteristic zero, with a minimal Cohen presentation $\phi\colon Q\to R$, then $\phi$ is strictly Koszul. More precisely, if the minimal resolution $A$ of $R$ 
admits a cyclic $\ai$-structure, then (regardless of $d$ or the characteristic) $A\cong \tmod{V}/{W}$ where
\begin{equation} \label{eq_golod_gor_koszul_presentation}
V={\textstyle \bigoplus_{i=1}^{d-1}A_i}\quad \text{and}\quad W=\ker\big( \langle -,-\rangle \colon V\otimes V\to \susp^dQ\big)\,,
\end{equation}
and $(A,V,W)$ is a strictly Koszul presentation for $\phi$.
\end{theorem}

\begin{proof}
Since $d$ is odd and $R$ is Gorenstein of characteristic zero, we may endow $A$ with a cyclic $\ai$-structure by \cref{p_cyclic_ai}. 

The pairing $\langle -,-\rangle $ defined in \cref{cyclic_ai_algebra} is nondegenerate, and this implies that $W$ is a summand of $V\otimes V$. We know that $A\otimes k\cong \talg{V\otimes k}/(W\otimes k)$ since $A\otimes k$ is short Gorenstein. It follows from Nakayama's lemma that $A\cong \tmod{V}/{W}$ as graded $Q$-modules. The assertion $(\bar{m}_n(\bar{A}^{\otimes n}))_d=0$ from \cref{t_golod_gorenstein_koszul} implies that $m_n(V^{\otimes n})\subseteq V$ for all $n$, and therefore the presentation $(A,V,W)$ is strict.
\end{proof} 

Taking an almost Golod Gorenstein ring $R$, with $Q$ and $A$ as in the theorem, we can describe the Priddy coalgebra explicitly:
\[
 \priddy[n]{V,W} = \Big\{{\textstyle{\sum} v_1 \otimes \cdots \otimes v_n} ~\Big|~\textstyle{\sum}  v_1 \otimes \cdots \otimes\langle v_i, v_{i+1} \rangle v_{i+2} \otimes \cdots \otimes v_{n}= 0,\ 1 \leqslant i < n\Big\}.
 \]
This is also the dual of a noncommutative hypersurface, as in \cref{e:koszul-ai:buchsbaum-eisenbud}.

If we let $M$ be a bounded complex of finitely generated $R$-modules, then there is a finite free $Q$-resolution $G \to M$, and $G$ can be given a strictly unital $\ai$-module structure over $A$ by \cref{c_ai_module}. All of this data can be constructed with finitely many computations, and it can be assembled into a resolution
\begin{equation*}
R \totimes \priddy{V,W} \totimes G \xrightarrow{\ \simeq \ } M
\end{equation*}
with an explicit differential given in \cref{priddy-resolution} in terms of the $\ai$-structures of $A$ and $G$. When $M=k$ is the residue field and $G=K^Q$ is the Koszul complex of $Q$, the Priddy resolution is minimal by \cref{t_inert}.

The last example provides a class of almost Golod Gorenstein rings where we verify they are Cohen Koszul without any assumptions on the characteristic of $k$ or the parity of the codimension. 


\begin{example}\label{ex_amost_lin_a_infinity}
Let $Q$ be a standard graded polynomial ring over $k$ and let $R=Q/I$, where $I$ is an ideal generated by forms of degree $e\geqslant 3$ admitting an almost linear free resolution, as in \cref{r_list_of_examples_almost_golod_gor}. We further assume that $R$ is Gorenstein of codimension $d$.

Writing $T_{i,j}=\Tor[i]{Q}{R}{k}_j=\H_i(K^R)_j$, we have, by assumption, $T_{i,j} = 0$ when $0<i < d$ and $j-i \neq e-1$. The only other nontrivial components are the unit $T_{0,0}=k$ and the socle $T_{d,2e-2+d}=k$. The product on $T$ must preserve both gradings, and it follows that $T$ is a short Gorenstein ring.

Let $A\xra{\simeq} R$ be the minimal graded free resolution of $R$ over $Q$. In this setting there exists an $\ai$-algebra structure $\{m_n\}$ on $A$ that is homogeneous of degree zero with respect to the internal grading. In particular the induced operations $m_n\otimes_Q k$ on $T$ restrict to maps
\begin{equation*}
T_{i_1,j_1} \otimes_k \cdots \otimes_k T_{i_n,j_n} \longrightarrow T_{i_1+\cdots+i_n+n-2,j_1+\cdots+j_n}\,.
\end{equation*}
If some $i_\ell=d$ the map above is zero for (homological) degree reasons, and so we may assume each  $i_\ell$ is strictly smaller than $d$. Therefore each $j_\ell-i_\ell = e-1$ and
\begin{equation*}
(j_1+\cdots+j_n)-(i_1+\cdots+i_n+n-2) = n(e-1) - n +2\,.
\end{equation*}
For $m_n\otimes_Q k$ to be nonvanishing on this component, this expression must equal $e-1$ or $2e-2$. In the former case we find $(n-1)(e-1) = n-2$, which cannot hold; in the latter case $(n-2)(e-1)=n-2$, and this cannot hold when $n$ and $e$ are not $2$. Hence $m_n\otimes_Q k=0$ for $n\neq 2$.

Thus we may apply \cref{koszul-ai-presentation} with $V=\bigoplus_{i=1}^{d-1} A_i$ and $W$ as in \cref{eq_golod_gor_koszul_presentation} to conclude that $Q \to R$ is Koszul. Finally, it follows from  \cref{t_almost_golod_gor_characterisation} that $R$ is not only Cohen Koszul, but also almost Golod Gorenstein.
\end{example}

\bibliographystyle{amsalpha}
\bibliography{refs}

\newcommand{\etalchar}[1]{$^{#1}$}
\providecommand{\bysame}{\leavevmode\hbox to3em{\hrulefill}\thinspace}
\providecommand{\MR}{\relax\ifhmode\unskip\space\fi MR }
\providecommand{\MRhref}[2]{%
  \href{http://www.ams.org/mathscinet-getitem?mr=#1}{#2}
}
\providecommand{\href}[2]{#2}
\begin{thebibliography}{CDG{\etalchar{+}}20}

\bibitem[AB00]{Avramov/Buchweitz:2000a}
Luchezar~L. Avramov and Ragnar-Olaf Buchweitz, \emph{Homological algebra modulo
  a regular sequence with special attention to codimension two}, J. Algebra
  \textbf{230} (2000), no.~1, 24--67. \MR{1774757}

\bibitem[ABIM10]{Avramov/Buchweitz/Iyengar/Miller:2010}
Luchezar~L. Avramov, Ragnar-Olaf Buchweitz, Srikanth~B. Iyengar, and Claudia
  Miller, \emph{Homology of perfect complexes}, Adv. Math. \textbf{223} (2010),
  no.~5, 1731--1781. \MR{2592508}

\bibitem[AG71]{Avramov/Golod:1971}
Luchezar~L. Avramov and Evgeni{\u\i}~S. Golod, \emph{The homology of algebra of
  the {K}oszul complex of a local {G}orenstein ring}, Mat. Zametki \textbf{9}
  (1971), 53--58 (Russian). \MR{279157}

\bibitem[AH86]{Avramov/Halperin:1986}
Luchezar Avramov and Stephen Halperin, \emph{Through the looking glass: a
  dictionary between rational homotopy theory and local algebra}, Algebra,
  algebraic topology and their interactions, {P}roc. of the {C}onf.
  ({S}tockholm, 1983) (Jan-Erik Roos, ed.), Lecture Notes in Math., vol. 1183,
  Springer, Berlin, 1986, pp.~1--27. \MR{846435}

\bibitem[AH{\c{S}}13]{Avramov/Henriques/Sega:2013}
Luchezar~L. Avramov, In{\^e}s Bonacho Dos~Anjos Henriques, and Liana~M.
  {\c{S}}ega, \emph{Quasi-complete intersection homomorphisms}, Pure Appl.
  Math. Q. \textbf{9} (2013), no.~4, 579--612. \MR{3263969}

\bibitem[AI03]{Avramov/Iyengar:2003}
Luchezar~L. Avramov and Srikanth Iyengar, \emph{Andr{\'e}-{Q}uillen homology of
  algebra retracts}, Ann. Sci. {\'E}c. Norm. Sup{\'e}r. (4) \textbf{36} (2003),
  no.~3, 431--462. \MR{1977825}

\bibitem[AKM88]{Avramov/Kustin/Miller:1988}
Luchezar~L. Avramov, Andrew~R. Kustin, and Matthew Miller, \emph{Poincar\'{e}
  series of modules over local rings of small embedding codepth or small
  linking number}, J. Algebra \textbf{118} (1988), no.~1, 162--204. \MR{961334}

\bibitem[Ale17]{Alesandroni:2017}
Guillermo Alesandroni, \emph{Minimal resolutions of dominant and semidominant
  ideals}, J. Pure Appl. Algebra \textbf{221} (2017), no.~4, 780--798.
  \MR{3574207}

\bibitem[Ame20]{Amelotte:2020}
Steven Amelotte, \emph{Connected sums of sphere products and minimally
  non-{G}olod complexes}, arXiv e-prints, 2020,
  \url{https://arxiv.org/abs/2006.16320v1}, pp.~1--9.

\bibitem[And82]{Andre:1982}
Michel Andr\'{e}, \emph{Le caract\`ere additif des d\'{e}viations des anneaux
  locaux}, Comment. Math. Helv. \textbf{57} (1982), no.~4, 648--675.
  \MR{694609}

\bibitem[AP93]{Allday/Puppe:1993}
C.~Allday and V.~Puppe, \emph{Cohomological methods in transformation groups},
  Cambridge Studies in Advanced Mathematics, vol.~32, Cambridge University
  Press, Cambridge, 1993. \MR{1236839}

\bibitem[Avr78]{Avramov:1978}
Luchezar~L. Avramov, \emph{Small homomorphisms of local rings}, J. Algebra
  \textbf{50} (1978), no.~2, 400--453. \MR{485906}

\bibitem[Avr81]{Avramov:1981}
\bysame, \emph{Obstructions to the existence of multiplicative structures on
  minimal free resolutions}, Amer. J. Math. \textbf{103} (1981), no.~1, 1--31.
  \MR{601460}

\bibitem[Avr86]{Avramov:1986}
\bysame, \emph{Golod homomorphisms}, Algebra, algebraic topology and their
  interactions, {P}roc. of the {C}onf. ({S}tockholm, 1983) (Jan-Erik Roos,
  ed.), Lecture Notes in Math., vol. 1183, Springer, Berlin, 1986, pp.~59--78.
  \MR{846439}

\bibitem[Avr98]{Avramov:1998}
\bysame, \emph{Infinite free resolutions}, Six lectures on commutative algebra,
  Papers from the Summer School ({B}allaterra, 1996) (Juan El{\'\i}as,
  Jos{\'e}~M. Giral, Rosa~M. Mir{\'o}-Roig, and Santiago Zarzuela, eds.), Mod.
  Birkh{\"a}user Class., Birkh{\"a}user/Springer,Basel, 1998, pp.~1--118.
  \MR{1648664}

\bibitem[Avr99]{Avramov:1999}
\bysame, \emph{Locally complete intersection homomorphisms and a conjecture of
  {Q}uillen on the vanishing of cotangent homology}, Ann. of Math. (2)
  \textbf{150} (1999), no.~2, 455--487. \MR{1726700}

\bibitem[BDS23]{Brown/Dao/Sridhar:2023}
Michael~K. Brown, Hailong Dao, and Prashanth Sridhar, \emph{Periodicity of
  ideals of minors in free resolutions}, arXiv e-prints, 2023,
  \url{https://arxiv.org/abs/2306.00903v3}, pp.~1--29.

\bibitem[BE77]{Buchsbaum/Eisenbud:1977}
David~A. Buchsbaum and David Eisenbud, \emph{Algebra structures for finite free
  resolutions, and some structure theorems for ideals of codimension {$3$}},
  Amer. J. Math. \textbf{99} (1977), no.~3, 447--485. \MR{453723}

\bibitem[BEH87]{Buchweitz/Eisenbud/Herzog:1985}
R.-O. Buchweitz, D.~Eisenbud, and J.~Herzog, \emph{Cohen-macaulay modules on
  quadrics}, Singularities, representation of algebras, and vector bundles
  (Lambrecht, 1985), Lecture Notes in Math., vol. 1273, Springer, Berlin, 1987,
  pp.~58--116. \MR{915169}

\bibitem[Ber05]{Berglund:2005}
Alexander Berglund, \emph{Poincar{\'e} series and homotopy lie algebras of
  monomial rings}, Research Reports in Mathematics, Department of Mathematics,
  Stockholm University (2005), no.~6, 1--40,
  \url{https://www2.math.su.se/reports/2005/6/2005-6.pdf}.

\bibitem[Ber14]{Berglund:2014}
\bysame, \emph{Koszul spaces}, Trans. Amer. Math. Soc. \textbf{366} (2014),
  no.~9, 4551--4569. \MR{3217692}

\bibitem[BF85]{Backelin/Froeberg:1985b}
J{\"o}rgen Backelin and Ralf Fr{\"o}berg, \emph{Poincar{\'e} series of short
  {A}rtinian rings}, J. Algebra \textbf{96} (1985), no.~2, 495--498.
  \MR{810542}

\bibitem[BGP24]{Briggs/Grifo/Pollitz:2024}
Benjamin Briggs, Elo{\'\i}sa Grifo, and Josh Pollitz, \emph{Bounds on
  cohomological support varieties}, Trans. Amer. Math. Soc. Ser. B \textbf{11}
  (2024), 703--726. \MR{4719765}

\bibitem[BGS96]{Beilinson/Ginzburg/Soergel:1996}
Alexander Beilinson, Victor Ginzburg, and Wolfgang Soergel, \emph{Koszul
  duality patterns in representation theory}, J. Amer. Math. Soc. \textbf{9}
  (1996), no.~2, 473--527. \MR{1322847}

\bibitem[BH98]{Bruns/Herzog:1998}
Winfried Bruns and J{\"u}rgen Herzog, \emph{{C}ohen-{M}acaulay rings}, revised
  ed., Cambridge Studies in Advanced Mathematics, vol.~39, Cambridge University
  Press, Cambridge, 1998. \MR{1251956}

\bibitem[BJ07]{Berglund/Joellenbeck:2007}
Alexander Berglund and Michael J{\"o}llenbeck, \emph{On the {G}olod property of
  {S}tanley-{R}eisner rings}, J. Algebra \textbf{315} (2007), no.~1, 249--273.
  \MR{2344344}

\bibitem[Bou81]{Bourbaki:1981}
Nicolas Bourbaki, \emph{{\'E}l{\'e}ments de math{\'e}matique. {A}lg{\`e}bre.
  {C}hapitres 4 {\`a} 7}, Masson, Paris, 1981 (French), Collection and reprint
  of the 1950, 1952 originals. \MR{643362}

\bibitem[BP15]{Buchstaber/Panov:2015}
Victor~M. Buchstaber and Taras~E. Panov, \emph{Toric topology}, Math. Surveys
  Monogr., vol. 204, American Mathematical Society, Providence, RI, 2015.
  \MR{3363157}

\bibitem[Bur]{Burke:notes}
Jesse Burke, \emph{Koszul duality for represenations of an a-infinty algebra
  defined over a commutative ring},
  https://maths-people.anu.edu.au/~burkej/koszul/old-Koszul.pdf, (accessed May
  22, 2023), pp.~1--73.

\bibitem[Bur68]{Burch:1968}
Lindsay Burch, \emph{On ideals of finite homological dimension in local rings},
  Proc. Cambridge Philos. Soc. \textbf{64} (1968), 941--948. \MR{229634}

\bibitem[Bur15]{Burke:2015}
Jesse Burke, \emph{Higher homotopies and {G}olod rings}, arXiv e-prints, 2015,
  \url{https://arxiv.org/abs/1508.03782v2}, pp.~1--26.

\bibitem[Bur18]{Burke:2018}
\bysame, \emph{Transfer of {A}-infinity structures to projective resolutions},
  arXiv e-prints, 2018, \url{https://arxiv.org/abs/1801.08933v1}, pp.~1--20.

\bibitem[Car83]{Carlsson:1983}
G.~Carlsson, \emph{On the homology of finite free {$({\bf
  Z}/2)^{n}$}-complexes}, Invent. Math. \textbf{74} (1983), no.~1, 139--147.
  \MR{722729}

\bibitem[CDG{\etalchar{+}}20]{Croll/etal:2020}
Amanda Croll, Roger Dellaca, Anjan Gupta, Justin Hoffmeier, Vivek Mukundan,
  Liana~M. \c{S}ega, Gabriel Sosa, Peder Thompson, and Denise Rangel~Tracy,
  \emph{Detecting {K}oszulness and related homological properties from the
  algebra structure of {K}oszul homology}, Nagoya Math. J. \textbf{238} (2020),
  47--85. \MR{4092847}

\bibitem[Che17]{Cheng:2017}
Zhi Cheng, \emph{Trivial extension of {K}oszul algebras}, Front. Math. China
  \textbf{12} (2017), no.~5, 1045--1056. \MR{3698413}

\bibitem[DE22]{Dao/Eisenbud:2022}
Hailong Dao and David Eisenbud, \emph{Linearity of free resolutions of monomial
  ideals}, Res. Math. Sci. \textbf{9} (2022), no.~2, Paper No. 35, 15.
  \MR{4431293}

\bibitem[DJ91]{Davis/Januszkiewicz:1991}
Michael~W. Davis and Tadeusz Januszkiewicz, \emph{Convex polytopes, {C}oxeter
  orbifolds and torus actions}, Duke Math. J. \textbf{62} (1991), no.~2,
  417--451. \MR{1104531}

\bibitem[DS95]{Dwyer/Spalinski:1995}
William~G. Dwyer and Jan Spali{\'n}ski, \emph{Homotopy theories and model
  categories}, Handbook of algebraic topology (Ioan~M. James, ed.),
  North-Holland, Amsterdam, 1995, pp.~73--126. \MR{1361887}

\bibitem[DS07]{Denham/Suciu:2007}
Graham Denham and Alexander~I. Suciu, \emph{Moment-angle complexes, monomial
  ideals and {M}assey products}, Pure Appl. Math. Q. \textbf{3} (2007), no.~1,
  25--60. \MR{2330154}

\bibitem[EHU06]{Eisenbud/Huneke/Ulrich:2006}
David Eisenbud, Craig Huneke, and Bernd Ulrich, \emph{The regularity of {T}or
  and graded {B}etti numbers}, Amer. J. Math. \textbf{128} (2006), no.~3,
  573--605. \MR{2230917}

\bibitem[Eis80]{Eisenbud:1980}
David Eisenbud, \emph{Homological algebra on a complete intersection, with an
  application to group representations}, Trans. Amer. Math. Soc. \textbf{260}
  (1980), no.~1, 35--64. \MR{0570778}

\bibitem[FHT01]{Felix/Halperin/Thomas:2001}
Yves F{\'e}lix, Stephen Halperin, and Jean-Claude Thomas, \emph{Rational
  homotopy theory}, Grad. Texts in Math., vol. 205, Springer-Verlag, New York,
  2001. \MR{1802847}

\bibitem[Fr{\"o}75]{Froeberg:1975}
Ralph Fr{\"o}berg, \emph{Determination of a class of {P}oincar\'{e} series},
  Math. Scand. \textbf{37} (1975), no.~1, 29--39. \MR{404254}

\bibitem[Fr{\"o}82]{Froeberg:1982}
Ralf Fr{\"o}berg, \emph{A study of graded extremal rings and of monomial
  rings}, Math. Scand. \textbf{51} (1982), no.~1, 22--34. \MR{681256}

\bibitem[Fr{\"o}99]{Froeberg:1999}
\bysame, \emph{Koszul algebras}, Advances in commutative ring theory, {P}roc.
  of the 3rd {I}ntl. {C}onf. on {C}ommutative {R}ing {T}heory ({F}ez, 1997)
  (David~E. Dobbs, Marco Fontana, and Salah-Eddine Kabbaj, eds.), Lecture Notes
  in Pure Appl. Math., vol. 205, Marcel Dekker, Inc., New York, 1999,
  pp.~337--350. \MR{1767430}

\bibitem[Gem76]{Gemeda:1976}
Demissu Gemeda, \emph{Multiplicative structure of finite free resolutions of
  ideals generated by monomials in an {$R$}-sequence}, ProQuest LLC, Ann Arbor,
  MI, 1976. \MR{2626146}

\bibitem[GM74]{Gugenheim/May:1974}
Victor K. A.~M. Gugenheim and J.~Peter May, \emph{On the theory and
  applications of differential torsion products}, Mem. Amer. Math. Soc., no.
  142, American Mathematical Society, Providence, R.I., 1974. \MR{0394720}

\bibitem[Gol78]{Golod:1978}
E.~S. Golod, \emph{Homology of some local rings.}, Uspekhi Mat. Nauk (1978),
  no.~no. 5(203),, 177--178. \MR{511890}

\bibitem[GS]{M2}
Daniel~R. Grayson and Michael~E. Stillman, \emph{Macaulay2, a software system
  for research in algebraic geometry}, Available at
  \url{http://www.math.uiuc.edu/Macaulay2/}.

\bibitem[HI05]{Herzog/Iyengar:2005}
J{\"u}rgen Herzog and Srikanth Iyengar, \emph{Koszul modules}, J. Pure Appl.
  Algebra \textbf{201} (2005), no.~1--3, 154--188. \MR{2158753}

\bibitem[HRW98]{Herzog/Reiner/Welker:1998}
J{\"u}rgen Herzog, Vic Reiner, and Volkmar Welker, \emph{The {K}oszul property
  in affine semigroup rings}, Pacific J. Math. \textbf{186} (1998), no.~1,
  39--65. \MR{1665056}

\bibitem[HW08]{He/Wu:2008}
Ji~Wei He and Quan~Shui Wu, \emph{Koszul differential graded algebras and {BGG}
  correspondence}, J. Algebra \textbf{320} (2008), no.~7, 2934--2962.
  \MR{2442004}

\bibitem[Iar84]{Iarrobino:1984}
Anthony Iarrobino, \emph{Compressed algebras: {A}rtin algebras having given
  socle degrees and maximal length}, Trans. Amer. Math. Soc. \textbf{285}
  (1984), no.~1, 337--378. \MR{748843}

\bibitem[IW18]{Iyengar/Walker:2018}
Srikanth~B. Iyengar and Mark~E. Walker, \emph{Examples of finite free complexes
  of small rank and small homology}, Acta Math. \textbf{221} (2018), no.~1,
  143--158. \MR{3877020}

\bibitem[Iye97]{Iyengar:1997}
Srikanth Iyengar, \emph{Free resolutions and change of rings}, J. Algebra
  \textbf{190} (1997), no.~1, 195--213. \MR{1442152}

\bibitem[Iye01]{Iyengar:2001}
\bysame, \emph{Free summands of conormal modules and central elements in
  homotopy {L}ie algebras of local rings}, Proc. Amer. Math. Soc. \textbf{129}
  (2001), no.~6, 1563--1572. \MR{1707520}

\bibitem[Kad82]{Kadeishvili:1982}
Tornike~V. Kadeishvili, \emph{The algebraic structure in the homology of an
  {$A(\infty)$}-algebra}, Soobshch. Akad. Nauk Gruzin. SSR \textbf{108} (1982),
  no.~2, 249--252 (1983). \MR{720689}

\bibitem[Kat17]{Katthan:2017}
Lukas Katth\"{a}n, \emph{A non-{G}olod ring with a trivial product on its
  {K}oszul homology}, J. Algebra \textbf{479} (2017), 244--262. \MR{3627285}

\bibitem[Kel01]{Keller:2001}
Bernhard Keller, \emph{Introduction to {$A$}-infinity algebras and modules},
  Homology Homotopy Appl. \textbf{3} (2001), no.~1, 1--35. \MR{1854636}

\bibitem[Kel02]{Keller:2002}
\bysame, \emph{{$A$}-infinity algebras in representation theory},
  Representations of algebra. {V}ol. {I}, {II}, {P}roc. of the {C}onf.
  ({B}eijing, 2000) (Dieter Happel and Ying~Bo Zhang, eds.), Beijing Norm.
  Univ. Press, Beijing, 2002, pp.~74--86. \MR{2067371}

\bibitem[Kon94]{Kontsevich:1994}
Maxim Kontsevich, \emph{Feynman diagrams and low-dimensional topology}, First
  {E}uropean {C}ongress of {M}athematics, {V}ol. {II} ({P}aris, 1992) (Anthony
  Joseph, Fulbert Mignot, Fran{\c{c}}ois Murat, Bernard Prum, and Rudolf
  Rentschler, eds.), Progr. Math., vol. 120, Birkh\"{a}user, Basel, 1994,
  pp.~97--121. \MR{1341841}

\bibitem[LA78]{Avramov/Levin:1978}
Gerson~L. Levin and Luchezar~L. Avramov, \emph{Factoring out the socle of a
  {G}orenstein ring}, J. Algebra \textbf{55} (1978), no.~1, 74--83. \MR{515760}

\bibitem[Les90]{Lescot:1990}
Jack Lescot, \emph{S{\'e}ries de {P}oincar{\'e} et modules inertes}, J. Algebra
  \textbf{132} (1990), no.~1, 22--49 (French). \MR{1060830}

\bibitem[Lev76]{Levin:1976}
Gerson Levin, \emph{Lectures on golod homomorphisms}, Matematiska
  Institutionen, Stockholms Universitet, no.~15, Univ. of Stockholm, 1976.

\bibitem[LH03]{LefevreHasegawa:2003}
Kenji Lef{\`e}vre-Hasegawa, \emph{Sur les {$A_\infty$}-catégories}, Ph.D.
  thesis, Universit{\'e} Paris 7 - Denis Diderot, 2003, p.~230.

\bibitem[Lim19]{Limonchenko:2019}
Ivan~Yu. Limonchenko, \emph{On higher {M}assey products and rational formality
  for moment-angle manifolds over multiwedges}, Tr. Mat. Inst. Steklova
  \textbf{305} (2019), 174--196. \MR{4017606}

\bibitem[L{\"o}f86]{Lofwall:1986}
Clas L{\"o}fwall, \emph{On the subalgebra generated by the one-dimensional
  elements in the {Y}oneda {E}xt-algebra}, Algebra, algebraic topology and
  their interactions, {P}roc. of the {C}onf. ({S}tockholm, 1983) (Jan-Erik
  Roos, ed.), Lecture Notes in Math., vol. 1183, Springer, Berlin, 1986,
  pp.~291--338. \MR{0846457}

\bibitem[Lut17]{Lutz:manifold}
Frank~H. Lutz, \emph{The manifold page},
  \url{https://www3.math.tu-berlin.de/IfM/Nachrufe/Frank_Lutz/stellar/}, 2017,
  Accessed: 2024-06-21.

\bibitem[LV12]{Loday/Vallette:2012}
Jean-Louis Loday and Bruno Vallette, \emph{Algebraic operads}, Grundlehren
  Math. Wiss., vol. 346, Springer, Heidelberg, 2012. \MR{2954392}

\bibitem[Lyu13]{Lyubashenko:2013}
Volodymyr~V. Lyubashenko, \emph{Bar and cobar constructions for curved algebras
  and coalgebras}, Mat. Stud. \textbf{40} (2013), no.~2, 115--131. \MR{3185246}

\bibitem[Mar06]{Markl:2006}
Martin Markl, \emph{Transferring {$A_\infty$} (strongly homotopy associative)
  structures}, The proceedings of the 25th winter school ``Geometry and
  physics'' (Srn{\'\i}, Czech Republic, 2006) (Martin {\v{C}}adek, ed.), Rend.
  Circ. Mat. Palermo (2) Suppl., no.~79, Palermo: Circolo Matem{\'a}tico di
  Palermo, 2006, pp.~139--151. \MR{2287133}

\bibitem[Mar21]{Martin:2021}
A.~Amadeus Martin, \emph{Curved {BGG} {C}orrespondence}, Ph.D. thesis, The
  University of Nebraska - Lincoln, 2021, p.~71. \MR{4326825}

\bibitem[May69]{May:1969}
J.~Peter May, \emph{Matric {M}assey products}, J. Algebra \textbf{12} (1969),
  533--568. \MR{238929}

\bibitem[MM22]{McCullough/Mere:2022}
Jason McCullough and Zachary Mere, \emph{G-quadratic, lg-quadratic, and koszul
  quotients of exterior algebras}, Communications in Algebra \textbf{50}
  (2022), no.~8, 3284--3300.

\bibitem[MR18]{Miller/Rahmati:2018}
Claudia Miller and Hamidreza Rahmati, \emph{Free resolutions of {A}rtinian
  compressed algebras}, J. Algebra \textbf{497} (2018), 270--301. \MR{3743182}

\bibitem[Mye21]{Myers:2021}
John Myers, \emph{Linear resolutions over {K}oszul complexes and {K}oszul
  homology algebras}, J. Algebra \textbf{572} (2021), 163--194. \MR{4198196}

\bibitem[Nag62]{Nagata:1962}
Masayoshi Nagata, \emph{Local rings}, Interscience Tracts in Pure and Applied
  Mathematics, no.~13, Interscience Publishers a division of John Wiley \&
  Sons, New York-London, 1962. \MR{0155856}

\bibitem[NR05]{Notbhom/Ray:2005}
Dietrich Notbohm and Nigel Ray, \emph{On {D}avis-{J}anuszkiewicz homotopy
  types. {I}. {F}ormality and rationalisation}, Algebr. Geom. Topol. \textbf{5}
  (2005), 31--51. \MR{2135544}

\bibitem[Pos11]{Posiselski:2011}
Leonid Positselski, \emph{Two kinds of derived categories, {K}oszul duality,
  and comodule-contramodule correspondence}, Mem. Amer. Math. Soc., no. 996,
  Amer. Math. Soc., Providence, RI., 2011. \MR{2830562}

\bibitem[Pri70]{Priddy:1970}
Stewart~B. Priddy, \emph{Koszul resolutions}, Trans. Amer. Math. Soc.
  \textbf{152} (1970), 39--60. \MR{265437}

\bibitem[R{\c{S}}14]{Rossi/Sega:2014}
Maria~Evelina Rossi and Liana~M. {\c{S}}ega, \emph{Poincar\'{e} series of
  modules over compressed {G}orenstein local rings}, Adv. Math. \textbf{259}
  (2014), 421--447. \MR{3197663}

\bibitem[Sch80]{Schenzel:1980}
Peter Schenzel, \emph{\"{U}ber die freien {A}ufl\"{o}sungen extremaler
  {C}ohen-{M}acaulay-{R}inge}, J. Algebra \textbf{64} (1980), no.~1, 93--101
  (German). \MR{575785}

\bibitem[Sha69]{Shamash:1969}
Jack Shamash, \emph{The {P}oincar\'{e} series of a local ring}, J. Algebra
  \textbf{12} (1969), 453--470. \MR{241411}

\bibitem[Sta63a]{Stasheff:1963a}
James~Dillon Stasheff, \emph{Homotopy associativity of {$H$}-spaces. {I}},
  Trans. Amer. Math. Soc. \textbf{108} (1963), 275--292. \MR{0158400}

\bibitem[Sta63b]{Stasheff:1963b}
\bysame, \emph{Homotopy associativity of {$H$}-spaces. {II}}, Trans. Amer.
  Math. Soc. \textbf{108} (1963), 293--312. \MR{0158400}

\bibitem[Sta83]{Stasheff:1983}
James Stasheff, \emph{Rational {P}oincar\'{e} duality spaces}, Illinois J.
  Math. \textbf{27} (1983), no.~1, 104--109. \MR{684544}

\bibitem[Tay66]{Taylor:1966}
Diana~Kahn Taylor, \emph{Ideals generated by monomials in an {$R$}-sequence},
  Ph.D. thesis, The University of Chicago, 1966, p.~(no paging). \MR{2611561}

\bibitem[Wal17]{Walker:2017}
Mark~E. Walker, \emph{Total {B}etti numbers of modules of finite projective
  dimension}, Ann. of Math. (2) \textbf{186} (2017), no.~2, 641--646.
  \MR{3702675}

\bibitem[Wie69]{Wiebe:1969}
Hartmut Wiebe, \emph{\"{U}ber homologische {I}nvarianten lokaler {R}inge},
  Math. Ann. \textbf{179} (1969), 257--274 (German). \MR{255531}

\end{thebibliography}

\end{document}